\documentclass[10pt]{amsart} 
\usepackage{verbatim, pstricks, graphicx, amssymb}
\usepackage[all]{xy}

\newtheorem{thm}{Theorem}[section]
\newtheorem{cor}[thm]{Corollary}
\newtheorem{prop}[thm]{Proposition}
\newtheorem{lem}[thm]{Lemma}

\newtheorem{conven}[thm]{Convention}
\theoremstyle{definition}
\newtheorem{defn}[thm]{Definition}

\newtheorem{rmk}[thm]{Remark}

\newtheorem*{ack}{Acknowledgments}
\newtheorem*{conv}{Convention}

\numberwithin{equation}{section}

\newcommand{\T}{\mathbb T}
\newcommand{\U}{\mathcal U}
\newcommand{\ot}{\otimes}

\newcommand{\gl}{\mathfrak{gl}}
\newcommand{\h}{\mathfrak{h}}
\newcommand{\hloc}{\mathfrak{h}^{(\Sigma_\bullet,\U)}}

\newcommand{\sgn}{{\text{sgn}}}

\newcommand{\R}{\mathbb R}
\newcommand{\C}{\mathbb C}
\newcommand{\Z}{\mathbb {Z}}
\newcommand{\N}{\mathbb {N}}

\newcommand{\End}{\operatorname{End}}
\newcommand{\Sets}{\mathcal Sets}

\newcommand{\NN}{\mathcal N}
\newcommand{\Om}{\Omega}

\newcommand{\pri}[1]{{^{(}}{}'_{#1}{^{)}}}
\newcommand{\Si}{\Sigma}

\newcommand{\ou}[2]{{\scriptsize {\begin{matrix}#1\\ #2\end{matrix}}}}
\newcommand{\ehor}{{\overleftrightarrow{e}}}
\newcommand{\ever}{{\updownarrow e}}
\newcommand{\Oma}{{\mathfrak A}}
\newcommand{\OmA}{{\mathfrak A}}
\newcommand{\Ch}{It}
\newcommand{\ie}{{\emph{i.e.}}}
\newcommand{\cf}{{\emph{c.f.}}}
\newcommand{\inc}[1]{W_#1}
\newcommand{\vf}{{\mathbf {t}}}
\newcommand{\Vf}{{\mathbf {t}}}
\newcommand{\wf}{{\mathbf {u}}}

\title{Equivariant Holonomy for Bundles and Abelian Gerbes}

\author[T.~Tradler]{Thomas~Tradler}
  \address{Thomas Tradler,
  Department of Mathematics, College of Technology, City University of New York, 300 Jay Street, Brooklyn, NY 11201, USA, and
  Max-Planck-Institut f\"ur Mathematik, Vivatsgasse 7, 53111 Bonn, Germany}
  \email{ttradler@citytech.cuny.edu}

\author[S.~Wilson]{Scott O. Wilson}
  \address{Scott O. Wilson, Department of Mathematics, Queens College, City University of New York, 65-30 Kissena Blvd., Flushing, NY 11367}
  \email{scott.wilson@qc.cuny.edu}

\author[M.~Zeinalian]{Mahmoud~Zeinalian}
  \address{Mahmoud Zeinalian, Department of Mathematics, C.W. Post Campus of Long Island University, 720 Northern Boulevard, Brookville, NY 11548, USA} 
  \email{mzeinalian@liu.edu}


\begin{document}

\begin{abstract} This paper generalizes Bismut's equivariant Chern character to the setting of abelian gerbes. In particular, associated to an abelian gerbe with connection, an equivariantly closed differential form  is constructed on the space of maps of a torus into the manifold. These constructions are made explicit using a new local version of the higher Hochschild complex, resulting in differential forms given by iterated integrals.  Connections to two dimensional topological field theories are indicated. Similarly, this local higher Hochschild complex is used to calculate the 2-holonomy of an abelian gerbe along any closed oriented surface, as well as the derivative of 2-holonomy, which in the case of a torus fits into a sequence of higher 
holonomies and their differentials.
\end{abstract}

\maketitle

\allowdisplaybreaks

\setcounter{tocdepth}{2}
\tableofcontents

\section{Introduction}

In \cite{B}, Bismut introduces the \emph{equivariant Chern character} of vector bundle with connection over a manifold, $E\to M$. This is done via analytic means, which results in a periodically closed form on the free loop space of the base of the bundle. The idea for constructing such a class was to provide a twisted counterpart for the view, suggested by Witten and carried out by Atiyah \cite{A}, that the index of the Dirac operator, given by a path integral over the loop space, can be calculated by stationary phase approximation of an invariant symplectic form on the loop space of the manifold. Bismut's equivariant Chern character is the contribution to the integrand of this path integral over the loop space, necessary to take into account the effects of the connection on the axillary bundle used for twisting the Dirac operator. Thus, localization techniques, that gave rise to $\int_M \hat{A}(TM)$ in the untwisted case, would now yield $\int_M \hat{A}(TM)\wedge Ch(E)$, as expected by the classical index theorem of the twisted Dirac operator.

More recently in \cite{H}, the Bismut Chern character has appeared in the passage from a $1|1$ supersymmetric (SUSY) field theory on $M$, determined by a connection on a vector bundle on $M$,  to a $0|1$ supersymmetric field theory on $LM$. This work relies on constructing a SUSY field theory out of a vector bundle and a superconnection as described in \cite{D}. Naively, a bundle and a connection can be thought of as a $1$-dimensional field theory in which to zero dimensional objects, points, one assigns a vector space, the fibre of the bundle above the point, and to one dimensional objects, paths in $M$, one assigns a map between the vector spaces assigned to the end points. In \cite{D}, this scheme is generalized to $0|1$ dimensional objects, super points in $M$, and $1|1$ dimensional objects, or super paths in $M$. Furthermore, by thinking of the space of maps $\mathbb R^{1|1} \to M$ as the space of maps $\mathbb{R}^{0|1} \to \{  \mathbb R^{1|0} \to M\}$, one can get from a $1|1$ theory on $M$ to a $0|1$ theory on the space of paths in $M$ and similarly $LM$, as described in \cite{H}.

Further work is on the way, as part of the general Stolz and Teichner program (\cite{ST}) of understanding the concordance classes of supersymmetric field theories, to construct $2|1$ supersymmetric field theories associated to gerbes, in much the same way $1|1$ theories were constructed out of a bundle endowed with a superconnection, see \cite{D}. One can then interpret such $2|1$ theory on $M$ as a $0|1$ theory on $M^\T$, the space of maps from the torus $\T$ to $M$, using the adjunction $\{\mathbb R^{2|1} \to M\} \simeq \{\mathbb{R}^{0|1} \to \{  \mathbb R^{2|0} \to M\}\}$. We learned this picture from P. Teichner and S. Stolz. 
 
Inspired by such connections to supersymmetric and non-supersymmetric field theories, we address the naturally arising question of the extension of the holonomy of an abelian gerbe to a torus equivariant class on the mapping space of the torus $\T$ into $M$.

In \cite{GJP} the authors reinterpret Bismut's construction using Hochschild complex of matrix valued forms, given an imbedding of the bundle and connection into a trivial bundle. This work shed light on the Bismut class and gave a specific Hochschild cocycle manufactured from the connection and its curvature.  This class produced was later shown in \cite{Z} to be independent of the imbedding and connection.

In this paper, as a warm up, we review and provide a conceptual and more fundamental, pre-trace, interpretation of this construction in terms of holonomy, its covariant derivative, and other naturally constructed higher degree forms and their covariant derivatives. More explicitly, we think of holonomy, $hol\in \Omega^0(LM,\mathcal E)$, as a section of the pullback of the endomorphism bundle on the free loop space via the map that sends a loop to it's base point,
\[
\xymatrix{
 \mathcal{E} \ar[r] \ar[d] & End(E) \ar[d] \\
 LM \ar[r]^-{ev_0} & M }
\]
This bundle on the loop space has an induced connection with respect to which we can take the covariant derivate of the holonomy as a section. The result is a $1$-form on the free loop space with values in this pullback bundle. We then show that there is a naturally defined $2$-form with values in this bundle whose contraction with the natural velocity vector field on the loop space is the covariant derivative of the holonomy. One then asks what the covariant derivative of this naturally defined $2$-form is, and the answer turns out to be a $3$-form which itself is given by the contraction of a naturally defined $4$-from on the loop space. This process continues ad infinitum to produced an infinite sequence of naturally defined even degree forms on the free loop space, starting with the holonomy  $0$-form,
\setcounter{section}{2}  \setcounter{thm}{17}
\begin{thm}
For all $k \geq 0$ we have forms $hol_{2k}\in \Omega^{2k}(LM,\mathcal{E})$, where $hol_0=hol$ is the holonomy, such that
\[
\nabla^* hol_{2k} = - \iota_{\vf} hol_{2k+2} \in \Omega^{2k+1}(LM , \mathcal{E}).
\]
where $\iota_{\vf}$ is contraction by the natural vector field on $LM$ given by the circle action.
\end{thm}
The result can be repackaged as something whose covariant derivative is basically the same as it's contraction with respect to the canonical vector field on $LM$. One notes that trace of the covariant derivative of a form with values in the pulled back bundle is the same as the exterior derivative of the trace of the form. This is due to the simple fact that the induced connection on the endomorphism bundle is obtained from the connection on the bundle using a commutator and that trace is zero on commutator of matrices.  Therefore, by applying trace to the above construction, one obtains an equivariantly closed form which is in fact equal to the Bismut class. 
\setcounter{section}{2}  \setcounter{thm}{18}
\begin{cor} [\cite{B}, \cite{GJP}]
For $Ch^{(u)}(E; \nabla) := \sum_{k\geq 0} u^{-k} hol_{2k}$, where $u$ is a formal variable of degree $2$, we have, that
\[ (\nabla^* + u \cdot\iota_{\vf}) \left( Ch^{(u)}(E; \nabla)  \right) = 0. \]
\end{cor}

We emphasize that this last result is derived in section \ref{SEC:local-vector-bundle} using a local version of the Hochschild complex, which is suitable for the local data of a bundle with connection, see theorem \ref{THM:local-nabla-hol_2k} and corollary \ref{COR:local-nabla+u(Ch)=0}. It is this local version that we generalize to the case of abelian gerbes with connection.

The second part of this paper (sections \ref{SEC:Chern-gerbes} and \ref{SEC:surface-hol}) comprise an analogous discussion for the $2$-holonomy of an abelian gerbe. Namely, starting with an abelian gerbe on $M$, and  given a fixed closed oriented surface $\Sigma$, one can consider $2$-holonomy as a real valued function on the mapping space $M^\Sigma$. We again use a local version of (higher) Hochschild complexes, which are suitable for the local data of an abelian gerbe with connection.
Note that the situation is in some sense simpler than in the case of the bundle, since $2$-holonomy is now complex valued. This is due to the fact that we are considering abelian gerbes, which are higher analogues of $1$-dimensional vector bundles. 

Section \ref{SEC:Chern-gerbes} considers the case of the torus $\Sigma= \T$, where we regard $2$-holonomy as a real valued zero form on the mapping space, $hol\in\Omega^0(M^\T)$, so that one can ask what is its exterior derivative. We show that the result is a $1$-from on $M^\mathbb T$ which is the contraction of a naturally defined 2-form on $M^\T$, with respect to a killing vector field on the torus. There are many such 2-forms: one for each choice of the killing vector field on the torus. Similarly to the bundle case above, we give a hierarchical construction which extends $2$-holonomy to an equivariantly closed element on the torus mapping space. More precisely, we obtain the following theorem and corollary.
\setcounter{section}{4}  \setcounter{thm}{16}
\begin{thm}
 For all $k, \ell \geq 0$, we have forms $hol_{2k,2\ell}\in \Omega^{2k+2\ell}(M^\T)$, where $hol_{0,0}=hol$ is $2$-holonomy,  such that
\begin{equation*}
 d(hol_{2k,2\ell})=-\iota_{\vf}(hol_{2k+2,2\ell})=-\iota_{\wf}(hol_{2k,2\ell+2}),
\end{equation*}
where $\iota_{\vf}$ and $\iota_{\wf}$ are contraction by the two natural vector fields $M^\T$ given by the two circle actions
from the torus $\T = S^1 \times S^1$.
\end{thm}
\begin{cor} 
For $a+b=1$ we have 
\[
Ch(\mathcal G,a,b) :=\sum_{k\geq 0, \ell\geq 0}a^{k}  \cdot b^{\ell}\cdot hol_{2k,2\ell} 
\quad \in \Omega(M^\T)^{inv(\vf+\wf)}
\]
is a closed element, that is,  
$ (d+ \iota_{\vf}+ \iota_{\wf})(Ch(\mathcal G,a,b))=0$.
\end{cor}

In this description of the equivariant Chern character, we omitted the use of any formal variables. The reason for this is that without formal variables, we can determine an interesting relationship between the equivariantly closed classes  $Ch(E;\nabla):=\sum_{k\geq 0} hol_{2k}$ and $Ch(\mathcal G,a,b)$ for a suitable setting, (see also remark \ref{REM:no-u's}). In fact, in section \ref{subsec:compat} we recall a well-known construction starting from an abelian gerbe on $M$ to induce a line bundle $E$ with connection $\hat A$ on the free loop space $LM$. Using this construction, we show that the equivariant forms can be identified.
\setcounter{section}{4}  \setcounter{thm}{23}
\begin{cor}
For $a, b \in \R$ with $a + b = 1$, let $\phi_{a,b}=
\begin{bmatrix}
a & b \\
-1 & 1
\end{bmatrix}  \in SL(2,\R)$. Then there is an induced map $\Phi_{a,b}^{-1}$ such that
\[
(\Phi^{-1}_{a,b})^*Ch(\mathcal G, a , b) =Ch(\mathcal G, 1 , 0).
\]
Furthermore, the adjoint map $\Gamma:L(LM)\to M^T, \Gamma(\gamma)(t,u)=\gamma(t)(u)$, induces a map $\Gamma^*$ such that the equivariant Chern classes are identified with each other,
\[
\Gamma^*\circ (\Phi^{-1}_{a,b})^* Ch(\mathcal G, a , b) = Ch(E; \hat A).
\]
\end{cor}

In the last section \ref{SEC:surface-hol}, we consider the case of a general surface $\Sigma$. We show how to produce the gerbe holonomy $hol\in \Omega^0(M^\Sigma)$ via local Hochschild methods, and calculate its De Rham differential in proposition \ref{D(hol)} as $d_{DR} \left(hol\right)= i\cdot It(H)\wedge hol$, where $H$ denotes the $3$-curvature of the gerbe. However, unlike the case $\Sigma=\T$, we do not have any natural candidates to complete holonomy to an equivariantly closed form.

We hope that the viewpoint presented in this paper is also suitable for defining 2-holonomy for \emph{non-abelian} gerbes, with the principle being that, in the non-abelian case, one must exponentiate in the local Hochschild complex before applying the iterated integral.
Furthermore, the local nature of our description also seem to be central in understanding the holonomy of not only non-abelian gerbes, but also higher dimensional analogs of gerbes. We hope to come back to this and discuss these issues in a future work. 
\setcounter{section}{1}
\setcounter{thm}{0}

\begin{ack}
We would like to thank Ralph Cohen, Kevin Costello, Thomas Schick, Stefan Stolz, Dennis Sullivan, and Peter Teichner for useful conversations concerning the topics of this paper.
The authors were partially supported by the NSF grant DMS-0757245. The first and second authors were supported in part by grants from The City University of New York PSC-CUNY Research Award Program.
\end{ack}

\begin{conven}\label{CONVENT1}
We now give some sign conventions that are used in this paper.
\begin{enumerate}
\item
We denote by $\Delta^k=\{(t_1,\dots,t_k)\in \R^k | 0\leq t_1\leq\dots\leq t_k\leq 1\}$ the standard $k$-simplex, which we will also simply write as $\Delta^k=\{0\leq t_1\leq\dots\leq t_k\leq 1\}$. The $k$-simplex $\Delta^k\subset \R^k$ obtains an orientation induced from $\R^k$, \emph{i.e.} we take the volume form $dt_1\wedge\dots\wedge dt_k$. For example, for a $(k,\ell)$-shuffle $\sigma$ the induced map $\beta^\sigma:\Delta^{k+\ell}\to \Delta^k\times \Delta^\ell, (t_1\leq\dots\leq t_{k+\ell})\mapsto (t_{\sigma(1)}\leq \dots\leq t_{\sigma(k)},t_{\sigma(k+1)}\leq \dots\leq t_{\sigma(k+\ell)})$ is orientation preserving iff $\sgn(\sigma)=+1$.
\item
If $X$ is a manifold with boundary $\partial X$, we call the induced orientation on $\partial X$ the one for which Stokes' theorem holds without signs in all dimensions. Namely, the orientation of $X$ is given by the outward pointing normal vector of $\partial X$, followed by the orientation of $\partial X$.
In particular, the $i^{\text{th}}$ boundary component $\partial_i\Delta^k=\{0\leq t_1\leq\dots\leq t_i=t_{i+1}\leq\dots\leq t_k\leq 1\}$ of $\Delta^k$ has outward pointing unit vector $\frac{1}{\sqrt{2}} (\frac{\partial}{\partial t_i}-\frac{\partial}{\partial t_{i+1}})$, so that the induced orientation on $\partial_i\Delta^k$ is given by $(-1)^{i-1} \cdot \{ \frac{\partial}{\partial t_1},
\ldots, \frac{\partial}{\partial t_{i-1}}, \frac{\partial}{\partial t_i}+ \frac{\partial}{\partial t_{i+1}}, \frac{\partial}{\partial t_{i+2}}, \ldots , 
\frac{\partial}{\partial t_k} \}$.
Thus, the canonical map $\Delta^{k-1}\to \partial_i\Delta^k,( t_1\leq\dots\leq t_i\leq\dots\leq t_{k-1})\mapsto (t_1\leq\dots\leq t_i\leq t_{i}\leq\dots\leq t_{k-1})$ is orientation preserving iff $i$ is odd.
\item
If $X$ is a compact manifold with boundary, and $Y$ is another manifold, we denote by the integration along the fiber of a form $\omega=f(x,y)\cdot dx_{i_1}\dots dx_{i_k}\wedge dy_{j_1}\dots dy_{j_\ell} \in\Omega(X\times Y)$ the form,
$$ \quad\quad \int_X\omega:= \Big(\int_X f(x,y)dx_{i_1}\dots dx_{i_k}\Big)\cdot dy_{j_1}\dots dy_{j_\ell}\in \Omega(Y). $$
We can calculate the De Rham differential $d_{DR}$ of an integral along the fiber as follows,
\begin{equation*}
\quad\quad \int_X d_{DR}(\omega) = \int_X d^X_{DR}(\omega)+\int_X d^Y_{DR}(\omega)
=\int_{\partial X} \omega+(-1)^{dim(X)}d_{DR}\left(\int_X \omega\right).
\end{equation*}
\end{enumerate}
\end{conven}

\section{Equivariant Chern character for vector bundles} \label{S:vector-bundles}

In \cite{B} Bismut introduces the \emph{equivariant Chern character} of a complex vector bundle with connection over a manifold. This is done via differential-geometric means with the result an equivariantly closed element on the free loop space of the base of the bundle.

In \cite{GJP} the authors reinterpet Bismut's construction by using Hochschild complex of forms on the base, which under mild assumptions is an algebraic model for the free loop space. This viewpoint is clarifying as a specific Hochschild cocycle is produced using the connection and its curvature. To do this, Getzler, Jones, and Petrack imbed the line bundle in a trivial bundle such
that the connection on the subbundle is the restriction of the derivative of a function on the total space of the trivial bundle. This class produced was later shown in \cite{Z} to be independent of the imbedding.

In this section we review and provide another interpretation of the equivariant Chern character. The new viewpoint 
is that the class produced is a closed, equivariant extension of the classical holonomy function on $LM$, induced 
by the connection. Indeed, the curvature allows us to construct a collection of even degree forms, which we call \emph{higher holonomy functions}, solving a certain set of linear equations related to exterior derivative and contraction by the natural vector field on $LM$. These imply that the total sum of these functions is a equivariantly-closed function on $LM$.

\begin{conv} Since this work is intimately related to the Chern character, we will choose to work with complex vector bundles
and complex valued differential forms. Nevertheless, many of the results below, in particular those expressing holonomy and 
it's higher extensions using local Hochschild methods, work equally well in the setting of real valued vector bundles and forms.
\end{conv}

\subsection{Definition of holonomy}
Let $E \to M$ be a sub vector bundle of $M\times \mathbb{C}^n \to M$ and assume $M\times \mathbb{C}^n \to M$ has a connection given by a globally defined matrix valued $1$-form $A \in \Omega^1(M, End(\mathbb{C}^n))$, keeping $E \to M$ invariant. Note that every abstract vector bundle and a connection, up to isomorphism, is realized in this way.  For a path $\gamma: [0,1] \to M$ let
$P_0^1(\gamma)$ be the path in $E$ given by parallel translation.  This function is defined
by solving the the ordinary differential equation $x'(t) =  x(t) A(t)$ in $End(\mathbb{C}^n)$ with initial value $x(0)=Id$, where $A(t) =A(\gamma'(t))$, whose unique solution given by
the matrix
\begin{equation} \label{Pexp}
P_0^1(\gamma) = \sum_{k \geq 0} \int_{\Delta^k} A(t_1) \cdots A(t_k) dt_1 \cdots d t_k
\end{equation}
where $\Delta^k$ is the set of numbers $0 \leq t_1 \leq \cdots \leq t_k \leq 1$ and $A(t)$ is, in the local trivialization, 
the connection matrix of $1$-forms evaluated at $\gamma'(t)$.

This matrix is to be interpreted as an endomorphism from $E_{\gamma(0)}$ to $E_{\gamma(1)}$, which are identified in 
the local trivialization. Similarly we may define $P_s^t(\gamma)$. These functions satisfy $P_s^t = \left(P_t^s \right)^{-1}$ and $P_s^t \circ P_r^s = P_r^t$, which implies we can define parallel translation along any path by covering the path with local trivializations.

For paths that are loops, this parallel transport $P_0^1$ produces a section of a bundle of the free loop space $LM$, which we now describe. Consider the pullback of the induced bundle $End(E) \to M$ with induced connection, via the map $ev_0 :LM \to M$ given by evaluation at time zero, \emph{i.e.} $\mathcal E:=(ev_0)^*(End(E))$.
\[
\xymatrix{
 \mathcal{E} \ar[r] \ar[d] & End(E) \ar[d] \\
 LM \ar[r]^-{ev_0} & M
 }
\]
Now, $P_0^1$ induces a section of $\mathcal{E} \to LM$, which we denote by $hol : LM \to \mathcal{E}$.
Generalizations of these facts will all be proved in the next chapter.

Let us denote the pullback connection on $\mathcal{E} \to LM$ by $\nabla^*$.

\begin{prop} \label{lem:nabla-hol0}
For any vector field $x$ along $\gamma$ we have
\[
\nabla^* hol (\gamma, x) = \int_0^1 P_t^1(\gamma) \circ R(\gamma'(t) , x(t) ) \circ P_0^t(\gamma) dt
\]
where $R$ is the curvature of the connection on $E \to M$.
\end{prop}
For line bundles this formula reproduces the known formula for exterior derivative of the holonomy function on the free loop space, see \cite[p. 234, proposition 6.1.1(2)]{Br}.

One can give a direct calculus proof of this proposition, by considering a path of loops, and calculating the total change in 
holonomy by computing along small squares that partition this homotopy. 
We will see instead this formula arising algebraically, after giving an interpretation of holonomy using Hochschild complexes
and iterated integrals (see proposition \ref{prop:De^a}). 

From this proposition we have the following well known corollaries.
\begin{cor} \label{COR:nabla-hol0}
Denote by $\vf$ the natural vector field on $LM$ coming from the circle action. Then, the function $hol: LM \to \mathcal{E}$ has the following two properties.
\begin{enumerate}
\item If $A$ is a flat connection then $hol$ is a flat section, \ie $\nabla^*(hol)=0$.
\item $\nabla^\ast_\Vf hol =0$.
\end{enumerate}
\end{cor}

\begin{proof} \quad
\begin{enumerate}
\item Since $A$ is flat, we have $R=0$ in proposition \ref{lem:nabla-hol0}.
\item $\nabla^*_\Vf hol (\gamma) =\iota_\Vf \nabla^* hol (\gamma) = \int_0^1 P_t^1(\gamma) \circ R(\gamma'(t) , \gamma'(t) ) \circ P_0^t(\gamma) dt=0.$
\end{enumerate}
\end{proof}

We'll see that more elaborate versions of these corollaries will also appear naturally in our algebraic setup (see corollary \ref{COR:nabla-hol02}).

We continue with our assumption that our vector bundle and connection are imbedded in a trivial bundle with fiber $\C^n$ and possibly non-trivial connection. (This includes the case of a local trivialization of a bundle restricted to a contractible neighborhood.) In this case, the connection of $E$ can be written as a 1-form $A$ on $M$ with values in the associative algebra $\gl =\gl(\C^n)$, and as we explain below, the holonomy function $hol$ above is realized as a \emph{Chen iterated integral}.

\subsection{The Chen iterated integral map and holonomy}
In this subsection, we recall an alternative description of the holonomy, using Chen's iterated integral map, as it was used by Getzler, Jones, Petrack in \cite{GJP}. First, we recall the Hochschild chain complex of an associative algebra (with values in itself), which is an associative algebra under the shuffle product.

\begin{defn}\label{DEF:Hoch+shuffle}
For a differential graded associative algebra $A$, the Hochschild chain complex is given by $CH_\bullet(A)=\bigoplus_{n\geq 0} A\ot(A[1])^{\otimes n} $, where $[1]$ denotes a shift down by $1$. Thus, the grading given by declaring a monomial $a_0 \ot \cdots \ot a_n \in A\ot(A[1])^{\otimes n}$ to be of total degree $|a_0 | + \cdots + |a_n| - n$, where $|a_i |$ is the degree of $a_i$. The differential is given by
\begin{eqnarray*}
D(a_0 \otimes\cdots \ot a_n)&=& -\sum_{i=0}^n  (-1)^{|a_0| + \cdots + |a_{i-1}| + i -1} 
a_0\otimes \cdots \ot d(a_i) \ot \cdots \ot a_n  \\
&&-\sum_{i=0}^{n-1} (-1)^{|a_0| + \dots +  |a_i| + i} a_0\otimes \cdots \ot (a_{i}\cdot a_{i+1}) \ot \cdots \ot a_n \\
&&+ (-1)^{(|a_n|+1)(|a_0| + \dots + |a_{n-1}| + n -1)} (a_n\cdot a_0) \ot a_1\ot \cdots \ot a_{n-1} .
\end{eqnarray*}
The shuffle product on $CH_\bullet(A)$ is defined by
\begin{multline*}
 (a_0\ot a_1\ot \dots\ot a_n)\bullet(a'_0\ot a_{n+1}\ot \dots\ot a_{n+m})\\
  = \sum_{\sigma\in S(n,m)} (-1)^\kappa (a_0\cdot a'_0)\ot a_{\sigma^{-1}(1)}\ot\dots\ot a_{\sigma^{-1}(n+m)},
\end{multline*}
where $S(n,m)$ is the set of all $(n,m)$-shuffles, $S(n,m)=\{\sigma\in S_{n+m}:\sigma(1)<\dots <\sigma(n),$ and $\sigma(n+1)<\dots <\sigma(n+m) \}$. The sign $(-1)^\kappa$ is the Koszul sign, determined according to the (potentially
shifted) degrees.
\end{defn}

The shuffle product is associative and, if $A$ is graded commutative, then $D$ is a derivation of the shuffle product and the shuffle product is graded commutative\footnote{In the non-commutative case, the Hochschild differential need not be a derivation of the shuffle product, see \cite{TTW}.}. The example of interest here is $A=\Omega(M; \gl(\C^n))$ with the De Rham differential $d_{DR}$. 

The space $LM = M^{S^1}$ naturally inherits the structure of an infinite dimensional Fr\'echet manifold \cite{Ha}.
In this context, vector fields, differential forms, contraction by vectors, the exterior derivative, etc., are all defined (and enjoy 
the similiar properties) as in the finite dimensional case. We denote the differential forms on $LM$ with values in $\gl(\C^n)$ by $\Om (LM; \gl)$. 

\begin{defn}\label{DEF:It-Chapter2}
The \emph{Chen iterated integral}, $It:  CH_\bullet(\Omega(M; \gl )) \to \Omega(LM; \gl)$ is defined as follows. For a monomial $a_0 \ot \cdots \ot a_n$  with 
$a_i \in \Omega^{j_i}(M; \gl)$ we have $It(a_0 \ot \cdots \ot a_n) \in \Omega^k(LM; \gl)$, where $k= \sum_{i=0}^n j_i - n$. The value of this $k$-form at a loop $\gamma$ with 
vector fields $x_1 , x_2 , \dots , x_k$ along $\gamma$ is defined to be
\begin{align*}
It(a_0 \ot \cdots \ot a_n)_\gamma &(x_1 , x_2 , \dots , x_k) \\ = & \int_{\Delta^{n}} 
[a_0(0) \wedge \iota_\Vf a_1(t_1) \wedge \cdots \wedge \iota_\Vf a_n(t_n)](x_1 , \dots, x_k) dt_1 \cdots dt_n
\end{align*}
where $\Vf$ is the canonical vector field on $LM$ coming from the circle action, so that
\begin{eqnarray*}
\iota_\Vf a(t)(y_1, \dots , y_m) & =&  a_{\gamma(t)}\left(\gamma'(t) , y_1(\gamma(t)), \dots , y_m(\gamma(t)) \right).
\end{eqnarray*}
In short, we have
\begin{align} \label{chen}
It(a_0 \ot \cdots \ot a_n) =  \int_{\Delta^{n}} a_0(0) \iota_\Vf a_1(t_1) \cdots \iota_\Vf a_n(t_n) dt_1\dots dt_n.
\end{align}
\end{defn}

Conceptually, this integral is understood using the evaluation map and integration along the fiber, as we will explain now. If we denote by $ev:LM\times \Delta^n\to M^{\times (n+1)}$ the evaluation at the given times,
$$ ev(\gamma,(0\leq t_1\leq \dots\leq t_n\leq 1))= (\gamma(0),\gamma(t_1),\dots, \gamma(t_n))$$
and consider integration along the fiber $\Delta^n$ in the diagram
\begin{equation*}
\xymatrix{
  LM\times \Delta^n \ar[r]^{ev} \ar[d]_{\int_{\Delta^n}} & M^{\times {(n+1)}} \\ LM & }
\end{equation*}
then, up to a sign of $(-1)^{(n-1) |a_1| +(n-2)|a_2| + \cdots + |a_{n-1}| + n(n-1)/2}$,
we can write the iterated integral as the composition (see \cite[section 1]{GJP})
\[
It|_{\Om(M)^{\ot {n+1}}}: \Om(M,\gl)\ot\Om(M,\gl)[1]^{\ot n}\stackrel {ev^*} \longrightarrow \Om(LM\times \Delta^n,\gl)\stackrel {\int_{\Delta^n}} \longrightarrow \Om(LM,\gl).
\]

We have the following lemma concerning the relationship of holonomy and the algebraic structure of the Hochschild complex.
\begin{lem} \label{hol=It(e^A)}
Given $A \in \Omega^1(M; \gl )$ we have 
\[
hol = It(e^{1 \ot A} )
\]
where $e^{1 \ot A} \in CH_\bullet(\Omega(M; \gl(V)) ) $ is given by
\[
e^{1 \ot A} = \sum_{k \geq 0} \frac{(1 \ot A)^{\bullet k}}{k!} = 1 + 1 \ot A + 1 \ot A \ot A + \dots
\]
\end{lem}

\begin{proof} The signs in the shuffle products and the iterated integral are all positive, 
so the formula for $It(e^{1 \ot A} )$ agrees with the definition of holonomy in equation \eqref{Pexp}.
\end{proof}

From the theory of ODE's, it is well known that the holonomy function (or more generally parallel transport) satisfies a gluing property along composable paths.
We give a purely iterated integral proof of this fact, which we will also use later.

\begin{lem} \label{holglue}
Let $X$ be an odd form on $M$ with values in $\gl$. Let $\gamma_1 :[a,b] \to M$ and 
$\gamma_2: [b,c] \to M$ such that $\gamma_1(b) = \gamma_2(b)$, and let $\gamma=\gamma_2 \circ  \gamma_1: [a,c] \to M$ be the composition of the paths. Then
\[
It (1, \underbrace{ X , \ldots , X }_k )(\gamma) = \sum_{\stackrel{k_1, k_2 \geq 0}{k_1 + k_2 = k}} 
It (1, \underbrace{ X , \ldots , X }_{k_1} ) (\gamma_1) \cdot It (1,  \underbrace{ X , \ldots , X }_{k_2} ) (\gamma_2)
\]
Moreover, if $A \in \Om(M;\gl)$, then we have
\[
It\left(e^{1\ot A}\right) (\gamma) = It\left(e^{1\ot A} \right)(\gamma_1) \cdot It \left(e^{1 \ot A} \right)(\gamma_2)
\]
\end{lem}

\begin{proof} 
We denote by $\Delta^j_{[r,s]} = \{ r \leq t_1\leq \dots \leq t_j \leq s \} $. For the first statement, we must show
\begin{multline*}
\int_{\Delta^k_{[a,c]}} \iota_\Vf X_1 (t_1)\dots \iota_\Vf X(t_k) dt_1\dots dt_k  \\ = \sum_{\stackrel{k_1, k_2 \geq 0}{k_1 + k_2 = k}} 
\left( \int_{\Delta^{k_1}_{[a,b]}} \iota_\Vf X_1 (t_1)\dots \iota_\Vf X(t_{k_1})dt_1\dots dt_{k_1} \right) \\
\cdot \left( \int_{\Delta^{k_2}_{[b,c]}} \iota_\Vf X_1 (t_1)\dots \iota_\Vf X(t_{k_2})dt_1\dots dt_{k_2} \right).
\end{multline*}
 We use the calculus notation for the integral over $\Delta^k_{[a,b]}$,
 \[
\int_{\Delta^k_{[a,b]}}(\dots) dt_1\dots dt_k= \int_{a}^{b}  \int_{a}^{t_k} \cdots  \int_{a}^{t_3} \int_{a}^{t_2} (\dots) dt_1\dots dt_k.
 \]
With this notation, we can write
\begin{multline*}
\sum_{\stackrel{k_1, k_2 \geq 0}{k_1 + k_2 = k}} \left( \int_{a}^{b} \cdots \int_{a}^{s_2} \iota_\Vf X(s_1) \dots \iota_\Vf X(s_{k_1})ds_1\dots ds_{k_1} \right)  \\
\cdot \left( \int_{b}^{c} \cdots \int_{b}^{t_2} \iota_\Vf X(t_1) \dots \iota_\Vf X(t_{k_2})dt_1\dots dt_{k_2} \right) \\
 =  \sum_{\stackrel{k_1, k_2 \geq 0}{k_1 + k_2 = k}} \int_{b}^{c} \cdots \int_{b}^{t_2}   \int_{a}^{b} \cdots \int_{a}^{s_2}
\iota_\Vf X(s_1)  \dots \iota_\Vf X(s_{k_1}) \iota_\Vf X(t_1) \dots \iota_\Vf X(t_{k_2}) \\
ds_1\dots d s_{k_1} dt_1\dots dt_{k_2}.
\end{multline*}
Since $\Delta^k_{[a,c]}\simeq \bigcup_{k_1+k_2=k} \Delta^{k_1}_{[a,b]} \times \Delta^{k_2}_{[b,c]}$ , we may repeatedly use the relation $\int_a^b + \int_b^{t_{j}} = \int_a^{t_{j}}$ to obtain
\[
 \int_{a}^{c} \cdots \int_{a}^{t_2} \iota_\Vf X(t_1) \dots \iota_\Vf X(t_k) dt_1\dots dt_k= \int_{\Delta^k_{[a,c]}} \iota_\Vf X_1 (t_1)\dots \iota_\Vf X(t_k) dt_1\dots dt_k.
\]
The second part follows from the first, since
\begin{eqnarray*}
&& It\left(e^{1\ot A} \right)(\gamma_1) \cdot It \left(e^{1 \ot A} \right)(\gamma_2)
 \\ 
 & =&\left( \sum_{k_1=0}^\infty 
\int_{a}^{b} \cdots \int_{a}^{s_2} \iota_\Vf A(s_1) \dots \iota_\Vf A(s_{k_1})ds_1\dots ds_{k_1} \right) \\
&& \quad\quad\quad\quad\quad\cdot 
\left( \sum_{k_2=0}^\infty \int_{b}^{c} \cdots \int_{b}^{t_2} \iota_\Vf A(t_1) \dots \iota_\Vf A(t_{k_2}) dt_1\dots dt_{k_2} \right) \\
&=& \sum_{k=0}^\infty \sum_{\stackrel{k_1, k_2 \geq 0}{k_1 + k_2 = k}}  
\int_{a}^{b} \cdots \int_{a}^{s_2} \iota_\Vf A(s_1) \dots \iota_\Vf A(s_{k_1})ds_1\dots ds_{k_1}\\
&& \quad\quad\quad\quad\quad\cdot 
 \int_{b}^{c} \cdots \int_{b}^{t_2} \iota_\Vf A(t_1) \dots \iota_\Vf A(t_{k_2})  dt_1\dots dt_{k_2} 
\\
&=& \sum_{k=0}^\infty   \int_{a}^{c} \cdots  \int_{a}^{t_2} \iota_\Vf A(t_1) \cdots \iota_\Vf A(t_k)  dt_1\dots dt_k=It\left(e^{1\ot A}\right) (\gamma).
\end{eqnarray*}
\end{proof}
\begin{rmk} \label{rmk:moreglue}
Under the hypothesis of the previous lemma, a similar argument shows
\begin{multline*}
It\left(
(1 \ot \omega^{\ot k}) \bullet e^{1\ot A}\right) (\gamma) \\ = 
\sum_{\stackrel{k_1, k_2 \geq 0}{k_1 + k_2 = k}} 
It\left( (1 \ot \omega^{\ot {k_1}}) \bullet e^{1\ot A} \right)(\gamma_1) \cdot It \left( (1 \ot \omega^{\ot {k_2}}) \bullet e^{1 \ot A} \right)(\gamma_2)
\end{multline*}
for any odd form $\omega$.
\end{rmk}

We are now interested in computing the exterior derivative of the holonomy function. We note that in the case of $\gl$-valued forms, the iterated integral map is not a chain map with respect to $D$ on $CH_\bullet(\Om(M;\gl))$ and the exterior $d_{DR}$ on $\Om(LM;\gl)$. We do have the following proposition, which is essentially proved in \cite{C1, C2, GJP}. 
\begin{prop} \label{prop:dIt}
We have:
\begin{multline*}
d_{DR}( It(a_0 \ot a_1 \ot \cdots \ot a_n)) \\
 = - \sum_{i=0}^n  (-1)^{|a_0| + \cdots + |a_{i-1}| + i -1} It ( a_0 \ot a_1 \ot \cdots \ot d(a_i) \ot \cdots \ot a_n ) \\
- \sum_{i=0}^{n-1} (-1)^{|a_0| + \dots +  |a_i| - i} It( a_0 \ot a_1 \ot \cdots \ot (a_{i}\cdot a_{i+1}) \ot \cdots \ot a_n ) \\
+ (-1)^{(|a_n|+1)(|a_0| + \dots + |a_{n-1}| + n -1)} It ( (a_n\cdot a_0) \ot a_1\ot \cdots \ot a_{n-1} )  \\
- (-1)^{(|a_n|+1)(|a_0| + \dots + |a_{n-1}| + n -1)} a_n \cdot It (a_0 \ot a_1 \ot \dots \ot a_{n-1})  \\
+ (-1)^{|a_0| + |a_1| + \cdots + |a_{n-1}| + n - 1} It(a_0 \ot a_1\ot \cdots \ot a_{n-1}) \cdot a_n.
\end{multline*}
\end{prop}
\begin{proof} See, for example, \cite[Proposition 1.6]{GJP}. A more general version will be proved in proposition \ref{dItpU} below.
\end{proof}

From this we can conclude

\begin{prop} \label{prop:nabla}
Let $A \in \Omega^1(M; \gl)$. Then,
\[
It(D e^{1 \ot A})  = \nabla^* It(e^{1\otimes A})  = \nabla^* hol\in \Omega^1(LM; \gl(V))
\]
In particular, $\nabla^* hol$ is a Chen form.
\end{prop} 

\begin{proof} Applying the previous proposition to 
\[
1 \ot \underbrace{A \ot \cdots \ot A}_k
\]
and summing over all $k \geq 0$ we get
\[
d_{DR} It( e^{1 \ot A}) = It(D e^{1 \ot A}) - A It( e^{1 \ot A}) + It( e^{1 \ot A})A
\]
Since $\nabla^*$ is the pullback of the induced connection from $End(E) \to M$, it is 
locally given by $\nabla^* =  d_{DR} + [A,-]$, so that the statement follows. 
\end{proof}

\begin{prop}  \label{prop:De^a}
We have the following explicit formula
\[
D e^{1 \ot A} = \left(1 \ot (-R)\right) \bullet e^{1 \ot A}
\]
and the associated iterated integral formula 
\[
\nabla^* hol = \sum_{k \geq 0} \left(\sum_{i=1}^k \int_{\Delta^k} - \iota_\Vf A(t_1) \cdots  \iota_\Vf A(t_{i-1}) \iota_\Vf R(t_i) \iota_\Vf A(t_{i+1})\cdots \iota_\Vf A(t_k) dt_1\dots dt_k
\right).
\]
For line bundles we have
\[
\nabla^* hol = It\left(1 \ot (-R)\right) \wedge hol = \left( - \int_I \iota_\Vf R(t) dt \right) \wedge hol
\]
\end{prop}

\begin{proof}
The formula $D e^{1 \ot A} = (1 \ot (-R)) \bullet e^{1 \ot A}$ follows immediately from the Hochschild differential since 
since $R = dA + A^2$.

The second equation is obtained by applying the Chen iterated integral to $D e^{1 \ot A} = (1 \ot (-R)) \bullet e^{1 \ot A}$, and using proposition \ref{prop:nabla} and the formula for the Chen iterated integral, equation \eqref{chen}.

For line bundles we have $\C$-valued forms, so that the third equation follows from the well-known fact that the iterated integral takes the shuffle product in the Hochschild complex to the wedge product of forms, \emph{c.f.} \cite[proposition 4.1]{GJP} or the more general proposition \ref{LEM:It-diff-prod} below.
\end{proof}

\begin{cor} \label{COR:nabla-hol02}\quad
\begin{enumerate}
\item
The connection $A$ is flat if and only if $De^{1 \ot A} = 0$, and furthermore, if $A$ is flat then $hol= It(e^{1 \ot A})$ is $\nabla^*$-flat, \emph{i.e.} $\nabla^*(hol)=0$.
\item
For the natural vector field $\Vf$ on $LM$ given by the circle action, we have that
\[
\iota_{\Vf} \nabla^* hol= 0 
\]
\end{enumerate}
\end{cor}

\begin{proof} The two claims follow from the preceding proposition. For the second we use the fact that for any $\gamma$, the 1-form $\nabla^* hol$ vanishes on $\Vf$ at $\gamma$ since the factor $R(\gamma'(t), x)$ in the integrand is zero when $x = d/dt (\gamma) = \gamma'$.
\end{proof}

Our final lemma concerning the properties of $hol$ is the following
\begin{lem} \label{dtr}
Let $tr : \Omega(LM; \gl) \to \Omega(LM; \C)$ be given by the trace. Then
\[
d_{DR} (tr (hol)) =  tr (\nabla^*( hol)) 
\]
\end{lem}
\begin{proof} By the chain rule we have
\[
d_{DR} (tr (hol)) = tr (d_{DR} (hol)) = tr \left(d_{DR} (hol) + [A,hol] \right) = tr (\nabla^* (hol)).
\]
\end{proof}

\subsection{The higher holonomies $hol_{2k}$}\label{SUBSEC:2-higher-hol}
Using the facts from the previous sections, we now describe how the equivariant Chern character arises naturally by inductively solving a sequence of linear equations involving $\Omega^{2k}(LM; \gl(\C))$.
Starting from the holonomy function $hol$, which we denote by 
\begin{equation}\label{EQ:hol_0}
 hol_0 := hol\in \Omega^0(LM; \gl(\C^n)),
\end{equation}
we seek higher forms $hol_{2k} \in \Omega^{2k}(LM; \gl(\C^n))$  such that for all $k\geq 0$ we have
\begin{equation}\label{zig-zag-w/trace}
\nabla^* (hol_{2k}) =  -\iota_{\Vf}(hol_{2k+2}).
\end{equation}
The last equation \eqref{zig-zag-w/trace} is depicted in figure \ref{FIG:nabla_hol=-i_hol}.
\begin{figure}
\begin{equation*}
\xymatrix{
hol_0 \ar[dr]^{\nabla^*} & & hol_2 \ar[dl]_{-\iota_\mathbf{ t}}  \ar[dr]^{\nabla^*} & & hol_4 \ar[dl]_{-\iota_\mathbf{ t}} \ar[dr]^{\nabla^*}& & hol_6 \ar[dl]_{-\iota_\bold t}\ar[dr]^{\nabla^*}\\
 & It(D(e^{1 \ot A})) & &\bullet &&  \bullet && \cdots
}
\end{equation*}
\caption{The relation $\nabla^* (hol_{2k}) =  -\iota_{\Vf}(hol_{2k+2})$}\label{FIG:nabla_hol=-i_hol}
\end{figure}
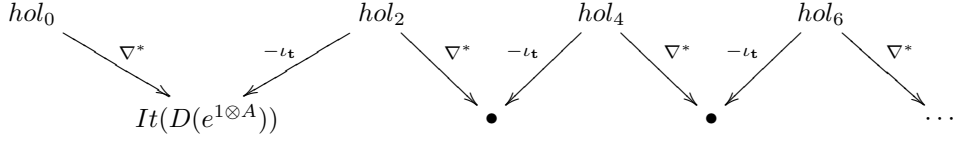
This is motivated in the first case, $\nabla^* hol_{0} =  -\iota_\Vf hol_{2}$, by the fact that
the necessary condition $\iota_\Vf \nabla^* hol_{0}=-(\iota_\Vf)^2 hol_2=0$ indeed holds by corollary \ref{COR:nabla-hol02}(2). To define the higher holonomies, it is convenient to introduce the following extended iterated integral map $E$.

\begin{defn} \label{Def:E}
Denote by $E: CH_\bullet(\Omega(M \times S^1;\gl)) \to \Omega(LM;\gl)$ the composition of the iterated integral map $It:CH_\bullet(\Omega(X\times S^1;\gl))\to \Om(L(X\times S^1);\gl)$ for the cartesian product $X\times S^1$, with the pullback of the map $\rho: \Omega(L(M \times S^1);\gl) \to \Omega(LM;\gl)$ of the map $\rho:LM \to L(M \times S^1)$ where $\rho(\gamma)(t)= (\gamma(t), -t)$. More explicitly, the extended iterated integral map is given by $E= \rho^* \circ It$,
\[
E :CH_\bullet(\Omega(M \times S^1;\gl)) \stackrel {It} \longrightarrow \Om(L(X\times S^1);\gl) \stackrel {\rho^*} \longrightarrow \Omega(LM;\gl).
\]
\end{defn}

\begin{rmk} The above definition is motivated by a similar extended iterated integral map used by Getzler, Jones, and Petrack in \cite{GJP} to define the equivariant Chern character on the free loop space. More precisely, the extended iterated integral map in \cite{GJP} is defined as the composition,
\begin{multline*}
CH_{k[u_2]} \Big(\Omega(X\times S^1)^{S^1}[u_2]\Big)[u_1]/((1,dt)-(1)) \otimes_{k[u_1,u_2]} \C[u,u^{-1}]] \\
\stackrel{It}\longrightarrow \Omega(L(X\times S^1))^{S^1}[u_1,u_2] \otimes_{k[u_1,u_2]} \C[u,u^{-1}]] \\
\stackrel{\rho^*}\longrightarrow \Omega(L(X))[u] \otimes_{k[u]} \C[u,u^{-1}]],
\end{multline*}
where $It$ and $\rho^*$ are extended linearly over $k[u]$, and $k[u]$ is thought of as a $k[u_1,u_2]$ module via $u_1\mapsto u$ and $u_2\mapsto u$.
\end{rmk}

An explicit formula for the map $E$ is given in \cite[section 5]{GJP} Namely,
\begin{multline} \label{E}
E((a_0 + b_0 dt) \ot  \dots \ot (a_n + b_n dt)) = \\
\int_{\Delta^n} a_0(0) (\iota_\Vf a_1(t_1) - b_1(t_1)) \dots  (\iota_\Vf a_n(t_n) - b_n(t_n)) dt_1 \cdots dt_n
\end{multline}
It follows that we can rewrite the holonomy using the map $E$.
\begin{lem}
We have $hol_0 =It(e^{1\otimes A})=E(e^{1\otimes A})$, where, by abuse of notation, the $1$-form $A$ in the last expression is taken as the pullback of $A\in\Om^1(M,\gl)$ along the projection $M \times S^1 \to M$.
\end{lem}
We now define the higher holonomies $hol_{2k}\in\Om^{2k}(LM;\gl)$, which serve as corrections for the failure of $\nabla^* hol=0$. Here, the term ``higher'' refers to the higher degree of the forms.
\begin{defn}\label{DEF::hol_2k} 
Denote by $A\in\Om^1(M,\gl)$ our connection $1$-form, by $R=dA+A^2\in\Om^2(M,\gl)$ the curvature, and denote by $dt\in\Om^1(S^1)$ be the canonical $1$-form on the circle $S^1$. By abuse of notation we will use the same symbols for the forms on $M\times S^1$ given by pullback under the projections $M\times S^1\to M$ and $M\times S^1\to S^1$, and we denote by $Rdt\in\Om^3(M\times S^1,\gl)$ the product of $R$ and $dt$. We define $hol_{2k} \in \Omega^{2k}(LM;\gl)$ by,
\begin{eqnarray*}
\mathfrak h_{2k} &:=&   \bigg(1 \ot  \underbrace{(-Rdt )\ot\dots\ot(-Rdt)}_{k\text{ tensor factors}}\bigg) \bullet e^{1 \ot A} \in CH_\bullet(\Omega(M\times S^1;\gl)), \\
hol_{2k}&:=& E(\mathfrak h_{2k}).
 \end{eqnarray*}
\end{defn}
The following lemma expresses the higher holonomies as iterated integrals. Heuristically, $hol_{2k}$ is given by an iterated integral similar to $hol_0$, where we replace exactly $k$ copies of $\iota_\Vf A$ by the curvature $R$. 
\begin{lem} \label{lem:Ehol}
We have the identity
\[
hol_{2k} = \sum_{m \geq k} \, \,  \sum_{1 \leq j_1 < \dots < j_k \leq m} \int_{\Delta^m} X_1(t_1) \cdots X_m(t_m) dt_1 \cdots dt_m,
\]
where
\[
X_j (t_j) = \left\{
\begin{array}{rl}
R(t_j) & \text{if  } j \in \{ j_1, \dots , j_k\} \\
\iota_\Vf A(t_j)  & \text{otherwise}
\end{array} \right.
\]
Here $R(t_j)$ is a $2$-form taking in two vectors on a loop $\gamma$ at $\gamma(t_j)$. \end{lem}
\begin{proof}
If we apply the definition of shuffle product, all of the signs are positive, since $A$ and $Rdt$ are odd, and we obtain
\begin{equation*}
E(\mathfrak h_{2k} ) = E \left(    \sum_{m \geq k} \, \,  \sum_{1 \leq j_1 < \dots < j_k \leq m}  1 \ot X_1 \ot \cdots \ot X_m  \right)
\end{equation*}
where 
\[
X_j = \left\{
\begin{array}{rl}
-Rdt & \text{if  } j \in \{ j_1, \dots , j_k\} \\
A  & \text{otherwise}
\end{array} \right.
\]
Now using the formula for $E$ in \eqref{E}, we obtain the claim.
\end{proof}
We now prove our main theorem of this section.
\begin{thm}\label{THM:nabla=id/dt}
For all $k \geq 0$ we have that
\[
\nabla^* hol_{2k} = - \iota_\Vf hol_{2k+2}.
\]
These terms are given by the explicit formula
\[
\sum_{m \geq k+1}\sum_{\scriptsize
\begin{matrix}
1 \leq j_1 < \dots < j_{k} \leq m\\
1\leq s\leq m\text{, with }s\neq  j_1, \ldots , j_{k}
\end{matrix}
} \int_{\Delta^m} X_1(t_1) \cdots X_m(t_m) dt_1 \cdots dt_m,
\]
where
\[
X_j(t_j) = \left\{
\begin{array}{rl}
 R(t_j) & \text{if  } j \in \{ j_1, \dots , j_{k}\} \\
-\iota_\Vf R(t_j) & \text{if  } j =s \\
\iota_\Vf A(t_j)  & \text{otherwise.}
\end{array} \right.
\]
\end{thm}
\begin{proof}
We first show that  $\nabla^* E (\mathfrak h_{2k} )= ED(\mathfrak h_{2k})$, where $\nabla^* = d_{DR} + [A(0), -]$. To do this we calculate $dE(\mathfrak h_{2k})$.
 Expanding $E(\mathfrak h_{2k})$ as in the proof of \ref{lem:Ehol}, and using the fact that $E = \rho^* \circ It$ and $\rho^*$ is a chain map, it follows from proposition \ref{prop:dIt} that
\begin{multline*}
d E (hol_{2k}) 
=  \rho^* \Bigg\{ \sum_{m \geq k} \, \,  \sum_{1 \leq j_1 < \dots < j_k \leq m} \bigg[  - \sum_{i=1}^m 
It( 1 \ot X_1 \ot \cdots \ot d X_i \ot \cdots \ot X_m ) \\
- \sum_{i=1}^{m-1} It( 1 \ot X_1 \ot \cdots \ot  (X_i X_{i+1}) \ot \cdots \ot X_m )  \\
- X_1(0) \cdot It(1 \ot X_2 \ot  \cdots \ot X_m )
+ It(1 \ot X_1 \ot \cdots \ot X_{m-1})\cdot X_m(0) ] \bigg] \Bigg\},
\end{multline*}
where
\[
X_j = \left\{
\begin{array}{rl}
-Rdt & \text{if  } j \in \{ j_1, \dots , j_k\} \\
A  & \text{otherwise.}
\end{array} \right.
\]
and $X_i(0) = ev^*_0 (X_i)$ with $ev_0: L(M \times S^1) \to M \times S^1$ being the evaluation at time zero. Note that since the composition $ev_0 \circ \rho: LM \to L(M \times S^1) \to M \times S^1$ sends $\gamma$ to $(\gamma(0), 0)$, the pullback $\rho^* (Rdt(0))$ is zero. Therefore, the last two terms in the sum are zero whenever $X_1 = -Rdt$, or 
$X_m = -Rdt$, respectively. So we have
\begin{multline} \label{dEhol2k}
d_{DR} \circ \rho^* \circ It (\mathfrak h_{2k})= \rho^* \circ d_{DR} \circ It (\mathfrak h_{2k}) \\
 =  \rho^* \Bigg\{ \sum_{m \geq k} \, \,  \sum_{1 \leq j_1 < \dots < j_k \leq m} \bigg[  -\sum_{i=1}^m
It( 1 \ot X_1 \ot \cdots \ot d X_i \ot \cdots \ot X_m ) \\
- \sum_{i=1}^{m-1} It( 1 \ot X_1 \ot \cdots \ot  (X_i X_{i+1}) \ot \cdots \ot X_m )  \\
- A(0) \cdot It(1 \ot X_1 \ot \cdots \ot X_m )
+ It(1 \ot X_1 \ot \cdots \ot X_{m})\cdot A(0) ] \bigg] \Bigg\}
\end{multline}
with the same $X_j$ as before. On the other hand, by definition \ref{DEF:Hoch+shuffle},
\begin{multline}\label{EQ:D(h2k)-Ch2}
D \mathfrak h_{2k} = 
\left\{ \sum_{m \geq k} \, \,  \sum_{1 \leq j_1 < \dots < j_k \leq m} \left[  -\sum_{i=1}^m  1 \ot X_1 \ot \cdots \ot d X_i \ot \cdots \ot X_m ) \right.\right.
\\
\left.\left. - \sum_{i=1}^{m-1}  1 \ot X_1 \ot \cdots \ot  (X_i X_{i+1}) \ot \cdots \ot X_m ) \right] \right\},
\end{multline}
since for all $X_1, \ldots , X_m$, with $X_j \in \{ A, -Rdt \}$, the terms  $A \ot X_1 \ot  \cdots \ot X_m$ and $Rdt \ot X_1 \ot \cdots \ot X_m$ all appear twice and cancel by sign due to $A$ and $Rdt$ being odd. Therefore equation \eqref{dEhol2k} shows that
\[
\nabla^* E (\mathfrak h_{2k} )= d_{DR}E (\mathfrak h_{2k} )+[A(0),E (\mathfrak h_{2k} )] =\rho^* (It (D(\mathfrak h_{2k}))) =ED(\mathfrak h_{2k}).
\]

Now we calculate $\nabla^* (hol_{2k})=ED(\mathfrak h_{2k})$ explicitly. Using the relations $dA + A^2 = R$, $dR + [A,R] = 0$, and $(Rdt)^2=0$, we obtain from equation \eqref{EQ:D(h2k)-Ch2} that
\begin{equation*}
D(\mathfrak h_{2k})= \sum_{m \geq k+1} \, \, \sum_{s=1}^m \,\,  \sum_{\stackrel{1 \leq j_1 < \dots < j_{k} \leq m}{ j_1, \ldots , j_k \neq s}}  1 \ot X_1 \ot \cdots \ot X_m,
\end{equation*}
where 
\[
X_j = \left\{
\begin{array}{rl}
-R & \text{if  } j =s \\
 -Rdt & \text{if  } j \in \{ j_1, \dots , j_{k}\} \\
A  & \text{otherwise}.
\end{array} \right.
\]
By the formula \eqref{E} for the extended iterated integral $E$, we have that
\begin{equation*}
\nabla^* (hol_{2k})=ED(\mathfrak h_{2k}) =   \sum_{m \geq k+1} \,\,\sum_{s=1}^m \,\,  \sum_{\stackrel{1 \leq j_1 < \dots < j_{k} \leq m}{ j_1, \ldots , j_k \neq s}} \int_{\Delta^m} X_1(t_1) \cdots X_m(t_m) dt_1 \cdots dt_m,
\end{equation*}
where 
\[
X_j(t_j) = \left\{
\begin{array}{rl}
-\iota_\Vf R(t_j) & \text{if  } j =s \\
 R(t_j) & \text{if  } j \in \{ j_1, \dots , j_{k}\} \\
\iota_\Vf A(t_j)  & \text{otherwise}
\end{array} \right.
\]
On the other hand, this equals
\[
-\iota_\Vf (hol_{2k+2})= -\iota_\Vf\Bigg( \sum_{m \geq k+1} \, \,  \sum_{1 \leq j_1 < \dots < j_{k+1} \leq m} \int_{\Delta^m} X_1(t_1) \cdots X_m(t_m) dt_1 \cdots dt_m\Bigg),
\]
where
\[
X_j = \left\{
\begin{array}{rl}
R(t_j) & \text{if  } j \in \{ j_1, \dots , j_{k}\} \\
\iota_\Vf A(t_j)  & \text{otherwise}
\end{array} \right.
\]
since contraction $\iota_\Vf$ is a linear graded derivation of square zero.
\end{proof}
For our starting data of a vector bundle $E \to M$ with connection $\nabla$, there is a nice way to combine all of this information as a closed periodic form in the following way. Consider 
\[
Ch^{(u)}(E; \nabla) := \sum_{k\geq 0} u^{-k}\cdot hol_{2k},
\]
where $u$ is a formal variable of degree $2$. From the discussion above we have that 
\begin{cor}\label{COR:nabla+i(Chern)=0}
$ (\nabla^* + u \cdot \iota_\Vf) \left( Ch^{(u)}(E; \nabla)  \right) = 0. $
\end{cor}
We note that the formula in corollary \ref{COR:nabla+i(Chern)=0} holds before taking the trace. Upon taking trace, the element 
\[
tr(Ch(E; \nabla)) =  \sum_{k\geq 0} u^{-k}\cdot tr(hol_{2k} ) \in \Om(LM)[u,u^{-1}]]
\]
satisfies by Lemma \ref{dtr}, that $(d + u\cdot \iota_\Vf) tr(Ch^{(u)}(E;\nabla)) = 0$, and by theorem \ref{THM:nabla=id/dt} that its Lie derivative $\mathcal L_\vf=d \iota_\vf+\iota_\vf d$ along $\vf$ vanishes. This element $tr(Ch^{(u)}(E; \nabla))$ is referred to in \cite{GJP} as the \emph{equivariant Chern character}. We thus may consider the complex of invariant forms
\[
\Om(LM)[u,u^{-1}]]^{inv(\vf)}=\bigg\{ \sum_{i\leq r} u^i\cdot \omega_i \Big| r\in \Z, \omega_i\in \Omega(LM) \text{ with } \mathcal L_\vf(\omega_i)=0 \bigg\},
\]
which has a differential $d+u\cdot \iota_\vf$ of square zero, since $(d+u\cdot \iota_\vf)^2=u\cdot \mathcal L_\vf=0$ on $\Om(LM)[u,u^{-1}]]^{inv(\vf)}$. Then, $tr(Ch(E; \nabla))\in \Om(LM)[u,u^{-1}]]^{inv(\vf)}$ is a closed element of total degree $0$.

For later purposes, it will also be useful to define the equivariant Chern character in a complex without formal variables, see \cite{AB, B, W}. More precisely, we denote by $\Omega(LM)^{inv(\vf)}=\{\omega\in\Omega(LM)| \mathcal L_\vf(\omega)=0\}$ the space of invariant forms on the free loop space $LM$ with the Witten differential $d+\iota_\vf$. Using the same reasoning as above, we set
\[
Ch(E; \nabla) := \sum_{k\geq 0}  hol_{2k},
\]
and, after taking the trace, we obtain an element that is concentrated in even degrees, $tr(Ch(E;\nabla))\in \Omega^{even}(LM)^{inv(\vf)}$, and moreover closed in $\Omega(LM)^{inv(\vf)}$, \emph{i.e.} $(d+\iota_\vf)(tr(Ch(E; \nabla)))=0$.

\section{Equivariant Chern character locally defined}\label{SEC:local-vector-bundle}
In this section, we give a local approach to defining the equivariant Chern character on general bundles, inspired by a local formula for holonomy. To this we introduce a generalization of the Hochschild complex that allows as input differential forms on local charts spread out over a manifold.

\subsection{The local Hochschild complex and iterated integral}

For any $p\in \mathbb N$, and $p$ open sets $U_{i_1},\dots, U_{i_p}$ from the cover $\{U_i\}_i$, which we abbreviate by $\U=(U_{i_1},\dots, U_{i_p})$, there is an induced open subset $\NN(p,\U)\subset LM$ given by
\[
\NN(p,\U)=\left\{\gamma\in LM: \left(\gamma\Big|_{\big[\frac{j-1}{p},\frac{j}{p}\big]}\right)\subset U_{i_j}, \forall j=1,\dots,p \right\}.
\]
Note that the collection $\{\NN(p,\U)\}_{p,i_1,\dots,i_p}$ forms an open cover of $LM$.

We now introduce a version of the Hochschild complex that is appropriate for dealing with forms defined on various open sets $U_{i_j}$.

\begin{defn}\label{DEF:local-Hoch}
 Fix $p\in \N$, and open sets $U_{i_1},\dots, U_{i_p}\subset M$. We define the 
 \emph{Hochschild chain complex subject to the data $\U=(U_{i_1},\dots, U_{i_p})$}, by 
\begin{multline*}
CH_\bullet^{(p,\U)}=\bigoplus_{n_1,\dots,n_p\geq 0}\Om(U_{i_p}\cap U_{i_1})\ot \left(\Om(U_{i_1})[1]\right)^{\otimes n_1}\otimes \Om(U_{i_1}\cap U_{i_2})\otimes \left(\Om(U_{i_2})[1]\right)^{\otimes n_2}\\
\otimes \Om(U_{i_2}\cap U_{i_3})\otimes \left(\Om(U_{i_3})[1]\right)^{\otimes n_3}\otimes \dots \otimes 
\left(\Om(U_{i_p})[1]\right)^{\otimes n_p}.
\end{multline*}
Note that the total degree of a monomial in $CH_\bullet^{(p,\U)}$ is then given by (the total degree of all forms)$-(n_1+\dots+n_p)$. We denote by $CH_k^{(p,\U)}$ the subspace of elements of degree $k$. 
The differential in $CH_\bullet^{(p,\U)}$ is similar to the one in $CH_\bullet(A)$ from definition \ref{DEF:Hoch+shuffle}.
To define it, we use the fact that $\Om(U_{i_j}\cap U_{i_{j+1}})$ is a left-$\Om(U_{i_j})$ and right-$\Om(U_{i_{j+1}})$ module via the inclusions $U_{i_j}\cap U_{i_{j+1}}\hookrightarrow U_{i_j}$ and $U_{i_j}\cap U_{i_{j+1}}\hookrightarrow U_{i_{j+1}}$. 
Then the differential on $CH_\bullet^{(p,\U)}$ is given by applying the DeRham differential to each term, as well as  multiplying adjacent elements cyclically in all possible positions, without ever multiplying two module elements. Explicitly,
\begin{multline*}
D\Big(a^1_0 \ot (a^1_{1}\ot  \dots \ot  a^1_{n_1})\ot a^2_0 \ot \dots\ot a^p_0 \ot(a^p_1\ot\dots\ot a^p_{n_p}) \Big)
\\ = - \sum_{i=1}^p \sum_{j=0}^{n_j} (-1)^{\epsilon^i_{j-1}} a^1_0 \ot (a^1_{1}\ot  \dots \ot  a^1_{n_1})\ot a^2_0 \ot \dots \ot d_{DR} (a^i_j) \ot \dots \ot a^p_0 \ot(a^p_1\ot\dots\ot a^p_{n_p})   \\
- \sum_{i=1}^p \sum_{j=0}^{n_i -1} (-1)^{\epsilon^i_j} a^1_0 \ot\dots\ot(a^i_j\cdot a^i_{j+1})\ot\dots \ot a^p_{n_p}  \\
- \sum_{i=1}^{p-1} (-1)^{\epsilon^i_{n_i}} a^1_0 \ot\dots\ot(a^i_{n_i} \cdot a^{i+1}_0)\ot\dots \ot a^p_{n_p}  +
(-1)^{(|a^p_{n_p}| + 1)\cdot \epsilon^p_{n_p -1}} (a^p_{n_p} \cdot a^1_0)\ot a^1_1\ot \dots\ot a^p_{n_p-1}
\end{multline*}
where for $1 \leq i \leq p$ and $0 \leq j \leq n_i$ we have
\[
\epsilon^i_j = \sum_{k=1}^{i-1}\left( |a_0^k|+ \sum_{\ell=1}^{n_k} (|a^k_\ell| + 1) \right) +|a_0^i|+ |a^i_1| + \cdots + |a^i_j| + j.
\]
Note that the signs here can be interpreted via the Koszul rule where the elements $a^i_j$ for $j>0$ are shifted by one. For $p=1$ this complex is precisely the Hochschild complex $CH_\bullet(\Om(U_1))$. It follows that $D^2=0$.

We also have a shuffle product, which is given by individually shuffling each of the $p$ tensor factors,
\begin{multline*}
\bigotimes_{1\leq i\leq p}\left(a^i _0 \otimes \left(a^i_1\otimes \dots\ot a^i_{n_i}\right) \right)\bullet \bigotimes_{1\leq i\leq p}\left(b^i_0 \ot \left(a^i_{n_i+1}\otimes \dots\ot a^i_{n_i+m_i}\right) \right) \\
=\sum_{\scriptsize
\begin{matrix}
\sigma_1\in S(n_1,m_1)\\ \dots  \\ \sigma_p\in S(n_p,m_p)
\end{matrix}
} (-1)^\kappa
\bigotimes_{1\leq i\leq p}   \left((a^i_0 \cdot b^i_0 )\ot\left(a^i_{\sigma^{-1}_i(1)}\ot \dots\ot a^i_{\sigma^{-1}_i(n_i+m_i)} \right)\right).
\end{multline*}
Again, the signs here are determined as with the shuffle product on $CH_\bullet(A)$, via the Koszul rule where, for $j > 0$, $a^i_j$ and $b^i_j$ are shifted by one.

We remark that, as with $CH_\bullet(\Om (M;\gl))$, $D$ on $CH_\bullet^{(p,\U)}$ is not a derivation of the shuffle product on 
$CH_\bullet^{(p,\U)}$ (though it is in the case $\gl = \C$).
\end{defn}
\begin{defn} \label{localIt}
There is an iterated integral map 
\[
\Ch^{(p,\U)}:CH_\bullet^{(p,\U)}\to \Om( \NN(p,\U) ;  \mathcal E |_{\NN(p,\U)} )
\]
 given by
\begin{multline*}
 \Ch^{(p,\U)}(a_0^1\ot a^1_1\ot \dots \ot a^p_{n_p})(\gamma) \\
= \bigwedge_{j=1}^p \left( \int_{\Delta^{n_j}} a_0^j ( t_0 ) \wedge \iota_\vf a^j_1 ( t_1 ) \wedge \dots \wedge \iota_\vf a^j_{n_j} (t_{n_j} ) dt_1 \dots dt_{n_j} \right)
 \end{multline*}
where we have set for vector fields $y_1,\dots , y_m$ along $\gamma$,
\begin{multline*}
 a_0^j ( t_0 )(y_1, \ldots , y_m)  = (a_0^j)_{\gamma (r)} \left(y_1\left(\gamma\left(\frac{j-1}{p}\right)\right), \ldots , y_m\left(\gamma\left(\frac{j-1}{p}\right)\right) \right), \\
  \iota_\vf a^j_k (t) (y_1, \ldots , y_m)  = (a^j_k)_{\gamma(s )} \left(\gamma'(s), y_1\left(\gamma\left(\frac{j-1+t}{p}\right)\right), \ldots , y_m\left(\gamma\left(\frac{j-1+t}{p}\right)\right)\right).
 \end{multline*}
\end{defn}

We note that, as before, this iterated integral map can be interpreted using integration along the fiber in the diagram
\begin{equation*}
\xymatrix{
  \NN(p,\U)\times \left(\Delta^{n_1} \times \dots \times \Delta^{n_p} \right)  \ar[r]^-{ev} 
  \ar[d]_{\int_{\Delta^{n_1} \times \dots \times \Delta^{n_p}}} 
  & (U_{i_1})^{\times (n_1 +1)} \times \dots \times (U_{i_p})^{\times (n_p +1)} \\ 
  \NN(p,\U) & }
\end{equation*}
where the evaluation map $ev = ev^1 \times \dots \times ev^p$ is given by
\begin{eqnarray*}
\NN(p,\U)\times \Delta^{n_j} &\stackrel {ev^j} \longrightarrow & (U_{i_j})^{\times (n_j +1)}\\
(\gamma, (0 \leq t_1\leq \dots\leq t_{n_j} \leq 1)) & \mapsto & \left(\gamma \left( \frac{j-1}{p} \right) , \gamma\left(\frac {j-1+t_1} p \right),\dots, \gamma\left(\frac {j-1+t_{n_j}} p\right)\right).
\end{eqnarray*}
Then, up to a sign, we can write the iterated integral $It^{(p,\U)}$ as the composition
\[
\Ch^{(p,\U)}: CH_\bullet^{(p,\U)} \stackrel {ev^*} \to \Om \left( \NN(p,\U) \times (\Delta^{n_1} \times \dots \Delta^{n_p})  ; \mathcal E |_{\NN(p,\U)}   \right) \stackrel {\int}  \to \Om(\NN(p,\U) ; \mathcal E |_{\NN(p,\U)} )
\]

In the case where all open sets $U_i$ are chosen to be $U_i=M$, and $M$ is simply connected, then the iterated integral map is a quasi-isomorphism $It:CH_\bullet\to \Omega(LM)$, see K.-T. Chen's papers \cite{C1, C2} or also \cite[Proposition 2.5.3]{GTZ}.

\subsection{Holonomy expressed locally}
We assume we have a bundle and connection given by local data as follows. We are given an open cover $\{U_i\}_i$ of $M$, and the complex vector bundle $E\to M$ is given by the local transition functions $g_{i,j}:U_i\cap U_j\to Gl(\C^n)$. Furthermore, the connection $1$-form is given locally by $A_i\in \Om^1(U_i,\gl)$ such that $A_j=g_{i,j}^{-1}A_i g_{i,j}+g_{i,j}^{-1}d_{DR}(g_{i,j})$, or, in other words, $d_{DR}(g_{i,j})=g_{i,j} A_j-A_i g_{i,j}$. The curvature is given by $d_{DR}(A_i)+A_i \wedge A_i= R_i\in \Omega^2(U_i)$. With a choice of open covering, every abstract bundle with connection induces these data. For the remainder of this section, all differential forms are $\gl(\C)$-valued, and we remark $\gl=\gl (\C)\subset GL$.

We will locally define sections $hol_0^{(p,\U)}$ of $\mathcal{E} \to LM$ on open subsets $\NN(p,\U)$, and then show that these glue together to a global section $hol_0$ of $\mathcal{E} \to LM$ where $\mathcal{E}$ is the pullback 
in
\[
\xymatrix{
 \mathcal{E} \ar[r] \ar[d] & End(E) \ar[d] \\
 LM \ar[r]^-{ev_0} & M
 }
\]

Using the given local connection $1$-forms $A_i\in \Om^1(U_i,\gl)$ and transition functions $g_{i,j}\in \Om^0(U_i\cap U_j,Gl )$, we can define an element $\h^{(p,\U)}_0\in CH_0 ^{(p,\U)}$ as follows. The iterated integral of this element is exactly holonomy. Recall from the previous subsection, that we have chosen local data $\U=(U_{i_1},\dots, U_{i_p})$, which is used to determine the open set $\NN(p,\U)\subset LM$.
\begin{defn}\label{DEF:h0(p,U)}
Let $\h^{(p,\U)}_0\in CH_0 ^{(p,\U)}$ be given by 
$$ \h^{(p,\U)}_0=\sum_{n_1,\dots,n_p\geq 0}  g_{i_p,i_1}\ot A_{i_1}^{\otimes n_1}\otimes g_{i_1,i_2}\otimes A_{i_2}^{\otimes n_2} \otimes \dots \otimes g_{i_{p-1},i_p}\otimes A^{\otimes n_p}_{i_p}.
$$
Alternatively, we may write $ \h^{(p,\U)}_0$ as a shuffle product 
$$ \h^{(p,\U)}_0=\tilde g_{i_p,i_1}\bullet e^{\tilde A_{i_1}}\bullet \tilde g_{i_1,i_2} \bullet e^{\tilde A_{i_2}} \bullet \dots \bullet e^{\tilde A_{i_p}},$$
 where  
$$\tilde A_{i_j}= 1\otimes \dots\otimes A_{i_j}\otimes \dots \otimes 1\in\Om(U_{i_p}\cap U_{i_1})\ot \Om(U_{i_1})^{\otimes 0}\ot \dots\ot \Om(U_{i_j})^{\otimes 1}\otimes \dots \ot \Om(U_{i_p})^{\ot 0}$$
has the $1$-form $A_{i_j}\in \Omega^1(U_{i_j})$ at the $j^{\text{th}}$ open set $U_{i_j}$ and $1$s everywhere else, and 
$$ \tilde g_{i_j,i_{j+1}}= 1\otimes \dots\otimes g_{i_j,i_{j+1}}\otimes \dots \otimes 1\in \Om(U_{i_p}\cap U_{i_1})\ot \Om(U_{i_1})^{\otimes 0}\ot \dots\ot \Om(U_{i_j}\cap U_{i_{j+1}})\otimes \dots \ot \Om(U_{i_p})^{\ot 0}
$$
has the $0$-form $g_{i_j,i_{j+1}}\in\Omega^0(U_{i_j}\cap U_{i_{j+1}})$ in the $j^{\text{th}}$ module $\Om(U_{i_j}\cap U_{i_{j+1}})$ and $1$s everywhere else.

We denote by $hol_{0}^{(p,\U)}=\Ch^{(p,\U)}(\h^{(p,\U)}_0)\in \Om^0(\NN(p,\U);  \mathcal E |_{\NN(p,\U)} )$ the iterated integral applied to $\h^{(p,\U)}_0$.
\end{defn}
The locally defined sections $hol_{0}^{(p,\U)}$ glue together to give a globally defined section of $\mathcal{E} \to LM$.
\begin{prop}\label{prop:hol_0-locally}
The local sections $hol_{0}^{(p,\U)}\in \Om^0(\NN(p,\U);  \mathcal E |_{\NN(p,\U)} )$ determine a global section $hol_0\in\Om^0(LM; \mathcal{E})$ satisfying
$$ hol_0|_{\NN(p,\U)}=hol_0^{(p,\U)} \text{ on } \NN(p,\U).  $$
\end{prop}
\begin{proof}
We use the following two properties:
\begin{enumerate}
\item Subdivision: Fix $p$ and $\U = \{ U_{i_1},\dots, U_{i_p} \}$. Subdivide each of the $p$ intervals of $[0,1]$ into $r$ subintervals, and use the same open set $U_{i_j}$ for all of the $r$ subintervals of the $j^{\text{th}}$ interval, to give a new cover $\U'$ with 
\[
U'_{i_{m\cdot r-r+1}}= \dots =U'_{i_{m\cdot r}} =U_{i_m}
\]
 for $1 \leq m \leq p$. 
Then $\NN(p,\U)=\NN(r\cdot p, \U')$, and $hol_0^{(p,\U)}=hol_0^{(r\cdot p, \U')}$.
\item Overlap: For any $p$ and $\U = \{U_{i_1} , \ldots , U_{i_p} \}$ and 
 $\U' = \{U_{j_1} , \ldots , U_{j_p} \}$, we have 
\[
hol_0^{(p,\U')} (\gamma) = g_{i_p , j_p}^{-1}(\gamma(0)) hol_0^{(p,\U)}(\gamma) g_{i_p , j_p} (\gamma(1))
\]
for $\gamma \in  \NN(p,\U \cap \U')$ where 
where \[
\U \cap \U' = \{ (U_{i_1} \cap U_{j_1}), \ldots , (U_{i_p} \cap U_{j_p}) \}.
\]
\end{enumerate}
The lemma follows from these two facts, since for $\gamma\in \NN(p,U_{i_1},\dots)\cap \NN(p',U_{j_1},\dots)$, we may assume by (1) that $p=p'$, and then by $(2)$ the sections agree on the overlap, up to conjugation by $g_{i_p , j_p}$ 
at $\gamma(0) = \gamma(1)$. 
Since the fiber of $\mathcal{E}$ over $\gamma$ is $End(E_{\gamma(0)}, E_{\gamma(0)})$ and the 
transition functions $g_{ij}$ on $E$ induce transition functions on $\End(E) \to M$ given by conjugation by $g_{ij}$,
this implies we have a well defined section $hol_0$ of $\mathcal{E} \to LM$.

To prove (1), for any $1 \leq m \leq p$ and $1 \leq s \leq r-1$,  whenever $U_{i_{m\cdot r-s}} = U_{i_{m\cdot r-s+1}}$,
we have $g = g_{i_{m\cdot r-s} , i_{m\cdot r-s+1}} = id$, so that
\begin{multline}
\sum_{k_1, k_2 \geq 0} \int_{\Delta^{k_1} \times \Delta^{k_2}} \iota_\vf A(t_1) \cdots \iota_\vf A(t_{k_1}) g( t )
\iota_\vf A(s_1) \cdots \iota_\vf A(s_{k_2}) dt_1\dots ds_{k_2}
\\
 = \sum_{k \geq 0} \int_{\Delta^{k}} \iota_\vf A(t_1) \cdots \iota_\vf A(t_k) dt_1\dots dt_k
\end{multline}
by the gluing lemma \ref{holglue}. Applying this $r-1$ times on each of the $p$ subintervals, we obtain $hol_0^{(p,\U)}=hol_0^{(r\cdot p, \U')}$.

For the proof of $(2)$, let  $\U = \{U_{i_1}, \ldots, U_{i_p} \}$ and $\U' = \{ U_{j_1},  \ldots , U_{j_p}\}$. 
We will first restrict our attention to the first subinterval, and show that
\begin{equation} \label{moveg}
g_{i_m,j_m} \left( \frac{m-1}{p} \right) It (e^{1 \ot A_{j_m}}) = It (e^{1 \ot A_{i_m} }) g_{i_m,j_m} \left(\frac{m}{p} \right),
\end{equation}
where the iterated integral is computed over the interval $\left[\frac{m-1}{p},\frac{m}{p}\right]$ and $A_{i_m}$, $A_{j_m}$ are the local 1-forms on the sets $U_{i_m}$ and $U_{j_m}$, respectively. Applying the identities \eqref{moveg} and $g_{i_{m-1},i_m}g_{i_m,j_m} =g_{i_{m-1},j_m}= g_{i_{m-1},j_{m-1}}g_{j_{m-1},j_m}$ repeatedly to the expression
\begin{multline*}
hol_0^{(p,\U)} g_{i_p,j_p}(1) \\ = g_{i_p,i_1}(0) It(e^{1 \ot A_{i_1}}) g_{i_1,i_2}\left(\frac{1}{p}\right)  \cdots g_{i_{p-1},i_p}
\left(\frac{p-1}{p}\right) It(e^{1 \ot A_{i_p}}) g_{i_p,j_p}(1)
\end{multline*}
we obtain $g_{i_p,j_p} (0) hol_0^{(p,\U')}$, as desired in $(2)$.

Note the formula in \eqref{moveg} says that the iterated integral (or parallel transport) is independent of $A_i$ up to conjugation by $g_{ij}$, which is well know from bundle theory and the existence and uniqueness of ODE's. We will give a iterated integral proof here instead. 

We will prove \eqref{moveg} for $m=1$, noting that the other cases have a similar proof. Furthermore, we abbreviate $i=i_1$ and $j=j_1$. Recall that $A_j = g^{-1} A_i g + g^{-1} dg$ where $g = g_{i,j}: U_i \cap U_j \to GL$ is the coordinate transition function. We use the following multiplicative version of the fundamental theorem of calculus for the iterated integral. For $r < s$,
\begin{multline} \label{Gltrans}
 g(r) It(e^{1 \ot g^{-1}dg}) \\
  = g(r) \sum_{k \geq 0} \int_{\Delta^k_{[r,s]}} \iota_\vf(g^{-1} dg)(t_1) \dots \iota_\vf (g^{-1} dg)(t_{k}) dt_1\dots dt_k = g(s).
\end{multline}
Here the $k$-simplex used in the integral is $\Delta^k_{[r,s]}=\{r \leq t_1 \leq \dots \leq t_k \leq s\}$. (One proof of this is given observing the function
$f(s) = g(r) It(e^{1 \ot g^{-1}dg}) g(s)^{-1}$ satisfies $f(r) = Id$ and $f'(s) = 0$.) We use this to calculate 
\[
g(0) It \left(e^{1 \ot A_j} \right) = g(0) It \left(e^{1 \ot (g^{-1} A_i g + g^{-1} dg)} \right)   
\]
and show that it equals $It (e^{1 \ot A_i}) g\left(\frac{1}{k} \right)$.
From the definition of the shuffle product we have
\begin{multline*}
e^{1 \ot (g^{-1} A_i g + g^{-1} dg)} \\ = \sum_{\stackrel{k_1,\ldots , k_r \geq 0}{r\geq 1}} 
 1\otimes (g^{-1} dg)^{\ot {k_1}} \ot  g^{-1} A_i g \ot (g^{-1} dg)^{\ot {k_2}} \ot g^{-1} A_i g \ot  \cdots \ot g^{-1} A_i g \ot 
 (g^{-1} dg)^{\ot {k_r}}
\end{multline*}
and applying definition \ref{localIt} of $It$ we have
\begin{multline} \label{gItAj}
g(0) It (e^{1 \ot A_j}) \\
= g(0) \sum_{\stackrel{k_1,\ldots , k_r \geq 0}{r\geq 0}}  \int_{\Delta^{n_r+r-1}_{[0,1/p]}}
\iota_\vf (g^{-1} dg)(t_1) \cdots \iota_\vf (g^{-1} dg)(t_{k_1})  \wedge \iota_\vf (g^{-1} A_i g)(t_{k_1+1})\\
 \wedge \iota_\vf (g^{-1} dg)(t_{k_1+2}) 
\cdots \iota_\vf (g^{-1} dg)(t_{n_{r-1}+r-2}) \wedge \iota_\vf (g^{-1} A_i g)(t_{n_{r-1} +r-1})
\\
 \wedge \iota_\vf(g^{-1} dg)(t_{n_{r-1} + r}) \cdots \iota_\vf (g^{-1} dg)(t_{n_r+r-1}) dt_1\dots dt_{n_r+r-1},
\end{multline}
where $n_i = k_1 + \dots + k_i$.
For each $r$, we apply the identity in \eqref{Gltrans} $r$ times, showing all the iterated integrals of $g^{-1}dg$ that appear, collapse. In the first case we have
\[
g(0) \sum_{k_1 \geq 0}  \int_{\{0 \leq t_1\leq \dots \leq t_{k_1} \leq t_{k_1 +1} \}} \iota_\vf (g^{-1} dg)(t_1) \dots \iota_\vf (g^{-1} dg)(t_{k_1}) dt_1 \cdots dt_{k_1}= g(t_{k_1+1})
\]
reducing $g(0)It (e^{1 \ot A_j})$ to 
\begin{multline*}
g(0) It (e^{1 \ot A_j}) 
= \sum_{\stackrel{k_2,\ldots , k_r \geq 0}{r\geq 0}}  \int  \iota_\vf A_i(t_{k_1 +1}) g(t_{k_1+1}) \\
\wedge \iota_\vf (g^{-1} dg)(t_{k_1+2}) 
\cdots \iota_\vf( g^{-1} dg)(t_{n_{r-1} +r-2})  \wedge \iota_\vf(g^{-1} A_i g)(t_{n_{r-1} +r-1})
\\
\wedge \iota_\vf( g^{-1} dg)(t_{n_{r-1} + r}) \cdots \iota_\vf (g^{-1} dg)(t_{n_r+r-1}) dt_{k_1+1}\dots dt_{n_r+r-1},
\end{multline*}
since the terms $g(t_{k_1+1})$ and $g^{-1}(t_{k_1+1})$ cancel.
Similarly, for each $\ell$ with $1 < \ell \leq r$, fixing $t_{n_\ell + \ell}$, we have 
\begin{multline*}
\int_{ \Delta(\ell) }  
\iota_\vf A_i(t_{n_\ell + \ell}) g(t_{n_\ell + \ell}) \iota_\vf (g^{-1} dg)(t_{n_\ell + \ell+1}) \\
\dots \iota_\vf (g^{-1} dg)(t_{n_{\ell+1} + \ell}) 
dt_{n_\ell + \ell+1} \cdots dt_{n_{\ell+1} + \ell} = \iota_\vf A_i(t_{n_\ell + \ell}) g(t_{n_{\ell+1} + \ell + 1} )
\end{multline*}
where $\Delta(\ell) =  \{t_{n_\ell + \ell} \leq t_{n_\ell + \ell +1 }\leq \dots \leq t_{n_{\ell+1} + \ell} \leq t_{n_{\ell+1} + \ell+1} \}$.
Again,  $g(t_{n_{\ell+1} + \ell + 1} )$ cancels with $g^{-1}(t_{n_{\ell+1} + \ell + 1} )$ and, continuing in this way, we see that entire sum in \eqref{gItAj} collapses to 
\[
\sum_{r \geq 1} \int_{\Delta^{r-1}_{\left[0, 1/ p\right]}} \iota_\vf A_i(t_1) \cdots \iota_\vf A_i(t_{r-1}) g\left(\frac{1}{p} \right) dt_1 \cdots dt_{r-1} =
It (e^{1 \ot A_i}) g \left(\frac{1}{p}\right).
\]
This completes the proof of the proposition.
\end{proof}
In the first two subsections, we constructed a holonomy function under the assumption that the bundle is included in a trivial bundle. If we can choose our cover for $M$ to be $\{M\}$ itself, we obtain the following corollary.
\begin{cor}
The restriction of the section $hol_0$ from lemma \ref{prop:hol_0-locally} to any $\NN(1,\U)$ where $\U=\{U_1\}$ equals the function $hol_0$ defined on $LU_1$ in equation \eqref{EQ:hol_0}.
\end{cor}
\begin{proof}
In this case the bundle is trivial and we have that the iterated integral formulas are the same since $g_{11} = id$ on $U_1$.
\end{proof}

Again, as before, $It^{(p,\U)}$ is not a chain map for $\gl$-valued forms, but we do have the following.
\begin{prop} \label{dItpU} If $\mathfrak a = a^1_0 \ot (a^1_{1}\ot  \dots \ot  a^1_{n_1})\ot a^2_0  \ot \dots\ot a^p_0 \ot(a^p_1\ot\dots\ot a^p_{n_p}) \in CH_\bullet^{(p,\U)}$, then
\begin{multline*}
d_{DR} \left(It^{(p,\U)} (\mathfrak a)\right) \\
= It^{(p,\U)} (D \mathfrak a) - (-1)^{(|a^p_{n_p}| +1)\epsilon} a^p_{n_p}(0) \cdot It^{(p,\U)} (\mathfrak b) + (-1)^{\epsilon} It^{(p,\U)}(\mathfrak b) \cdot a^p_{n_p}(1),
\end{multline*}
where $\mathfrak b = a^1_0 \ot (a^1_{1}\ot  \dots \ot  a^1_{n_1})\ot a^2_0  \ot \dots\ot a^p_0 \ot(a^p_1\ot\dots\ot a^p_{n_p -1})$, and 
\[
\epsilon = \sum_{i=1}^p\left( |a^i_0| + \sum_{j=1}^{n_i} (|a^i_j| +1) \right)- (|a^p_{n_p}| +1).
\]
\end{prop}

\begin{proof} The proof uses the same techniques that were used in \cite{GJP} to prove proposition \ref{prop:dIt} stated above. Let $\epsilon^i_j$ be as in definition \ref{DEF:local-Hoch}, and note that $\epsilon = \epsilon^p_{n_p-1}$. We have
\begin{multline*}
d_{DR} (It^{(p,\U)} (\mathfrak a)) = \bigwedge_{i=1}^p  \int_{\Delta^{n_i}} (-1)^{\epsilon^{i-1}_{n_{i-1}} } (da^i_0)\left(\frac{i-1}{p}\right)
\iota_\vf a^i_1(t_1) \cdots \iota_\vf a^i_{n_i}(t_{n_i}) dt_1\dots dt_{n_k}\\
+
\sum_{i=1}^p \sum_{j=1}^{n_i} (-1)^{\epsilon^i_{j-1}} \Bigg[ \bigg(
\int_{\Delta^{n_1}} (a^1_0)\left(\frac{0}{p}\right) \iota_\vf a^1_1(t_1) \cdots  \iota_\vf a^1_{n_1}(t_{n_1}) dt_{1}\dots dt_{n_1}\bigg)\\
\wedge 
\cdots \wedge \bigg(
\int_{\Delta^{n_i}} (a^i_0)\left(\frac{j-1}{p}\right) \iota_\vf a^i_1(t_1) \cdots ([d,\iota_\vf] - \iota_\vf d) a^i_j(t_j) \cdots \iota_\vf a^i_{n_i}(t_{n_i})dt_{1}\dots dt_{n_i}
\bigg)\\
\wedge
\cdots \wedge
\bigg( \int_{\Delta^{n_p}} (a^p_0)\left(\frac{p-1}{p}\right) \iota_\vf a^p_1(t_1) \cdots \iota_\vf a^p_{n_p}(t_{n_p}) dt_{1}\dots dt_{n_p}\bigg) \Bigg]
\end{multline*}
where we have use identity $d \iota_\vf = [d,\iota_\vf] - \iota_\vf d$.
Now we use the fact that $[d,\iota_\vf]a = da/dt$, the fundamental theorem of calculus 
\[
\int_{t_j}^{t_{j+1}} [d,\iota_\vf]a(t) dt=\int_{t_j}^{t_{j+1}} \frac{d}{dt}a(t) dt
= a(t_{j+1})  - a(t_j),
\] 
and the relation 
\[
((\iota_\vf a_j)a_{j+1})(t_j)+ (-1)^{|a_j|} (a_j(\iota_\vf a_{j+1}))(t_j)=(\iota_\vf(a_j a_{j+1}))(t_j).
\]
Then we see this sum contains all of the terms in $It^{(p,\U)}(D\mathfrak a )$, except for the term $(-1)^{(|a^p_{n_p}| +1)\epsilon} a^p_{n_p}(0) \cdot It (\mathfrak b)$,
and it additionally includes the term $(-1)^{\epsilon}  It(\mathfrak b) \cdot a^p_{n_p}(1)$.
\end{proof}

\begin{prop} \label{nablapU}
We have $It^{(p,\U)}(D(\mathfrak h^{(p,\U)}_0))= \nabla^* (It^{(p,\U)}(\mathfrak h^{(p,\U)}_0))=\nabla^*(hol_0^{(p,\U)})$, where $\nabla^* = d_{DR} + [A_{i_p},-]$.
\end{prop}
\begin{proof} 
This follows from the previous proposition \ref{dItpU} by letting $\mathfrak a = \mathfrak h^{(p,\U)}_0$, (\emph{c.f.} also propositions \ref{prop:dIt} and \ref{prop:nabla}). Note, that the signs work out since $A_{i_p}$ is of degree $+1$.
\end{proof}

\begin{rmk}
In the complex $CH_\bullet^{(p,\U)}$ we chose for simplicity the modules on the $j^{\text{th}}$ vertex to be given by $\Om(U_{i_j}\cap U_{i_{j+1}})$, \emph{c.f.} definition \ref{DEF:local-Hoch}. For a conceptual link of this choice with the local approach for gerbes in subsection \ref{SUBSEC:local-gerbe-Hochschild-51} below, we also mention how we could choose alternate modules as follows. Beside the open sets $U_{i_1},\dots , U_{i_p}$, choose additional open sets $U_{\ell_1},\dots, U_{\ell_p}$, which are open sets chosen for the vertices of the subdivided circle. More precisely, we are considering the open subset of $LM$ given by,
\[
 \left\{\gamma\in LM: \left(\gamma\Big|_{\big[\frac{j-1}{p},\frac{j}{p}\big]}\right)\subset U_{i_j}\text{, and } \gamma\left(\frac j p\right)\in U_{\ell_j}, \forall j=1,\dots,p \right\}.
 \]
With this in mind, we consider the left-$\Om(U_{i_j})$ and right-$\Om(U_{i_{j+1}})$ module $\Om(U_{i_j}\cap U_{\ell_j})\otimes_{\Om(U_{\ell_j})}  \Om(U_{\ell_j}\cap U_{i_{j+1}})$. Note that there is a map
$$ \Om(U_{i_j}\cap U_{\ell_j})\otimes_{\Om(U_{\ell_j})}\Om(U_{\ell_j}\cap U_{i_{j+1}})\to \Om(U_{i_j}\cap U_{\ell_j}\cap U_{i_{j+1}}), $$ coming from the inclusions of open sets. 
We have the right module map $\rho:\Om(U_{i_j}\cap U_{\ell_j})\otimes_{\Om(U_{\ell_j})}\Om(U_{\ell_j}\cap U_{i_{j+1}})\to \Om(U_{i_j}\cap U_{i_{j+1}})$, which maps the element $g_{i_j,\ell_j}\otimes g_{\ell_j,i_{j+1}}$ to $g_{i_j,\ell_j}\cdot g_{\ell_j,i_{j+1}}=g_{i_j,i_{j+1}}$. There are appropriate Hochschild and iterated integral constructions, and the above considerations show that this approach is indeed equivalent to the previous setup.
\end{rmk}

\subsection{The higher holonomies $hol_{2k}$}

Let $\Om(LM,\mathcal{E})$ denote differential forms on $LM$ with values in the pullback bundle $\mathcal{E}$.
We now define the higher holonomies $hol_{2k}\in \Om^{2k}(LM,\mathcal{E})$ in a local way on charts $\NN(p,\U)$ and then show that they glue together properly to give a well defined element of $\Om^{2k}(LM,\mathcal{E})$. 
Finally, these will satisfy the relation from equation \eqref{zig-zag-w/trace}, namely $\nabla^* (hol_{2k}) =  -\iota_{\Vf}(hol_{2k+2})$,
\begin{equation*}
\xymatrix{
hol_0 \ar[dr]^{\nabla^*} & & hol_2 \ar[dl]_{-\iota_\vf}  \ar[dr]^{\nabla^*} & & hol_4 \ar[dl]_{-\iota_\vf}\ar[dr]^{\nabla^*} \\
 & \bullet & & \bullet && \cdots
}
\end{equation*}

\begin{defn}\label{DEF:E(p,U)}
For a choice of $p\in \N$ and a tuple of open sets $\U=(U_{i_1},\dots, U_{i_p})$ of $M$, we have induced open sets $\U\times S^1=(U_{i_1}\times S^1, \dots, U_{i_p}\times S^1)$ of $M\times S^1$. Denote by $E^{(p,\U)}:CH_\bullet^{(p,\U\times S^1)}\to \Omega(\NN(p,\U))$ the composition of maps
$$ E^{(p,\U)}:CH_\bullet^{(p,\U\times S^1)}\stackrel {It^{(p,\U\times S^1)}} \longrightarrow\Omega(\NN(p,\U\times S^1)) \stackrel {(\rho|_{\NN})^*} \longrightarrow \Omega(\NN(p,\U)), $$
where $\rho|_\NN:\NN(p,\U)\to \NN(p,\U\times S^1), \gamma\mapsto (\gamma,-id)$, is the restriction of $\rho$ from definition \ref{Def:E}. We call the $E^{(p,\U)}$ the extended iterated integral subject to the local data $\U$.

With this, we define $\h_{2k}^{(p,\U)}\in CH_{2k}^{(p,\U\times S^1)}$ to be given by the formula
\begin{multline*}
\h_{2k}^{(p,\U)} = \sum_{\scriptsize
\begin{matrix}
k_1, \dots, k_p \geq 0 \\
k_1+\dots+k_p=k
\end{matrix}}  
\tilde g_{i_p,i_1}\bullet \left( \Big(1 \ot \underbrace{ (-R_{i_1}dt) \ot \cdots \ot (-R_{i_1}dt) }_{k_1\text{ times}}\Big) \bullet e^{\tilde A_{i_1}}\right) \bullet \tilde g_{i_1,i_2}\\
\bullet \cdots \bullet  \tilde g_{i_{p-1},i_p} \bullet
 \left( \Big(1 \ot \underbrace{ (-R_{i_p}dt) \ot \cdots \ot (-R_{i_p}dt)}_{k_p\text{ times}}\Big) \bullet e^{\tilde A_{i_p}}\right) 
\end{multline*}
where $\tilde g_{i_j,i_{j+1}}$, $\tilde A_{i_j}$ are as before (see definition \ref{DEF:h0(p,U)}). We define
\[
hol_{2k}^{(p,\U)}=E^{(p,\U)}(\h_{2k}^{(p,\U)})\in \Omega^{2k}(\NN(p,\U),\gl).
\]
\end{defn}
Similarly to lemma \ref{lem:Ehol}, applying the definition of $E$ we can rewrite $hol_{2k}^{(p,\U)}$ in terms of iterated integrals.
\begin{lem} \label{lem:hol2k}
For $k\in \mathbb N$, $hol_{2k}^{(p,\U)} \in \Om^{2k}(LM,\gl)$ can be written as
\begin{multline}\label{AAAgA}
\sum_{n_1,\dots,n_p\geq 0} \quad \sum_{\scriptsize
\begin{matrix}
J\subset S\\
|J| = k
\end{matrix}
}\quad g_{i_p,i_1}\left(\gamma (0) \right) 
\wedge \bigg( \int_{\Delta^{n_1}} X^1_{i_1}\left(\frac {t_1} p\right) \cdots X^{n_1}_{i_1}\left(\frac {t_{n_1}} p\right) dt_1 \cdots dt_{n_1} \bigg)
\\
\wedge  g_{i_1,i_2}\left(\gamma\left(\frac 1 p\right)\right)
\cdots   g_{i_{p-1},i_p}\left(\gamma\left(\frac {p-1} p\right)\right)
\\
\wedge \bigg(
\int_{\Delta^{n_p}} X^1_{i_p}\left(\frac {p- 1 + t_1} p \right) \cdots X^{n_p}_{i_p}\left(\frac {p- 1+t_{n_p}} p\right) dt_1 \cdots dt_{n_p} 
\bigg)
\end{multline}
where the second sum is a sum over all $k$-element index sets $J\subset S$ of the sets
$S= \{(i_r,j): r=1,\dots, p,\text{ and } 1\leq j\leq n_r\}$, and
\[
X^j_i = \left\{
\begin{array}{rl}
R_i & \text{if  } (i,j) \in J\\
\iota_\vf A_i & \text{otherwise.}
\end{array} \right.
\]
\end{lem}
\begin{proof} This follows from the definition of $E$, as in the proof of Lemma \ref{lem:Ehol}.
\end{proof}

\begin{rmk} \label{abelianhol2k}
The expression for $hol_{2k}^{(p,\U)} $ simplifies greatly for line bundles. In this case, the 
forms are all complex valued, so the iterated integral takes shuffle to wedge. Therefore we can unshuffle the terms $Rdt$ in
$\h_{2k}^{(p,\U)}$, then apply the iterated integral, and write
 \begin{align*}
 hol_{2k}^{(p,\U)} (\gamma) &=  \left( \int_{\Delta^k} R(t_1) \cdots R(t_k) dt_1\dots dt_k\right) \cdot \bigwedge_{m=1}^p g_{i_{m-1},i_m} \left( \gamma \left( \frac{m-1}{p}
 \right) \right) e^{\int_{I^m_p} \iota_\vf A_{i_m} dt} \\
 &= \left( \int_{\Delta^k} R(t_1) \cdots R(t_k)dt_1\dots dt_k \right) \wedge hol_0^{(p,\U)}(\gamma)
 \end{align*}
 where we have set $I^m_p = [\frac{m-1}{p} ,\frac{ m}{p}]$ and $i_0 = i_p$. Here we have used that $R$ is a global form, and
 the iterated additivity in Lemma \ref{holglue}, namely
 \[
 \sum_{\scriptsize
\begin{matrix}
k_1, \dots, k_p \geq 0 \\
k_1+\dots+k_p=k
\end{matrix}} 
 \bigwedge_{i=1}^p \int_{\Delta^{k_i}_{\left[\frac{i-1}{p},\frac{i}{p}\right]}} R(t_1) \dots R(t_{k_i}) dt_1\dots dt_{k_i}= \int_{\Delta^k} R(t_1) \dots R(t_k)dt_1\dots dt_k
 \]
 where $k = k_1 + \dots + k_p$ and $\Delta^{k_i}_{\left[\frac{i-1}{p},\frac{i}{p}\right]} = \{ \frac{i-1}{p} \leq t_1 \leq \dots \leq t_{k_i} \leq \frac{i}{p} \}$, and $\Delta^k = \{ 0 \leq t_1 \leq \dots \leq t_{k} \leq 1 \}$. Alternatively, the first factor can also be written as
 \begin{multline*}
  \int_{\Delta^k} R(t_1) \dots R(t_k)dt_1\dots dt_k=E\Big(1\otimes (-Rdt)\otimes\dots\otimes(-Rdt)\Big)
  \\
  =E\left(\frac{1}{k!} (1\otimes(-Rdt))^{\bullet k}\right)=\frac{1}{k!}\left(\int_I R(t) dt \right)^{\wedge k}.
 \end{multline*}
\end{rmk}

\begin{prop}
The locally defined forms $hol_{2k}^{(p,\U)}$ define a globally defined $2k$-form $hol_{2k}\in \Om^{2k}(LM,\mathcal{E})$
with values in $\mathcal{E}$.
\end{prop}
\begin{proof}
We can repeat the proof as in proposition \ref{prop:hol_0-locally}, showing the subdivision and overlap properties. The subdivision property follows as in proposition \ref{prop:hol_0-locally}, using the gluing property of the iterated integral, see remark \ref{rmk:moreglue}.
The overlap property follows as in proposition \ref{prop:hol_0-locally} almost verbatim for $hol_{2k}$, where we replace $k$ of the $\iota_\vf A$'s by $R$, and using the fact that  $g_{ij} R_i = R_j g_{ij}$ (since $R$ is a global 2-form, so that $R_i = R_j$ on $U_i \cap U_j$).
\end{proof}

Finally, we state our main result of this section.
\begin{thm}\label{THM:local-nabla-hol_2k} 
For all $k \geq 0$ we have that
\[
\nabla^* hol_{2k} = -\iota_\vf hol_{2k+2} \in \Omega^{2k+1}(LM ; \mathcal{E})
\]
and this is given by the explicit formula
\begin{multline} \label{iddthol2k}
\sum_{n_1,\dots,n_p\geq 0} \quad\sum_{\scriptsize
\begin{matrix}
(i_r,s)\in S \\
J\subset S \\
|J| = k, (i_r,s)\notin J
\end{matrix}
}\quad g_{i_p,i_1}\left(\gamma\left(0 \right)\right)
\\
\wedge \bigg( \int_{\Delta^{n_1}} X^1_{i_1}\left(\frac {t_1} p\right) \cdots X^{n_1}_{i_1}\left(\frac {t_{n_1}} p\right) dt_1 \cdots dt_{n_1} \bigg) \wedge
\cdots
\wedge  g_{i_{p-1},i_p}\left(\gamma\left(\frac {p-1} p\right)\right)
\\
\wedge 
\bigg( \int_{\Delta^{n_p}} X^1_{i_p}\left(\frac {p-1 + t_1} p\right) \cdots X^{n_p}_{i_p}\left(\frac {p-1+t_{n_p}} p\right) dt_1 \cdots dt_{n_p} \bigg)
 \end{multline}
where $S= \{(i_r,j): r=1,\dots, p,\text{ and } 1\leq j\leq n_r\}$ and 
\[
X^j_i = \left\{
\begin{array}{rl}
 R_i & \text{if  } (i,j) \in J \\
-\iota_\vf R_i & \text{if  } (i,j) =(i_r,s) \\
\iota_\vf A_i & \text{otherwise.}
\end{array} \right.
\]
\end{thm}
\begin{proof}
The proof is the same as in theorem \ref{THM:nabla=id/dt}, where the idea is, as before, that $-\iota_\vf(hol_{2k})$ and $\nabla^* E\left( \h_{2k}^{(p,\U)} \right)$ are both given by formula \eqref{iddthol2k}, \emph{i.e.} a sum over all $n_1,\dots , n_p \geq 0$ of products of iterated integrals where exactly $k$ of the $X_i$'s are equal to $R$ and exactly one is equal to $(-\iota_\vf R)$. 

By lemma \ref{lem:hol2k} we see that $-\iota_\vf (hol_{2k+2})$ is equal to \eqref{iddthol2k}.
For $\nabla^* E\big( \h_{2k}^{(p,\U)} \big) = E\big(D \h_{2k}^{(p,\U)} \big)$, the calculation is the same as in theorem \ref{THM:nabla=id/dt}, where the only new feature is the apparent terms $g_{i,j}$ in $\h_{2k}^{(p,\U)}$. But, all of these terms in $E\big(D \h_{2k}^{(p,\U)} \big)$ cancel with correct sign since $d g_{i,j} - g_{i,j} A_j + A_i g_{i,j} =0$.
\end{proof}
From the previous theorem, it follows that 
\[
Ch(E; \nabla) = \sum_{k\geq 0} tr(hol_{2k})
\]
is in the kernel of the Lie derivative $\mathcal L_\vf=[d_{DR},\iota_\vf]$ (since $(d_{DR})^2=(\iota_\vf)^2=0$), and is therefore an element of  $\Omega(LM)^{inv(\vf)}$, just as in section \ref{SUBSEC:2-higher-hol}. Moreover,
\begin{cor}\label{COR:local-nabla+u(Ch)=0} $($\cite{B}, \cite{GJP}$)$
For any vector bundle $E$ with connection $\nabla$,
\[
( d +  \iota_\vf ) \left(Ch(E;\nabla)\right)= 0,
\] 
\emph{i.e.} $Ch(E;\nabla)$ is a closed element of the space $\Omega(LM)^{inv(\vf)}$ with differential $d+\iota_\vf$.
\end{cor}

Just as in section \ref{SUBSEC:2-higher-hol}, we can also define $Ch^{(u)}(E; \nabla) = \sum_{k\geq 0} u^{-k}\cdot tr(hol_{2k})\in \Omega(LM)[u,u^{-1}]]^{inv(\vf)}$, which is closed under the differential $d+u\cdot \iota_\vf$. However, in foresight of the Chern character that will be appearing in the gerbe case in the next section, it will be more useful to focus on $Ch(E;\nabla)$, see also remarks \ref{REM:no-u's} and \ref{REM:no-u's-compat} below.

\section{Equivariant holonomy for abelian gerbes}\label{SEC:Chern-gerbes}
In this section, we work with abelian gerbes with connnection. We develop and apply the local approach from the previous section to the torus mapping space. Starting with holonomy of the gerbe, we define a sequence of higher holonomies 
$hol_{2k,2\ell}\in \Omega^{2k+2\ell}(M^\T)$ that are related via the diagram in figure \ref{DIAG:hol_k,l}, with $hol_{0,0}$ being the usual holonomy. Together these will give a equivariantly closed form on the torus mapping space, see corollary \ref{gerbechern}.

\subsection{Iterated integral for the torus in the $(p,q)$-simplicial model}\label{SUBSEC:torus-hol}

Let $\T$ denote the torus $\T=S^1 \times S^1$, which we will frequently view as a quotient of $[0,1]\times [0,1]$, and let $M^\T=Map(\T,M)$ be the space of smooth maps from the torus to a fixed manifold $M$. Recall that $M^\T$ is a smooth Fr\'echet manifold \cite{Ha} on which  we can perform the usual constructions of differential topology (vector fields, differential forms, exterior derivative, etc.).
We will construct forms on $M^\T$ by first defining them locally on open sets $\NN(p,q,\U)\subset M^\T$. 

Fix numbers $p,q\in \N$, and open sets $U_{i_{(k,\ell)}}$ of $M$ for $k=1, \dots, p$ and $\ell=1, \dots, q$, which we abbreviate by $\U=\{U_{i_{(k,\ell)}}\}_{1\leq k\leq p,1\leq \ell\leq q}$. Define the open subset $\NN(p,q,\U)\subset M^\T$ by,
\[
 \NN(p,q,\U)= \left\{\gamma\in M^\T: \left(\gamma\Big|_{\big[\frac{k-1}{p},\frac{k}{p}\big]\times \big[\frac{\ell-1}{q},\frac{\ell}{q}\big]}\right)\subset U_{i_{(k,\ell)}}, \text{ for }  1\leq k\leq p, 1\leq\ell\leq q \right\}.
 \]
Intuitively, we divide the torus into a $p\times q$ grid, and consider those maps for which the rectangular component labelled by $(k,\ell)$ has image in the open subset $U_{i_{(k,\ell)}}$. We denote the vertices, horizontal and vertical edges of this subdivision as 
in figure \ref{EQ:torus-v-e-f}. Here, the $(k,q)^{\text{th}}$ face is glued to the $(k,1)^{\text{st}}$ vertical edge, and similarly the $(1,\ell)^{\text{th}}$ face is glued to the $(p,\ell)^{\text{th}}$ horizontal edge to obtain a suitable cubical subdivision of the torus.

\begin{figure}
\[
\xymatrix@=10pt{
\ou{v}{1,1} \ar[rr]_{\ou {\ehor}{1,1}}\ar[dd]^{\ou{\ever}{1,1}} &&  \ou{v}{1,2} \ar[dd]^{\ou{\ever}{1,2}}\ar[rr]_{\ou{\ehor}{1,2}} && \ou{v}{1,3} \ar[rr]_{\ou{\ehor}{1,3}}\ar[dd]^{\ou{\ever}{1,3}} &&\dots & \ou{v}{1,q} \ar[dd]^{\ou{\ever}{1,q}}\ar[rr]_{\ou{\ehor}{1,q}}&& \\
&\ou{f}{1,1}&& \ou{f}{1,2}&& &&& \ou{f}{1,q}  \\
\ou{v}{2,1} \ar[rr]_{\ou {\ehor}{2,1}}\ar[dd]^{\ou{\ever}{2,1}}  &&   \ou{v}{2,2} \ar[dd]^{\ou{\ever}{2,2}} \ar[rr]_{\ou{\ehor}{2,2}} && \ou{v}{2,3}  \ar[rr]_{\ou{\ehor}{2,3}} \ar[dd]^{\ou{\ever}{2,3}}  &&\dots  & \ou{v}{2,q} \ar[dd]^{\ou{\ever}{2,q}} \ar[rr]_{\ou{\ehor}{2,q}}&& \\
&\ou{f}{2,1}&& \ou{f}{2,2}&& &&& \ou{f}{2,q}  \\
\ou{v}{3,1} \ar[rr]_{\ou {\ehor}{3,1}}\ar[dd]^{\ou{\ever}{3,1}}  &&   \ou{v}{3,2} \ar[dd]^{\ou{\ever}{3,2}} \ar[rr]_{\ou{\ehor}{3,2}} && \ou{v}{3,3}  \ar[rr]_{\ou{\ehor}{3,3}} \ar[dd]^{\ou{\ever}{3,3}}  &&\dots  & \ou{v}{3,q} \ar[dd]^{\ou{\ever}{3,q}} \ar[rr]_{\ou{\ehor}{3,q}}&& \\
\\ \vdots  &&\vdots  && \vdots  && \ddots&\vdots \\
\ou{v}{p,1} \ar[rr]_{\ou {\ehor}{p,1}}\ar[dd]^{\ou{\ever}{p,1}} &&   \ou{v}{p,2} \ar[dd]^{\ou{\ever}{p,2}}\ar[rr]_{\ou{\ehor}{p,2}} && \ou{v}{p,3}  \ar[rr]_{\ou{\ehor}{p,3}} \ar[dd]^{\ou{\ever}{p,3}} &&\dots  & \ou{v}{p,q} \ar[dd]^{\ou{\ever}{p,q}}\ar[rr]_{\ou{\ehor}{p,q}}&& \\
& \ou{f}{p,1} && \ou{f}{p,2} && &&& \ou{f}{p,q} \\   &&   &&  &&&  &
}
\]
\caption{Vertices, edges, and faces of the torus (The orientation of the edges is as indicated, and faces are oriented by their inclusion in the plane.)}\label{EQ:torus-v-e-f}
\end{figure}
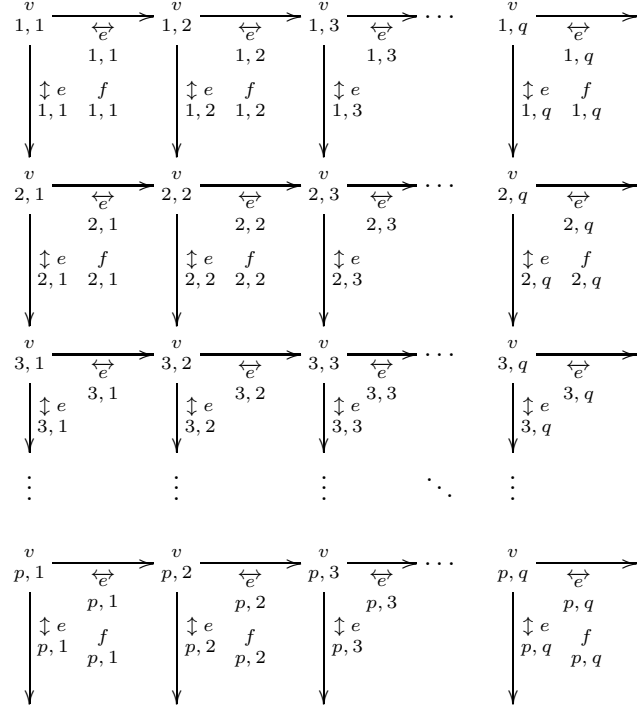

For given $p,q\in \N$, and a choice of open sets $\U$ as above, we now introduce some notation for the various open sets used in the torus subdivision. 
First, denote by $U\ou{f}{k,\ell}:=U_{i_{(k,\ell)}}$ the open set to which the $(k,\ell)^{\text{th}}$ face $\ou{f}{k,\ell}$ will be mapped to for a $\gamma\in \NN(p,q,\U)$. 
We denote the intersections of the open sets along a horizontal edge by $U^\ehor_{k,\ell}$, respectively along a vertical edge by $U^\ever_{k,\ell}$,  \emph{i.e.}
\begin{eqnarray*} 
U^\ehor_{1,\ell}=U_{i_{(1,\ell)}}\cap U_{i_{(p,\ell)}},&& \text{ and for } k>1:\quad  U^\ehor_{k,\ell}=U_{i_{(k-1,\ell)}}\cap U_{i_{(k,\ell)}},\\
 U^\ever_{k,1}=U_{i_{(k,q)}}\cap U_{i_{(k,1)}}, &&\text{ and for } \ell >1:\quad  U^\ever_{k,\ell}=U_{i_{(k,\ell-1)}}\cap U_{i_{(k,\ell)}},
 \end{eqnarray*}
 and the intersection of open sets at a vertex by $U^v_{k,\ell}$, \emph{i.e.}
 \begin{eqnarray*}
&&U^v_{1,1}=U_{i_{(1,1)}}\cap U_{i_{(p,1)}}\cap U_{i_{(1,q)}}\cap U_{i_{(p,q)}},  \\
\text{for } k>1:&& U^v_{k,1}=U_{i_{(k,1)}}\cap U_{i_{(k,q)}}\cap U_{i_{(k-1,1)}}\cap U_{i_{(k-1,q)}},  \\
\text{for } \ell>1:&& U^v_{1,\ell}=U_{i_{(1,\ell)}}\cap U_{i_{(1,\ell-1)}}\cap U_{i_{(p,\ell)}}\cap U_{i_{(q,\ell-1)}},  \\
\text{for } k,\ell>1:&&U^v_{k,\ell}=U_{i_{(k,\ell)}}\cap U_{i_{(k,\ell-1)}}\cap U_{i_{(k-1,\ell)}}\cap U_{i_{(k-1,\ell-1)}}.
\end{eqnarray*}
Differential forms on these open sets are denoted by
$$ \Oma^f_{k,\ell}=\Om(U\ou{f}{k,\ell}), \quad \Oma^\ehor_{k,\ell}=\Om(U^\ehor_{k,\ell}),\quad  \Oma^\ever_{k,\ell}=\Om(U^\ever_{k,\ell}), \quad \Oma^v_{k,\ell}=\Om(U^v_{k,\ell}). $$
We refer to these as \emph{forms on the faces, edges, and vertices}.
We now define the torus-Hochschild chain complex and the iterated integral subject to the data $\U$.
\begin{defn}\label{DEF:Torus-CH}
Fix $p,q\in \N$, and open sets $\U=\{U_{i_{(k,\ell)}}\}_{1\leq k\leq p,1\leq \ell\leq q}$ as above. For integers $n_1,\dots,n_q,m_1,\dots,m_p\geq 0$, we first define $\Oma^{n_1,\dots, n_q}_{m_1,\dots, m_p}$ to be given by the tensor product of differential forms defined in the following matrix of tensor products,
\[
\begin{tabular}{|cc|cc|cc}\hline
$\Oma^v_{1,1}$& $\ot (\Oma^\ehor_{1,1})^{\ot n_1} $& $\ot \Oma^v_{1,2}$&$ \ot (\Oma^\ehor_{1,2})^{\ot n_2}$ & $\dots $&$ \ot (\Oma^\ehor_{1,q})^{\ot n_q}$\\
$\ot(\Oma^\ever_{1,1})^{\ot m_1}$&$ \ot (\Oma^f_{1,1})^{\ot m_1  n_1} $& $\ot (\Oma^\ever_{1,2})^{\ot m_1}$&$ \ot (\Oma^f_{1,2})^{\ot m_1 n_2}$ & $\dots$ &$ \ot (\Oma^f_{1,q})^{\ot m_1 n_q}$\\ \hline
$\ot\Oma^v_{2,1}$&$ \ot (\Oma^\ehor_{2,1})^{\ot n_1} $&$  \ot \Oma^v_{2,2}$&$ \ot (\Oma^\ehor_{2,2})^{\ot n_2} $&$ \dots $&$ \ot (\Oma^\ehor_{2,q})^{\ot n_q} $\\
$\ot(\Oma^\ever_{2,1})^{\ot m_2}$&$ \ot (\Oma^f_{2,1})^{\ot m_2 n_1} $& $\ot (\Oma^\ever_{2,2})^{\ot m_2}$& $\ot (\Oma^f_{2,2})^{\ot m_2 n_2} $&$ \dots $&$ \ot (\Oma^f_{2,q})^{\ot m_2 n_q}$\\ \hline
$\vdots$&$\vdots$&$\vdots$&$\vdots$&$\ddots$&$\vdots$\\
$\ot(\Oma^\ever_{p,1})^{\ot m_p}$&$ \ot (\Oma^f_{p,1})^{\ot m_p n_1} $& $\ot (\Oma^\ever_{p,2})^{\ot m_p}$& $\ot (\Oma^f_{p,2})^{\ot m_p n_2} $&$ \dots $&$ \ot (\Oma^f_{p,q})^{\ot m_p n_q}$
\end{tabular}
\]
If we write $m=m_1+\dots+m_p$ and $n=n_1+\dots+n_q$ for short, then this matrix has $m+p$ rows and $n+q$ columns. Note, that the algebras of forms on the edges, $\Oma^\ehor_{k,\ell}$ and $\Oma^\ever_{k,\ell}$, are modules over the algebras of forms on the adjacent faces $\Oma^f_{k',\ell'}$ via the inclusion of open sets. Similarly the algebras of forms on the vertices $\Oma^v_{k,\ell}$ are modules over the algebras of forms on the adjacent edges and faces.

We would like to remark, that in the above tensor product for $\Oma^{n_1,\dots, n_q}_{m_1,\dots, m_p}$ no particular linear order was chosen to write this tensor product, but it is rather a coequalizer over all possible linear orderings. For example, a monomial $\mathfrak a \in \Oma^{n_1,\dots, n_q}_{m_1,\dots, m_p}$ may given explicitly by a choice of linear ordering as follows,
\begin{equation}\label{EQ:a-in-CH^T}
\mathfrak a= \bigotimes_{1\leq k\leq p} \bigotimes_{1\leq \ell\leq q} 
\left(a^v_{k,\ell}\otimes 
\bigotimes_{1\leq i\leq m_k}(a^{\ever}_{k,\ell})_i \ot 
\bigotimes_{1\leq i\leq n_\ell}(a^{\ehor}_{k,\ell})_i \ot 
\bigotimes_{ 1\leq i\leq m_k}\bigotimes_{ 1\leq j\leq n_\ell } (a^f_{k,\ell})_{i,j}\right).
\end{equation}
To identify $\mathfrak a$ for another choice of linear ordering, a sign factor coming from the Koszul sign rule has to be applied.

 With this notation, the Hochschild complex subject to $\U$ is defined as $$CH_\bullet^{(p,q,\U)}:=\bigoplus_{\scriptsize\begin{matrix}n_1,\dots,n_q\geq 0 \\ m_1,\dots,m_p\geq 0 \end{matrix}} \Oma^{n_1,\dots, n_q}_{m_1,\dots, m_p}[m+n],$$ where the square bracket ``$[m+n]$'' denotes the operation of a shift of total degree down by $n+m$. For an element $\mathfrak a \in \Oma^{n_1,\dots, n_q}_{m_1,\dots, m_p}[m+n]\subset CH_\bullet^{(p,q,\U)}$ as in equation \eqref{EQ:a-in-CH^T} we will call $m+n$ the simplicial degree of an element in $CH_\bullet^{(p,q,\U)}$, and denote by $|\mathfrak a|=\sum_k \sum_\ell
\left(|a^v_{k,\ell}|+ \sum_i |(a^{\ever}_{k,\ell})_i| + \sum_i |(a^{\ehor}_{k,\ell})_i | +
\sum_i \sum_j |(a^f_{k,\ell})_{i,j}|\right)
-(m+n)$ the total degree of $\mathfrak a$.

To obtain a differential, we first define the face maps $d^{\leftrightarrow}_{r,i}:\Oma^{n_1,\dots,n_r,\dots, n_q}_{m_1,\dots,m_p}\to \Oma^{n_1,\dots,n_r-1,\dots, n_q}_{m_1,\dots,m_p}$ (for $r=1,\dots,q$ and $i=0,\dots,n_r$), and $d^{\updownarrow}_{r,i}:\Oma^{n_1,\dots, n_q}_{m_1,\dots,m_r,\dots,m_p}\to \Oma^{n_1,\dots, n_q}_{m_1,\dots,m_r-1,\dots,m_p}$ (for $r=1,\dots,p$ and $i=0,\dots,m_r$). The map $d^{\leftrightarrow}_{r,i}$ is defined by multiplying the $(r+n_1+\dots+n_{r-1}+i)^{\text{th}}$ and the $(r+n_1+\dots+n_{r-1}+i+1)^{\text{th}}$ column,
$$
\begin{bmatrix} (a^\ehor_{1,r})_i\\ \vdots\\ (a^f_{p,r})_{m_p, i} \end{bmatrix} \ot 
\begin{bmatrix} (a^\ehor_{1,r})_{i+1}\\ \vdots\\ (a^f_{p,r})_{m_p, i+1} \end{bmatrix} \mapsto
\begin{bmatrix} (a^\ehor_{1,r})_i\cdot (a^\ehor_{1,r})_{i+1}\\ \vdots\\ (a^f_{p,r})_{m_p, i} \cdot (a^f_{p,r})_{m_p, i+1} \end{bmatrix} 
$$
and leaving the other tensor factors unchanged. Note, that taking products respects the coequalizer over linear orderings, and thus defines a map $d^{\leftrightarrow}_{r,i}:\Oma^{n_1,\dots,n_r,\dots, n_q}_{m_1,\dots,m_p}\to \Oma^{n_1,\dots,n_r-1,\dots, n_q}_{m_1,\dots,m_p}$. 
These maps are defined in a cyclic way, meaning that in the case 
$r=q$ and $i=n_q$ the operation $d^{\leftrightarrow}_{q,n_q}:\Oma^{n_1,\dots, n_q}_{m_1,\dots,m_p}\to \Oma^{n_1,\dots, n_{q-1}}_{m_1,\dots,m_p}$ multiplies the last column to the first column from the left.
  
Similarly, $d^{\updownarrow}_{r,i}$ multiplies the $(r+m_1+\dots+m_{r-1}+i)^{\text{th}}$ and the $(r+m_1+\dots+m_{r-1}+i+1)^{\text{th}}$ row,
\[
\begin{matrix}
[ (a^\ever_{r,1})_i \, \dots \, (a^f_{r,q})_{i,n_q} ] & & \\
 \ot & \mapsto & [ ((a^\ever_{r,1})_i\cdot (a^\ever_{r,1})_{i+1}) \dots((a^f_{r,q})_{i,n_q}\cdot (a^f_{r,q})_{i+1,n_q})] \\
[ (a^\ever_{r,1})_{i+1} \dots(a^f_{r,q})_{i+1,n_q} ] & & 
\end{matrix}
\]
and leaving the other tensor factors unchanged. Again, this is done in a cyclic way so that the last operator multiplies the
last row and the first row. We note that the indices for $d^{\leftrightarrow}_{r,i}$ (and $d^{\updownarrow}_{r,i}$) are chosen in such a way, that two edge columns (respectively two edge rows) are never being multiplied.

With this, we can finally define the differential $D$ on $CH_\bullet^{(p,q,\U)}$ to be given on $\Oma^{n_1,\dots, n_q}_{m_1,\dots,m_p}$ by
\begin{multline}\label{EQ:D-torus-Hoch}
 D:=(-1)^{m+n} \cdot d_{DR}
 \\
\quad \quad \,\,\,\,+ (-1)^{m+n+1}\cdot\sum_{\scriptsize\begin{matrix}j=1,\dots m+p\text{, where}\\j=(1+m_1)+\dots+(1+m_{r-1})+i+1 \end{matrix}} (-1)^{j+1} \cdot d^{\updownarrow}_{r,i}
\\
 +(-1)^{m+n+1}\cdot\sum_{\scriptsize\begin{matrix}j=1,\dots n+q\text{, where}\\j=(1+n_1)+\dots+(1+n_{r-1})+i +1\end{matrix}} (-1)^{m+p+j+1} \cdot d^{\leftrightarrow}_{r,i}.
\end{multline}
If we rewrite the matrix in a linear order of tensor factors and using signs by the usual Koszul rule, a lengthy but straightforward calculation, using that $(d_{DR})^2=(d')^2=0$ and $d'\circ d_{DR}=d_{DR}\circ d'$ for $d'=\sum (-1)^{j+1} d^{\updownarrow}_{r,i}+\sum(-1)^{m+p+j+1} d^{\leftrightarrow}_{r,i}$,  shows $D^2=0$. We denote the homology of this complex by $HH_\bullet^{(p,q,\U)}:=H_\bullet(CH_\bullet^{(p,q,\U)},D)$.
\end{defn}

\begin{defn}
We also have a shuffle product, which is defined using the degeneracy maps $s^{\leftrightarrow}_{r,i}:\Oma^{n_1,\dots,n_r,\dots, n_q}_{m_1,\dots,m_p}\to \Oma^{n_1,\dots,n_r+1,\dots, n_q}_{m_1,\dots,m_p}$ (for $r=1,\dots,p$ and $i=0,\dots,n_r$), and $s^{\updownarrow}_{r,i}:\Oma^{n_1,\dots, n_q}_{m_1,\dots,m_r,\dots,m_p}\to \Oma^{n_1,\dots, n_q}_{m_1,\dots,m_r+1,\dots,m_p}$ (for $r=1,\dots,q$ and $i=0,\dots,m_r$). The map $s^{\leftrightarrow}_{r,i}$ is defined by adding a column of units $1\in\Om(U)$ (for the respective open sets $U$) between the $(r+n_1+\dots+n_{r-1}+i)^{\text{th}}$ and the $(r+n_1+\dots+n_{r-1}+i+1)^{\text{th}}$ column. Similarly, $s^{\updownarrow}_{r,i}$ adds a row of units ``$1$'' between the $(r+m_1+\dots+m_{r-1}+i)^{\text{th}}$ and the $(r+m_1+\dots+m_{r-1}+i+1)^{\text{th}}$ row. The operations are well-defined on the coequalizer due the units $1$ being of degree $0$. With this, we set,
\begin{multline*}
\bullet^\leftrightarrow_r (\mathfrak a \ot\mathfrak b) = 
\sum_{ \sigma\in S(n_r,n'_r)} \sgn(\sigma)\cdot 
 \left(s^\leftrightarrow_{r,\sigma(n_r+n'_r)}\dots  s^\leftrightarrow_{r,\sigma(n_r+1)}(\mathfrak a)\right)
 \otimes \left(s^\leftrightarrow_{r,\sigma(n_1)} \dots  s^\leftrightarrow_{r,\sigma(1)}(\mathfrak b)\right)\\
 \in\Oma^{n_1,\dots,n_r+n'_r,\dots, n_q}_{m_1,\dots, m_p}\ot \Oma^{n'_1,\dots,n_r+n'_r,\dots, n'_q}_{m'_1,\dots, m'_p},
\end{multline*}
\begin{multline*}
\bullet^\updownarrow_r (\mathfrak a \ot\mathfrak b) = 
\sum_{ \sigma\in S(m_r,m'_r)} \sgn(\sigma)\cdot 
 \left(s^\updownarrow_{r,\sigma(m_r+m'_r)}\dots  s^\updownarrow_{r,\sigma(m_r+1)}(\mathfrak a)\right)
 \otimes \left(s^\updownarrow_{r,\sigma(m_1)} \dots  s^\updownarrow_{r,\sigma(1)}(\mathfrak b)\right)\\
 \in\Oma^{n_1,\dots, n_q}_{m_1,\dots,m_r+m'_r,\dots, m_p}\ot \Oma^{n'_1,\dots, n'_q}_{m'_1,\dots,m_r+m'_r,\dots, m'_p},
\end{multline*}
Note, that the two operators $\bullet^ \leftrightarrow_r $ and $\bullet^\updownarrow_{r'}$ commute for any $r$ and $r'$. With this and given $\mathfrak a\ot \mathfrak b\in \Oma^{n_1,\dots, n_q}_{m_1,\dots, m_p}\ot \Oma^{n'_1,\dots, n'_q}_{m'_1,\dots, m'_p}$, the shuffle product $\bullet:CH_\bullet^{(p,q,\U)}\ot CH_\bullet^{(p,q,\U)}\to CH_\bullet^{(p,q,\U)}$ is defined by the composition  $\bullet^\updownarrow_p\circ\dots \circ\bullet^\updownarrow_1\circ\bullet^\leftrightarrow_q\circ\dots \circ\bullet^\leftrightarrow_1$ 
and a multiplication $\star$ by tensor factors,
\begin{multline*}
\Oma^{n_1,\dots, n_q}_{m_1,\dots, m_p}\ot \Oma^{n'_1,\dots, n'_q}_{m'_1,\dots, m'_p}\quad  \stackrel {\bullet^\updownarrow_p\dots \bullet^\updownarrow_1\bullet^\leftrightarrow_q\dots \bullet^\leftrightarrow_1}{\longrightarrow}\\
\Oma^{n_1+n'_1,\dots, n_q+n'_q}_{m_1+m'_1,\dots, m_p+m'_p} \ot \Oma^{n_1+n'_1,\dots, n_q+n'_q}_{m_1+m'_1,\dots, m_p+m'_p} 
\quad \stackrel \star \longrightarrow\quad \Oma^{n_1+n'_1,\dots, n_q+n'_q}_{m_1+m'_1,\dots, m_p+m'_p},
\end{multline*}
\[
\mathfrak a\bullet \mathfrak b := (-1)^{|\mathfrak a|\cdot (m'+n')+n_q\cdot (m'+n'-n'_q)+\dots+m_2\cdot m'_1} \cdot \star \circ \bullet^\updownarrow_p\circ\dots \circ\bullet^\updownarrow_1\circ\bullet^\leftrightarrow_q\circ\dots \circ\bullet^\leftrightarrow_1(\mathfrak a\ot \mathfrak b),
\]
where $\star(\bigotimes_x a_x,\bigotimes_x b_x)=\bigotimes_x (a_x\cdot b_x)$ denotes the multiplication of the individual tensor factors.

The shuffle map is graded commutative and associative, and $D$ is a derivation of the shuffle map.
\end{defn}
Finally, we define the iterated integral map $It^{(p,q,\U)}:CH_\bullet^{(p,q,\U)}\to \Om(\NN(p,q,\U))$. Rather than defining it via an explicit formula as in the previous definitions \ref{DEF:It-Chapter2} and \ref{localIt}, it will now be more convenient to use the pullback description for our definition.
\begin{defn}\label{DEF:torus-It}
Subdivide the torus into a grid determined by $(\Delta^{m_1}\times\dots \times \Delta^{m_p})\times(\Delta^{n_1}\times\dots \times \Delta^{n_q})$ as in figure \ref{FIG:torus-subdiv}. 

\begin{figure}  
\begin{pspicture}(-1,0)(9.5,7.4)
 \pspolygon(1,1)(1,7)(9,7)(9,1) \psline(1,5)(9,5) \psline(1,3)(9,3)
 \psline(3,1)(3,7) \psline(5,1)(5,7)  \psline(7,1)(7,7) 
 \rput[t](2,.5){$\underbrace{\quad\quad\quad\quad\quad}_{\Delta^{n_1}}$} 
 \rput[t](4,.5){$\underbrace{\quad\quad\quad\quad\quad}_{\Delta^{n_2}}$} 
 \rput(6,.4){$\dots$}
 \rput[t](8,.5){$\underbrace{\quad\quad\quad\quad\quad}_{\Delta^{n_q}}$} 
 \rput(0,6){$_{\Delta^{m_1}}\left\{\begin{matrix} \\ \\ \\ \\ \end{matrix}\right.$} 
 \rput(.3,4){$\vdots$}
 \rput(0,2){$_{\Delta^{m_p}}\left\{\begin{matrix} \\ \\ \\ \\ \end{matrix}\right.$} 
 \psline(1,6.8)(.8,6.8)  \psline(1,6.7)(.8,6.7)  \psline(1,6.3)(.8,6.3)  \psline(1,6)(.8,6) 
 \psline(1,5.8)(.8,5.8)  \psline(1,5.3)(.8,5.3)  
 \psline(1,2.7)(.8,2.7)  \psline(1,2.3)(.8,2.3)  \psline(1,1.8)(.8,1.8)  \psline(1,1.5)(.8,1.5) 
 \psline(1,1.3)(.8,1.3) 
 \psline(1.3,1)(1.3,.8)  \psline(2,1)(2,.8)  \psline(2.2,1)(2.2,.8)  \psline(2.8,1)(2.8,.8)
 \psline(3.1,1)(3.1,.8)  \psline(3.2,1)(3.2,.8)  \psline(3.7,1)(3.7,.8)  \psline(4,1)(4,.8)
 \psline(4.3,1)(4.3,.8)  \psline(4.5,1)(4.5,.8)  \psline(4.8,1)(4.8,.8) 
 \psline(7.5,1)(7.5,.8)  \psline(7.8,1)(7.8,.8)  \psline(8.4,1)(8.4,.8)  \psline(8.6,1)(8.6,.8)
 \psline(8.7,1)(8.7,.8)
 \rput(.6,6){$_{t^1_i}$}  \rput(4.5,.6){$_{u^2_j}$} \rput(4.5,6){$\times$}
 \psline[linestyle=dotted](1,6)(4.5,6)  \psline[linestyle=dotted](4.5,1)(4.5,6)
 \rput[l](4.75,5.8){$\left(\frac {t^1_i} p, \frac {1+u^2_j} q\right)$}
\end{pspicture} 
\caption{Coordinates of $m\cdot n$ points on the torus determined by $(\Delta^{m_1}\times\dots \times \Delta^{m_p})\times(\Delta^{n_1}\times\dots \times \Delta^{n_q})$}\label{FIG:torus-subdiv}
\end{figure}
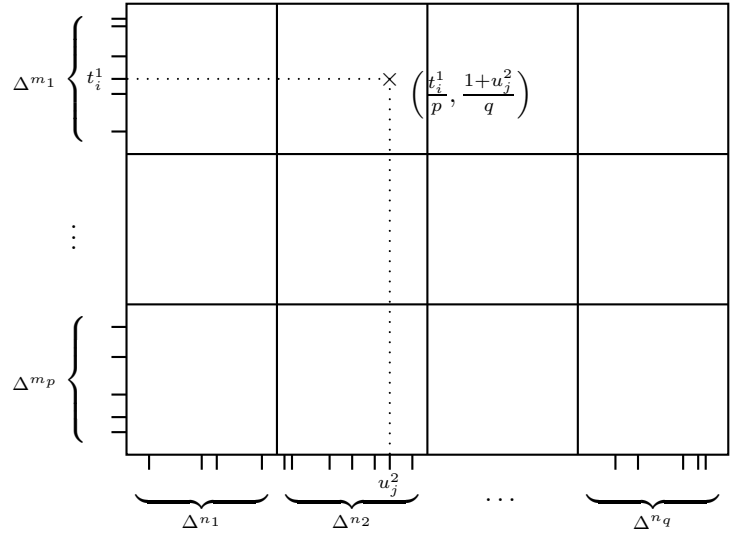

Denote by $m=m_1+\dots+m_p $, and $n=n_1+\dots+n_q $. We have an evaluation map $ev$,
\begin{multline*}
\NN(p,q,\U)\times (\Delta^{m_1}\times\dots \times \Delta^{m_p})\times(\Delta^{n_1}\times\dots \times \Delta^{n_q}) \stackrel {ev} \longrightarrow  M^{(m+p) \times (n+q)},\\
\Big(\gamma, (0=t^k_0\leq t^k_1\leq \dots\leq t^k_{m_k}\leq 1)_{k=1,\dots, p},(0=u^\ell_0\leq u^\ell_1\leq \dots\leq u^\ell_{n_\ell}\leq 1)_{\ell=1,\dots,q}\Big)
\\ \mapsto \left(\gamma\left(\frac {k-1+t^k_i} p, \frac {\ell-1+u^\ell_j} q\right)\right)_{\scriptsize
\begin{matrix}
1\leq k\leq p,\quad 0\leq i\leq m_k\\  1\leq \ell\leq q,\quad 0\leq j\leq n_\ell
\end{matrix}}.
\end{multline*}
Note, that starting from $\gamma\in \NN(p,q,\U)$, the evaluation map lands in the corresponding open set $U_{i_{(k,\ell)}}$, or intersection of open sets for edges or vertices, \emph{e.g.}
$$ \gamma\left(\frac {k-1} p,\frac {\ell-1} q\right)\in U^v_{k,\ell}=U_{i_{(k,\ell)}}\cap U_{i_{(k,\ell-1)}}\cap U_{i_{(k-1,\ell)}}\cap U_{i_{(k-1,\ell-1)}}. $$
Thus, we can define the iterated integral $It:CH_\bullet^{(p,q,\U)}\to \Omega(\NN(p,q,\U))$ to be the pullback along $ev$ 
composed with integration along the fiber,
\begin{equation*}
\xymatrix{
  \NN(p,q,\U)\times (\Delta^{m_1}\times\dots \times \Delta^{m_p})\times(\Delta^{n_1}\times\dots \times \Delta^{n_q}) \ar[r]^{\quad\quad \quad\quad \quad\quad \quad\quad ev} \ar[d]_{\int_{(\Delta^{m_1}\times\dots \times \Delta^{m_p})\times(\Delta^{n_1}\times\dots \times \Delta^{n_q})}} & M^{(m+p)\times (n+q)} \\ \NN(p,q,\U)& }
\end{equation*}
\begin{equation*}
 \Ch^{(p,q,\U)}(\mathfrak a):= \int_{(\Delta^{m_1}\times\dots \times \Delta^{m_p})\times(\Delta^{n_1}\times\dots \times \Delta^{n_q})} ev^*(\mathfrak a),
\quad\quad
\text{for }\mathfrak a \in \Oma^{n_1,\dots, n_q}_{m_1,\dots, m_p}.
\end{equation*}
In order to exhibit the above definition as an explicit iterated integral over the tensor factors of $\mathfrak a$, we note that for a $\T$-space $Y$, the pullback of a form $\omega$ along $\T \times Y \to Y$ is given by 
\begin{equation}\label{EQ:T-pullback}
\omega(t,u) + dt\wedge \iota_{\vf}\omega(t,u) + du\wedge \iota_{\wf} \omega(t,u)+ dtdu\wedge \iota_{\wf}\iota_{\vf}\omega(t,u),
\end{equation}
where $\vf$ and $\wf$ are the canonical vector fields on $Y$ coming from the two circle actions of the first and second factors of $\T$, respectively, and $\omega(t,u)$ is the pullback of $\omega$ at $(t,u) \in \T$. Thus, when taking the integral over the variables $\frac{k-1}{p}\leq t^k_1\leq \dots\leq t^k_{m_k}\leq \frac{k}{p}\in \Delta^{m_k}$ and $\frac{\ell-1}{q}\leq u^\ell_1\leq \dots\leq u^\ell_{m_\ell}\leq \frac{\ell}{q}\in \Delta^{n_\ell}$, and an element $\mathfrak a$ as in \eqref{EQ:a-in-CH^T}, the iterated integral becomes,
\begin{multline}\label{EQ:It-torus-explicit}
It^{(p,q,\U)}(\mathfrak a)=\pm\int_{(\Delta^{m_1}\times\dots \times \Delta^{m_p})\times(\Delta^{n_1}\times\dots \times \Delta^{n_q})} \bigwedge_{k=1}^p \bigwedge_{\ell=1}^q \Bigg[ a^v_{k,\ell}\left(\frac{k-1}{p},\frac{\ell-1}{q}\right)
\\
\wedge \bigwedge_{i=1}^{m_k} \left((a^{\ever}_{k,\ell})_i\left(t^k_i,\frac{\ell-1} q\right)+dt^k_i\wedge \iota_\vf (a^{\ever}_{k,\ell})_i\left(t^k_i,\frac{\ell-1} q\right) \right)
\\
\wedge \bigwedge_{j=1}^{n_\ell} \left((a^{\ehor}_{k,\ell})_j\left(\frac{k-1} p,u^\ell_j\right)+du^\ell_j \wedge \iota_\wf (a^{\ehor}_{k,\ell})_j\left(\frac{k-1} p,u^\ell_j\right) \right)
\\
\wedge\bigwedge_{i=1}^{m_k}\bigwedge_{j=1}^{n_\ell}\bigg(
(a^f_{k,\ell})_{i,j}(t^k_i,u^\ell_j) +dt^k_i \wedge \iota_{\vf}(a^f_{k,\ell})_{i,j}(t^k_i,u^\ell_j)
\\
+ du^\ell_j\wedge \iota_{\wf} (a^f_{k,\ell})_{i,j}(t^k_i,u^\ell_j)+dt^k_i du^\ell_j\wedge \iota_{\wf}\iota_{\vf}(a^f_{k,\ell})_{i,j}(t^k_i,u^\ell_j) 
\bigg)\Bigg].
\end{multline}
Thus, for any of the variables $t^k_i$ (or $u^\ell_j$), there is not just one unique form that depends on this variable, but a whole row (or column) of forms that depend on the variable, and integration over this variable occurs for exactly one differential $dt^k_i$ (respectively $du^\ell_j$).
We will show below, that in specific cases (such as the holonomy function), it is nevertheless possible to simplify the iterated integral and compute it.
\end{defn}

\begin{prop} \label{LEM:It(q,p,U)-d-sh}
$It^{(p,q,\U)}$ commutes with differentials and products, \emph{i.e.},
\[
It^{(p,q,\U)}(D(\mathfrak a))=d_{DR}(It^{(p,q,\U)}(\mathfrak a)), \text{ and } It^{(p,q,\U)}(\mathfrak a\bullet\mathfrak b)=It^{(p,q,\U)}(\mathfrak a)\wedge It^{(p,q,\U)}(\mathfrak b).
\]
\end{prop}
\begin{proof}
The proof is similar to \cite[lemma 2.2.2, proposition 2.4.6]{GTZ} and proposition \ref{LEM:It-diff-prod} below (compare this also with \cite[proposition 1.6]{GJP} and proposition \ref{dItpU} above).  We start with comparing the differentials. For this proof we use the integration along a fiber formula $d_{DR}\left(\int_X \omega \right)=(-1)^{dim(X)} \left(\int_{X} d_{DR}(\omega) - \int_{\partial X} \omega\right)$ for $X=(\Delta^{m_1}\times\dots \times \Delta^{m_p})\times(\Delta^{n_1}\times\dots \times \Delta^{n_q})$. With this, we have for $\mathfrak a\in \Oma^{n_1,\dots, n_q}_{m_1,\dots, m_p}$, 
\begin{eqnarray*}
d_{DR}(It^{(p,q,\U)}(\mathfrak a))&=& d_{DR}\left(\int_{(\Delta^{m_1}\times\dots \times \Delta^{m_p})\times(\Delta^{n_1}\times\dots \times \Delta^{n_q})} ev^*(\mathfrak a)\right)\\
&=&(-1)^{m+n}\cdot\Bigg(\int_{(\Delta^{m_1}\times\dots \times \Delta^{m_p})\times(\Delta^{n_1}\times\dots \times \Delta^{n_q})}ev^* (d_{DR}\mathfrak a)\\ 
&&\quad\quad -\int_{\partial((\Delta^{m_1}\times\dots \times \Delta^{m_p})\times(\Delta^{n_1}\times\dots \times \Delta^{n_q}))} ev^* (\mathfrak a)\Bigg)\\
&=&(-1)^{m+n} \cdot It^{(p,q,\U)}\big( d_{DR}(\mathfrak a)\big)\\
&& + \sum_{j=1}^{m+p} (-1)^{m+n+1+j+1} \cdot It^{(p,q,\U)}\big(d^{\updownarrow}_{r,i}(\mathfrak a)\big)\\
&& + \sum_{j=1}^{n+q} (-1)^{m+n+1+(m+p+j+1)} \cdot It^{(p,q,\U)}\big( d^{\leftrightarrow}_{r,i}(\mathfrak a)\big) \\
&=& It^{(p,q,\U)}(D(\mathfrak a)).
\end{eqnarray*}

Next, we compare the shuffle and wedge product.
We start with rewriting the wedge product of two iterated integrals,
\begin{multline*}
It^{(p,q,\U)}(\mathfrak a)\wedge It^{(p,q,\U)}(\mathfrak b)
\\
=\int_{(\Delta^{m_1}\times\dots \times \Delta^{m_p})\times(\Delta^{n_1}\times\dots \times \Delta^{n_q})} ev^*(\mathfrak a)\wedge \int_{(\Delta^{m'_1}\times\dots \times \Delta^{m'_p})\times(\Delta^{n'_1}\times\dots \times \Delta^{n'_q})} ev^*(\mathfrak b)
\\
=(-1)^{\epsilon}\int_{\scriptsize \begin{matrix}((\Delta^{m_1}\times\Delta^{m'_1})\times\dots \times (\Delta^{m_p}\times \Delta^{m'_p}))\\ \times((\Delta^{n_1}\times \Delta^{n'_1})\times\dots \times (\Delta^{n_q}\times \Delta^{n'_q}))\end{matrix}} (ev,ev)^*(\mathfrak a\wedge \mathfrak b)\quad\quad\quad\quad\quad\quad\quad\quad\quad\quad
\\
=(-1)^{\epsilon}\sum_{\scriptsize\begin{matrix}\sigma_1\in S(m_1,m'_1)\\ \dots \\ \sigma_p\in S(m_p,m'_p)\end{matrix}}\quad \sum_{\scriptsize\begin{matrix}\rho_1\in S(n_1,n'_1)\\ \dots \\ \rho_q\in S(n_q,n'_q)\end{matrix}} \quad \int_{\scriptsize \begin{matrix}\beta^{\sigma_1}(\Delta^{m_1+m'_1})\times \dots\times \beta^{\sigma_p}(\Delta^{m_p+m'_p})\\ \times\beta^{\rho_1}(\Delta^{n_1+n'_1})\times \dots\times \beta^{\rho_q}(\Delta^{n_q+n'_q})
\end{matrix}} (ev,ev)^* (\mathfrak a\wedge \mathfrak b),
\end{multline*}
where $\epsilon=|\mathfrak a|\cdot (m'+n')+n_q\cdot (m'+n'-n'_q)+\dots+m_2\cdot m'_1$, and for a $(k,\ell)$-shuffle $\sigma\in S(k,\ell)$, the map $\beta^\sigma:\Delta^{k+\ell}\to \Delta^k\times \Delta^\ell, (t_1\leq\dots\leq t_{k+\ell})\mapsto (t_{\sigma(1)}\leq \dots\leq t_{\sigma(k)},t_{\sigma(k+1)}\leq \dots\leq t_{\sigma(k+\ell)})$ is used to decompose $\Delta^k\times \Delta^\ell=\bigcup_{\sigma\in S(k,\ell)}\beta^\sigma(\Delta^{k+\ell})$. Now, we have a commutative diagram,
\begin{equation*}
\xymatrix{
{ \begin{matrix} \NN(p,q,\U)\quad\quad\\ \times (\Delta^{m_1+m'_1}\times\dots)\\ \times (\Delta^{n_1+n'_1}\times\dots)\end{matrix}}\ar[rr]^{id\times \beta^{\sigma_1}\times\dots\times \beta^{\sigma_p}\times \beta^{\rho_1}\times\dots\times \beta^{\rho_q}\quad\quad} \ar[d]_{ev} && {\begin{matrix} \NN(p,q,\U)\quad\quad\quad\quad\quad\quad\quad \\ \times ((\Delta^{m_1}\times \Delta^{m'_1})\times\dots)\\ \times ((\Delta^{n_1}\times\Delta^{n'_1})\times\dots)\end{matrix}} \ar[d]^{(ev,ev)}\\   
 M^{\scriptsize \begin{matrix}(m+p+m'+p')\\ \times (n+q+n'+q')\end{matrix}} \ar[r]^{diag\quad\quad } &  {\begin{matrix} M^{(m+p+m'+p')\times (n+q+n'+q')}\\ \times   M^{(m+p+m'+p')\times (n+q+n'+q')}\end{matrix}} \ar[r]^{\quad\quad \eta'_{\sigma,\rho}\times \eta''_{\sigma,\rho}} &  {\begin{matrix} M^{(m+p)\times (n+q)}\\ \times M^{(m'+p')\times (n'+q')} \end{matrix}}}
\end{equation*}
Here, $diag$ is the diagonal on $M^{(m+p+m'+p')\times (n+q+n'+q')}$ with $m\pri{}=m\pri{1} +\dots+m\pri{p}$ and $n\pri{}=n\pri{1}+\dots+n\pri{q}$, and $\eta'_{\sigma,\rho}:M^{(m+p+m'+p')\times (n+q+n'+q')}\to M^{(m+p)\times (n+q)}$ and $\eta''_{\sigma,\rho}:M^{(m+p+m'+p')\times (n+q+n'+q')}\to M^{(m'+p')\times (n'+q')}$ are defined as
\begin{eqnarray*}
\eta'_{\sigma,\rho}&=&\big(M{^\updownarrow_{1,\sigma_1(m_1+1)}}\circ \dots\circ M{^\updownarrow_{1,\sigma_1(m_1+m'_1)}}\big)\circ\dots\circ\big( M{^\updownarrow_{p,\sigma_p(m_p+1)}}\circ \dots\circ M{^\updownarrow_{p,\sigma_1(m_p+m'_p)}}\big)
\\
&& \circ \big(M{^\leftrightarrow_{1,\rho_1(n_1+1)}}\circ \dots\circ M{^\leftrightarrow_{1,\rho_1(n_1+n'_1)}}\big)\circ\dots\circ\big( M{^\leftrightarrow_{q,\rho_q(n_q+1)}}\circ \dots\circ M{^\leftrightarrow_{q,\rho_1(n_q+n'_q)}}\big),
\\
\eta''_{\sigma,\rho}&=&\big(M{^\updownarrow_{1,\sigma_1(1)}}\circ \dots\circ M{^ \updownarrow_{1,\sigma_1(m_1)}}\big)\circ\dots\circ\big( M{^\updownarrow_{p,\sigma_p(1)}}\circ \dots\circ M{^ \updownarrow_{p,\sigma_1(m_p)}}\big)
\\
&& \circ \big(M{^\leftrightarrow_{1,\rho_1(1)}}\circ \dots\circ M{^\leftrightarrow_{1,\rho_1(n_1)}}\big)\circ\dots\circ\big( M{^\leftrightarrow_{q,\rho_q(1)}}\circ \dots\circ M{^ \leftrightarrow_{q,\rho_1(n_q)}}\big),
\end{eqnarray*}
where $M^\updownarrow_{r,i}:M^{k\times \ell}\hookrightarrow M^{(k+1)\times \ell}$ misses the $(r+m_1+\dots+m_{r-1}+i)^{\text{th}}$ row, and, similarly, $M^\leftrightarrow_{r,i}:M^{k\times \ell}\hookrightarrow M^{k\times (\ell+1)}$ misses the $(r+n_1+\dots+n_{r-1}+i)^{\text{th}}$ column.
If we further set
\begin{align*}
 0= t_1\leq \dots \leq \frac{r}{p}=t_{r+\sum_{i=1}^{r-1}m_i+m'_i}\leq t_{r+1+\sum_{i=1}^{r-1}m_i+m'_i}  \leq\dots\leq t_{r+\sum_{i=1}^{r}m_i+m'_i}\\
 \quad\quad\quad\quad\quad\quad \leq \dots \leq t_{m+p+m'+p'}\leq 1 \in \Delta^{m_1+m'_1}\times \dots\times\Delta^{m_r+m'_r}\times \dots\times \Delta^{m_p+m'_p},\\
 \sigma':\{1,\dots,m+p\}\to \{1,\dots,m+p+m'+p'\}, \sigma'(r+\sum_{i=1}^{r-1}m_i)=r+\sum_{i=1}^{r-1}m_i+m'_i, \\
\quad\quad \quad\text{and for }1\leq j\leq m_r: \sigma'(r+j+\sum_{i=1}^{r-1}m_i)=r+\sigma_r(j)+\sum_{i=1}^{r-1}m_i+m'_i, \\
 \sigma'':\{1,\dots,m'+p'\}\to \{1,\dots,m+p+m'+p'\}, \sigma''(r+\sum_{i=1}^{r-1}m'_i)=r+\sum_{i=1}^{r-1}m_i+m'_i, \\
\quad\quad \quad\text{and for }1\leq j\leq m'_r: \sigma''(r+j+\sum_{i=1}^{r-1}m'_i)=r+\sigma_r(j)+\sum_{i=1}^{r-1}m_i+m'_i, 
\end{align*}
and similar definitions for $u_1\leq\dots\leq u_{n+q+n'+q'}, \rho'$, and $\rho''$, then we can check that the above diagram commutes:
\begin{multline*} 
(ev,ev)\circ (id\times \beta^{\sigma_1}\times\dots\times \beta^{\rho_q})(\gamma,t_1\leq\dots\leq t_{m+p+m'+p'},u_1\leq \dots\leq u_{n+q+n'+q'})
\\
=\Big(\gamma(t_{\sigma'(1)},u_{\rho'(1)}), \dots,\gamma(t_{\sigma'(m+p)},u_{\rho'(n+q)})
,\gamma(t_{\sigma''(1)},u_{\rho''(1)}), \dots, \gamma(t_{\sigma''(m'+p')},u_{\rho''(n'+q')})\Big)
\\
=(\eta'_{\sigma,\rho}\times \eta''_{\sigma,\rho}) \circ \,\, diag \circ ev (\gamma,t_1\leq\dots\leq t_{m+p+m'+p'},u_1\leq \dots\leq u_{n+q+n'+q'}).
\end{multline*}
(Note also, that in the above diagram, the evaluation maps really land in the corresponding subsets $\prod_{k=1}^p\prod_{\ell=1}^q (U_{i_{(k,\ell)}})^{\times (m_k\times n_\ell)}\times (U^\ever_{k,\ell})^{\times m_k}\times (U^\ehor_{k,\ell})^{\times n_\ell}\times U^v_{k,\ell} \subset M^{m\times n}$, etc., as described in definition \ref{DEF:torus-It}, which we have suppressed for better readability.) Thus,
\begin{multline*}
It^{(p,q,\U)}(\mathfrak a)\wedge It^{(p,q,\U)}(\mathfrak b)\\
=\sum_{\sigma_1,\dots,\sigma_p}\sum_{\rho_1,\dots,\rho_q} (-1)^{\epsilon}\cdot \sgn \cdot \int_{\Delta} ev^*\circ diag^*( (\eta'_{\sigma,\rho})^* (\mathfrak a)\wedge(\eta''_{\sigma,\rho})^* (\mathfrak b)),
\end{multline*}
where we abbreviated $\sgn=\sgn(\sigma_1)\cdot\ldots\cdot \sgn(\sigma_p)\cdot\sgn(\rho_1)\cdot\ldots\cdot\sgn(\rho_q)$, and $\Delta=(\Delta^{m_1+m'_1}\times\dots\times \Delta^{m_p+m'_p})\times (\Delta^{n_1+n'_1}\times\dots\times\Delta^{n_q+n'_q})$.
Now, the $\eta'_\sigma$ and $\eta''_\sigma$ add degeneracies to the $\mathfrak a$ and $\mathfrak b$ which become on forms, $(M{^{\leftrightarrow}_{r,i}}|_{\dots})^*=s^{\leftrightarrow}_{r,i}:\Oma^{n_1,\dots,n_r,\dots, n_q}_{m_1,\dots,m_p}\to \Oma^{n_1,\dots,n_r+1,\dots, n_q}_{m_1,\dots,m_p}$, and similar for $(M{^{\updownarrow}_{r,i}}|_{\dots})^*=s^{\updownarrow}_{r,i}$.
This, together with $diag$, which is the $\star$-product after passing to forms, induces the shuffle product in the above equation,
\[
It^{(p,q,\U)}(\mathfrak a)\wedge It^{(p,q,\U)}(\mathfrak b)
=\int_{\Delta} ev^* (\mathfrak a\bullet \mathfrak b)=It^{(p,q,\U)}(\mathfrak a\bullet \mathfrak b).
\]
This completes the proof of the proposition.
\end{proof}
\subsection{Holonomy for the torus}
Now, we use the local datum of an abelian gerbe with connection to define an element in the Hochschild complex. We start by recalling the definition of an abelian gerbe with connection, see \emph{e.g.} \cite{GR}.
\begin{defn}\label{DEF:gerbe}
Let $M$ be a manifold, and let $\{U_i\}_i$ be a sufficiently fine open cover of $M$. We denote by $U_{i_1,\dots,i_r}=\bigcap_{i=i_1,\dots,i_r } U_i$.
A gerbe $\mathcal G$ with connection consists of local data $g_{i,j,k}\in \Omega^0(U_{i,j,k}, U(1))$, $A_{i,j}\in \Omega^1(U_{i,j},\R)$, and $B_i\in \Omega^2(U_i,\R)$, which are symmetric in their indices,
$$ A_{i,j}=-A_{j,i}, \text{ and }  g_{i,j,k}=g^{-1}_{j,i,k}=g_{j,k,i}$$
and subject to the relations,
\begin{eqnarray*}
&g_{j,k,l}g^{-1}_{i,k,l}g_{i,j,l}g^{-1}_{i,j,k}=1 &\text{on } U_{i,j,k,l}\\
&g_{i,j,k}\cdot (A_{j,k}-A_{i,k}+A_{i,j})=i\cdot dg_{i,j,k}& \text{on } U_{i,j,k}\\
&B_j-B_i=d A_{i,j} &\text{on } U_{i,j}
\end{eqnarray*}
\end{defn}

The second relation may also be rewritten as $A_{j,k}-A_{i,k}+A_{i,j}=i\cdot g^{-1}_{i,j,k}dg_{i,j,k}=i\cdot d(\log(g_{i,j,k}))$. 
The last relation implies $d B_i=dB_j$ on $U_{i,j}$, so that there is an induced closed $3$-form $H\in \Omega^3(M,\R)$ given by $H=dB_i$ on $U_i$. For more details, see \emph{e.g.} \cite[section 2.2]{GR}. This form will be referred to as the $3$-curvature. 
\begin{rmk}
There is another, equivalent description of an abelian gerbe with connection, which we mention briefly in this remark. The data of a gerbe may be given by a surjective submersion $Y \to M$ and a principle $S^1$-bundle $P$ over $Y^{[2]}= Y \times_M Y$, with some mild compatibilities required. The data of a connective structure on a gerbe is given by $A \in \Omega^1(P; \mathbb{R})$ and $B \in \Omega^2(Y; \mathbb{R})$ such that $dA = \pi_1^* (B) - \pi_2^* (B)$. It follows that we have a well defined closed 3-form $H$ on $M$. To see that this description is equivalent with the one in definition \ref{DEF:gerbe} above, we refer the reader to \cite[section 2.3]{GR}.
\end{rmk}

\begin{defn}\label{DEF:torus-hol_0}
Starting from a connection on an abelian gerbe with local data $g_{i,j,k}\in \Omega^0(U_{i,j,k}, U(1))$, $A_{i,j}\in \Omega^1(U_{i,j},\R)$, and $B_i\in \Omega^2(U_i,\R)$, we take the $(p+1) \times (q+1)$ matrix
$$
\tilde B_{{(k,\ell)}}= 1\otimes \dots\otimes B_{{(k,\ell)}}\otimes \dots \otimes 1\in\Oma^{0,\dots,1,\dots,0}_{0,\dots,1,\dots,0},
$$
with the $2$-form $B_{{(k,\ell)}}=B_{i_{(k,\ell)}}$ at the $(k,\ell)^{\text{th}}$ entry on the open set $U_{i_{(k,\ell)}}$  and $1$s elsewhere, the $(p+1) \times q$ matrix
$$
\tilde A^\leftrightarrow_{{(k,\ell)}}= 1\otimes \dots\otimes A^\leftrightarrow_{{(k,\ell)}}\otimes \dots \otimes 1\in\Oma^{0,\dots,1,\dots,0}_{0,\dots,0}, 
$$
with the $1$-form $A^\leftrightarrow_{{(k,\ell)}}=A_{i_{(k-1,\ell)},i_{(k,\ell)}}$ at the $(k,\ell)^{\text{th}}$ horizontal edge on the open set $U^\ehor_{k,\ell}$  and $1$s elsewhere, the $p \times (q+1)$ matrix
$$
\tilde A^\updownarrow_{{(k,\ell)}}= 1\otimes \dots\otimes A^\updownarrow_{{(k,\ell)}}\otimes \dots \otimes 1\in\Oma^{0,\dots,0}_{0,\dots,1,\dots,0}, 
$$
with the $1$-form $A^\updownarrow_{{(k,\ell)}}=A_{i_{(k,\ell)},i_{(k,\ell-1)}}$ at the $(k,\ell)^{\text{th}}$ vertical edge on the open set $U^\ever_{k,\ell}$  and $1$s elsewhere, and the $p \times q$ matrix
$$
\tilde g_{{(k,\ell)}}= 1\otimes \dots\otimes g_{(k,\ell)}\otimes \dots \otimes 1\in\Oma^{0,\dots,0}_{0,\dots,0},
$$
with the $0$-form $g_{(k,\ell)}=g_{i_{(k,\ell)},i_{(k-1,\ell)},i_{(k-1,\ell-1)}}|_{U^v_{k,\ell}}\cdot g^{-1}_{i_{(k,\ell)},i_{(k,\ell-1)},i_{(k-1,\ell-1)}}|_{U^v_{k,\ell}}\in \Oma^v_{k,\ell}$
at the $(k,\ell)^{\text{th}}$ vertex on the open set $U^v_{k,\ell}$  and $1$s elsewhere.

Note, that all elements $\tilde B_{(k,l)}, \tilde A^\leftrightarrow_{{(k,\ell)}}, \tilde A^\updownarrow_{{(k,\ell)}}, \tilde g_{{(k,\ell)}}\in CH_0^{(p,q,\U)}$ are elements of total degree zero.
\end{defn}
The choice of elements in the previous definition is such that under the iterated integral map, the elements $\Ch^{(p,q,\U)}(\tilde B_{(k,l)}), \Ch^{(p,q,\U)}(\tilde A^\leftrightarrow_{{(k,\ell)}}), \Ch^{(p,q,\U)}(\tilde A^\updownarrow_{{(k,\ell)}}), \Ch^{(p,q,\U)}(\tilde g_{{(k,\ell)}})\in \Omega^0(\NN(p,q,\U))$ can be interpreted as functions on $\NN(p,q,\U)$ in an explicit way as follows.
\begin{lem}\label{LEM:local-It-gerbe-data}
We have for $\gamma\in\NN(p,q,\U)$,
\begin{eqnarray*}
\Ch^{(p,q,\U)}(\tilde B_{(k,l)})(\gamma)&=&\int_{\ou{f}{k,\ell}} \gamma^*(B_{i_{(k,\ell)}}),  \\
\Ch^{(p,q,\U)}(\tilde A^\leftrightarrow_{{(k,\ell)}})(\gamma)&=&\int_{\ou{\ehor}{k,\ell} } \gamma^*(A_{i_{(k-1,\ell)},i_{(k,\ell)}}), \\
\Ch^{(p,q,\U)}(\tilde A^\updownarrow_{{(k,\ell)}})(\gamma)&=& \int_{\ou{\ever}{k,\ell} } \gamma^*(A_{i_{(k,\ell)},i_{(k,\ell-1)}}), \\
\Ch^{(p,q,\U)}(\tilde g_{{(k,\ell)}})(\gamma)&=& g_{i_{(k,\ell)},i_{(k-1,\ell)},i_{(k-1,\ell-1)}}(\gamma(\ou{v}{k,\ell}))\cdot g^{-1}_{i_{(k,\ell)},i_{(k,\ell-1)},i_{(k-1,\ell-1)}}(\gamma(\ou{v}{k,\ell})).
\end{eqnarray*}
In the above integrals, the induced orientations of the cells from figure \ref{EQ:torus-v-e-f} is given as follows: $\ou{f}{k,\ell}$ has the induced orientation from the plane $\R^2$, $\ou{\ehor}{k,\ell}$ is oriented from left to right, and $\ou{\ever}{k,\ell}$ is oriented from top to bottom.
\end{lem}
\begin{proof}
We have, that $\Ch^{(p,q,\U)}(\tilde B_{(k,l)})(\gamma)=$
\begin{equation*}
\int_{\scriptsize\begin{matrix}(\Delta^0\times\dots\times\Delta^1\times\dots\times\Delta^0)\\
\times(\Delta^0\times\dots\times\Delta^1\times\dots\times\Delta^0)\end{matrix}} ev^*(1\otimes\dots\otimes B_{(k,\ell)}\otimes\dots\otimes 1)(\gamma)
=\int_{\ou{f}{k,\ell}} \gamma^*(B_{i_{(k,\ell)}}).
\end{equation*}
Furthermore, we have, $\Ch^{(p,q,\U)}(\tilde A^\leftrightarrow_{{(k,\ell)}})(\gamma)=$
\begin{multline*}
\int_{\scriptsize\begin{matrix}(\Delta^0\times\dots\times\Delta^0)\quad\quad\quad\quad\quad \\
\times(\Delta^0\times\dots\times\Delta^1\times\dots\times\Delta^0)\end{matrix}} ev^*(1\otimes\dots\otimes A^\leftrightarrow_{{(k,\ell)}}\otimes\dots\otimes 1)(\gamma)\\
=\int_{\ou{\ehor}{k,\ell} } \gamma^*(A_{i_{(k-1,\ell)},i_{(k,\ell)}}).
\end{multline*}
and a similar calculation for $\Ch^{(p,q,\U)}(\tilde A^\updownarrow_{{(k,\ell)}})$. Finally,
\begin{multline*}
\Ch^{(p,q,\U)}(\tilde g_{{(k,\ell)}})(\gamma)= \int_{(\Delta^0\times\dots\times\Delta^0)\times(\Delta^0\times\dots\times\Delta^0)} ev^*(1\otimes\dots\otimes g_{{(k,\ell)}}\otimes\dots\otimes 1)(\gamma)\\
= g_{i_{(k,\ell)},i_{(k-1,\ell)},i_{(k-1,\ell-1)}}(\gamma(\ou{v}{k,\ell}))\cdot g^{-1}_{i_{(k,\ell)},i_{(k,\ell-1)},i_{(k-1,\ell-1)}}(\gamma(\ou{v}{k,\ell})).
\end{multline*}
This completes the proof of the lemma.
\end{proof}

\begin{defn}
With this notation, define $\h^{(p,q,\U)}\in CH_0^{(p,q,\U)}$ to be 
$$
 \h^{(p,q,\U)}=
\exp\left(\sum_{k,\ell}{i \tilde B_{(k,\ell)}}+{i \tilde A^\leftrightarrow_{(k,\ell)}}+{i \tilde A^\updownarrow_{(k,\ell)}}\right)\bullet \prod_{k,\ell} \tilde g_{(k,\ell)},
$$
where $\exp$ is the exponential under the shuffle product $\bullet$, and let $hol^{(p,q,\U)}\in\Om^0(\NN(p,q,\U),\C)$ be given by 
\[
hol^{(p,q,\U)}=It^{(p,q,\U)}(\h^{(p,q,\U)}).
\]
\end{defn}
\begin{rmk}\label{REM:h-via-allowmatrix}
We can give an explicit formula for $hol^{(p,q,\U)}$ by using the combinatorics of the entries in the matrix for $\h^{(p,q,\U)}$. In fact, we obtain, that
\begin{equation*}
\h^{(p,q,\U)}\quad =\quad \sum_{\scriptsize\begin{matrix}m_1,\dots,m_p \geq 0\\ n_1,\dots, n_q\geq 0\end{matrix}}\quad \sum_{\scriptsize\begin{matrix}\text{allowable}\\ \text{ matrix } M\end{matrix}}\quad  \pm \mathfrak a(M), \text{ where } \mathfrak a(M)\in CH_0^{(p,q,\U)}.
\end{equation*}
Here an {\it allowable matrix} is defined as an $(m_1+\dots+m_p+p)\times (n_1+\dots+n_q+q)$ matrix with entries $0$'s, $1$'s or $2$'s, such that each vertex entry (\emph{i.e.} entries at $(m_1+\dots+m_k+k+1,n_1+\dots+n_\ell+\ell+1), \forall k, \ell$) has $0$'s, there is exactly one $2$ in each row except for the edge rows (\emph{i.e.} rows $1, m_1+2, m_1+m_2+3, \dots$), and exactly one $2$ in each column except for the edge columns (\emph{i.e.} columns $1, n_1+2, n_1+n_2+3, \dots$), and $1$'s everywhere else. The Hochschild chain $\mathfrak a(M)$ is given by replacing $0$'s by the appropriate $g_{(k,\ell)}$, and the $2$'s by the appropriate $A^\leftrightarrow_{{(k,\ell)}}$, $A^\updownarrow_{{(k,\ell)}}$, or $B_{{(k,\ell)}}$. An example is displayed in figure \ref{FIG:a(M)}.
\begin{figure}
 \[ \resizebox{12.5cm}{6cm}{
\begin{pspicture}(0.5,0.5)(16.5,9.5)
\pspolygon[fillstyle=solid,fillcolor=lightgray, linestyle=none](.5,.5)(1.5,.5)(1.5,9.5)(.5,9.5)
\pspolygon[fillstyle=solid,fillcolor=lightgray, linestyle=none](7.5,.5)(8.5,.5)(8.5,9.5)(7.5,9.5)
\pspolygon[fillstyle=solid,fillcolor=lightgray, linestyle=none](11.5,.5)(12.5,.5)(12.5,9.5)(11.5,9.5)
\pspolygon[fillstyle=solid,fillcolor=lightgray, linestyle=none](.5,9.5)(.5,8.5)(16.5,8.5)(16.5,9.5)
\pspolygon[fillstyle=solid,fillcolor=lightgray, linestyle=none](.5,5.5)(.5,4.5)(16.5,4.5)(16.5,5.5)
\rput(1,9){$g_{(1,1)}$} \rput(2,9){$1$} \rput(3,9){$1$} \rput(4,9){$A^\leftrightarrow_{(1,1)}$}
\rput(5,9){$1$} \rput(6,9){$1$} \rput(7,9){$1$} \rput(8,9){$g_{(1,2)}$}
\rput(9,9){$1$} \rput(10,9){$A^\leftrightarrow_{(1,2)}$} \rput(11,9){$1$} \rput(12,9){$g_{(1,3)}$}
\rput(13,9){$A^\leftrightarrow_{(1,3)}$} \rput(14,9){$1$} \rput(15,9){$1$} \rput(16,9){$1$}
\rput(1,8){$1$} \rput(2,8){$1$} \rput(3,8){$1$} \rput(4,8){$1$}
\rput(5,8){$B_{(1,1)}$} \rput(6,8){$1$} \rput(7,8){$1$} \rput(8,8){$1$}
\rput(9,8){$1$} \rput(10,8){$1$} \rput(11,8){$1$} \rput(12,8){$1$}
\rput(13,8){$1$} \rput(14,8){$1$} \rput(15,8){$1$} \rput(16,8){$1$}
\rput(1,7){$1$} \rput(2,7){$1$} \rput(3,7){$1$} \rput(4,7){$1$}
\rput(5,7){$1$} \rput(6,7){$1$} \rput(7,7){$1$} \rput(8,7){$1$}
\rput(9,7){$1$} \rput(10,7){$1$} \rput(11,7){$1$} \rput(12,7){$A^\updownarrow_{(1,3)}$}
\rput(13,7){$1$} \rput(14,7){$1$} \rput(15,7){$1$} \rput(16,7){$1$}
\rput(1,6){$1$} \rput(2,6){$1$} \rput(3,6){$1$} \rput(4,6){$1$}
\rput(5,6){$1$} \rput(6,6){$1$} \rput(7,6){$B_{(1,1)}$} \rput(8,6){$1$}
\rput(9,6){$1$} \rput(10,6){$1$} \rput(11,6){$1$} \rput(12,6){$1$}
\rput(13,6){$1$} \rput(14,6){$1$} \rput(15,6){$1$} \rput(16,6){$1$}
\rput(1,5){$g_{(2,1)}$} \rput(2,5){$A^\leftrightarrow_{(2,1)}$} \rput(3,5){$1$} \rput(4,5){$1$}
\rput(5,5){$1$} \rput(6,5){$A^\leftrightarrow_{(2,1)}$} \rput(7,5){$1$} \rput(8,5){$g_{(2,2)}$}
\rput(9,5){$A^\leftrightarrow_{(2,2)}$} \rput(10,5){$1$} \rput(11,5){$1$} \rput(12,5){$g_{(2,3)}$}
\rput(13,5){$1$} \rput(14,5){$1$} \rput(15,5){$A^\leftrightarrow_{(2,3)}$} \rput(16,5){$A^\leftrightarrow_{(2,3)}$}
\rput(1,4){$1$} \rput(2,4){$1$} \rput(3,4){$1$} \rput(4,4){$1$}
\rput(5,4){$1$} \rput(6,4){$1$} \rput(7,4){$1$} \rput(8,4){$1$}
\rput(9,4){$1$} \rput(10,4){$1$} \rput(11,4){$1$} \rput(12,4){$1$}
\rput(13,4){$1$} \rput(14,4){$B_{(2,3)}$} \rput(15,4){$1$} \rput(16,4){$1$}
\rput(1,3){$A^\updownarrow_{(2,1)}$} \rput(2,3){$1$} \rput(3,3){$1$} \rput(4,3){$1$}
\rput(5,3){$1$} \rput(6,3){$1$} \rput(7,3){$1$} \rput(8,3){$1$}
\rput(9,3){$1$} \rput(10,3){$1$} \rput(11,3){$1$} \rput(12,3){$1$}
\rput(13,3){$1$} \rput(14,3){$1$} \rput(15,3){$1$} \rput(16,3){$1$}
\rput(1,2){$1$} \rput(2,2){$1$} \rput(3,2){$1$} \rput(4,2){$1$}
\rput(5,2){$1$} \rput(6,2){$1$} \rput(7,2){$1$} \rput(8,2){$1$}
\rput(9,2){$1$} \rput(10,2){$1$} \rput(11,2){$B_{(2,2)}$} \rput(12,2){$1$}
\rput(13,2){$1$} \rput(14,2){$1$} \rput(15,2){$1$} \rput(16,2){$1$}
\rput(1,1){$1$} \rput(2,1){$1$} \rput(3,1){$B_{(2,1)}$} \rput(4,1){$1$}
\rput(5,1){$1$} \rput(6,1){$1$} \rput(7,1){$1$} \rput(8,1){$1$}
\rput(9,1){$1$} \rput(10,1){$1$} \rput(11,1){$1$} \rput(12,1){$1$}
\rput(13,1){$1$} \rput(14,1){$1$} \rput(15,1){$1$} \rput(16,1){$1$}
\end{pspicture}}
\]
\caption{An example of $\mathfrak a(M)$ for an allowable matrix $M$, with $p=2, m_1=3, m_2=4$, and $q=3, n_1=6, n_2=3, n_3=4$; edge rows and columns are indicated in gray.}\label{FIG:a(M)}
\end{figure}

With this, we can write a formula for $It^{(p,q,\U)}(\mathfrak a(M))$ for an allowable matrix $M$, applied to an element $\gamma\in \NN(p,q,\U)$, up to sign, as follows,
\begin{multline*}
It^{(p,q,\U)}(\mathfrak a(M)) = \pm \prod_{k=1}^p \prod_{\ell=1}^q g_{(k,\ell)}\left(\gamma\left(\frac{k-1}{p},\frac{\ell-1}{q}\right)\right)\\
 \cdot \int_{(\Delta^{m_1}\times\dots \times \Delta^{m_p})\times(\Delta^{n_1}\times\dots \times \Delta^{n_q})} \,\,
\bigwedge_{\scriptsize\begin{matrix}\text{entries ``$2$'' in $M$,}\\ \text{given at position }(r,s) \\ \text{in the $(k,\ell)$'s subrectangle}\end{matrix}}
X_{(k,\ell)}\left(\frac{k-1+t^k_r}{p},\frac{\ell-1+u^\ell_s}{q}\right)\\ dt^1_1 \dots dt^q_{m_q} du^1_1 \dots du^p_{n_p},
\end{multline*}
where 
\[
X_{(k,\ell)}(t,u) = \left\{
\begin{array}{ll}
\iota_\wf A^\leftrightarrow_{(k,\ell)}(u), & \text{if the ``$2$'' is on an edge row} \\
\iota_\vf A^\updownarrow_{(k,\ell)}(t), & \text{if the ``$2$'' is on an edge column} \\
\iota_\wf\iota_\vf B_{(k,\ell)}(t,u), & \text{otherwise}
\end{array} \right.
\]
and where $\vf$ and $\wf$ being the two natural vector fields on $M^\T$ as before. To justify this formula, note that the integral is non-zero  only for the terms that have a volume form for the fiber, and equation \eqref{EQ:T-pullback} provides $dt^k_r$ and $du^\ell_s$ only for the entries of ``$2$'' in the matrix $M$.
\end{rmk}

\begin{prop}\label{PROP:hol-for-torus}
Let $\gamma\in M^\T$ and subdivide the torus $\T$ by some faces $f$, edges $e$ and vertices $v$.  Choose open sets $U_{i_v}$ for each vertex $v$, $U_{i_e}$ for each edge $e$, and $U_{i_f}$ for each face $f$, and assume that $\gamma|_f\subset U_{i_f}, \gamma|_e\subset U_{i_e}$, and $\gamma(v)\in U_{i_v}$ for all faces $f$, edges $e$, and vertices $v$.
Then, the holonomy function $hol:M^\T\to U(1)$ defined by
\begin{multline}\label{EQ:hol(q,p,U)}
 hol(\gamma):=\\
\exp\left(\sum_f i \int_f \gamma^*(B_{i_f})-\sum_{e\subset f} i \cdot \rho(e,f)\int_e \gamma^*(A_{i_e,i_f})\right) \prod_{v\subset e\subset f} g^{\rho(v, e, f)}_{i_v,i_e,i_f}(\gamma(v))
 \end{multline}
is independent of the subdivision and choice of open sets $U_{i_v}, U_{i_e}$, and $U_{i_f}$. Here, we sum over all faces $f$ in the first sum, all possible subset combinations $e\subset f$ in the second sum, and all possible subset combinations $v\subset e\subset f$ in the last product. Furthermore, $\rho(e,f)\in\{+1,-1\}$ with $\rho(e,f)=+1$ iff $e$ has the induced orientation coming from $f$, and $\rho(e,f)=-1$ otherwise, and $\rho(v,e,f)=\rho(e,f)$ if $v$ is the beginning point of $e$, and $\rho(v,e,f)=-\rho(e,f)$ otherwise (\emph{c.f.} definition \ref{DEF:rho(v,e,f)}).

In particular,  we have that $hol^{(p,q,\U)}(\gamma)=hol(\gamma)$.
\end{prop}
\begin{proof}
The first part on the independence of $hol$ is a well-known fact, see \emph{e.g.} \cite[section 2.2]{GR} or \cite[section 2.2]{RS} and corollary \ref{LEM:Sigma-hol-glues} below.

In order to calculate $hol^{(p,q,\U)}(\gamma)=\Ch^{(p,q,\U)}(\mathfrak h^{(p,q,\U)})(\gamma)$, we make the following choices. Subdivide the torus as in figure \ref{EQ:torus-v-e-f}, and choose the open sets to be $ U_{i_f}=U_{i_e}=U_{i_v}=U_{i_{(k,\ell)}}$ for the face $f=\ou {f}{k,\ell}$, the edges $e=\ou{\ehor}{k,\ell}$ or $e=\ou{\ever}{k,\ell}$, and for the vertex $v=\ou{v}{k,\ell}$. Now, under this choice, the index $i_e$ for each edge $e$ coincides with one of the two faces adjacent to $e$, \emph{i.e.} $i_{\ou {f}{k,\ell}}=i_{\ou {\ehor}{k,\ell}}=i_{\ou {\ever}{k,\ell}}=i_{(k,\ell)}$, so that $A_{i_{\ou {\ehor}{k,\ell}},i_{\ou {f}{k,\ell}}}=A_{i_{\ou {\ever}{k,\ell}},i_{\ou {f}{k,\ell}}}=0$. Furthermore, a straightforward check shows, that for a fixed vertex $v=\ou{v}{k,\ell}$, the above choice gives $\prod_{v\subset e\subset f} g^{\rho(v,e,f)}_{i_v,i_e,i_f}=g_{i_{(k,\ell)},i_{(k-1,\ell)},i_{(k-1,\ell-1)}}\cdot g^{-1}_{i_{(k,\ell)},i_{(k,\ell-1)},i_{(k-1,\ell-1)}}$. Thus, under the above choices the right hand side of the equation \eqref{EQ:hol(q,p,U)} becomes,
\begin{multline}\label{EQ:hol-gamma-for-gerbe-expliciti}
hol(\gamma)=\\
 \exp\left(\sum_{k,\ell} i \int_{\ou f {k ,\ell}} \gamma^*(B_{i_{(k,\ell)}})+i \int_{\ou{\ehor}{k,\ell} } \gamma^*(A_{i_{(k-1,\ell)},i_{(k,\ell)}})+i \int_{\ou{\ever}{k,\ell} } \gamma^*(A_{i_{(k,\ell)},i_{(k,\ell-1)}}) \right)  \\
 \cdot \prod_{k,\ell} g_{i_{(k,\ell)},i_{(k-1,\ell)},i_{(k-1,\ell-1)}}\cdot g^{-1}_{i_{(k,\ell)},i_{(k,\ell-1)},i_{(k-1,\ell-1)}}.
\end{multline}
Now, we may use lemma \ref{LEM:local-It-gerbe-data} and proposition \ref{LEM:It(q,p,U)-d-sh} to see that this coincides with
\begin{multline*}
\exp\left(\sum_{k,\ell}i \cdot \Ch^{(p,q,\U)}(\tilde B_{(k,\ell)})+i \cdot\Ch^{(p,q,\U)}(\tilde A^\leftrightarrow_{(k,\ell)})+i \cdot\Ch^{(p,q,\U)}(\tilde A^\updownarrow_{(k,\ell)})\right)(\gamma)\\
\quad\quad\quad\quad\quad\quad\quad\quad\quad\quad\quad\quad\quad\quad\quad\quad\quad\quad\quad\quad \cdot \prod_{k,\ell} \Ch^{(p,q,\U)}( \tilde g_{(k,\ell)})(\gamma) \\
= It^{(q,p,\U)}\left( \exp\left(\sum_{k,\ell}{i \tilde B_{(k,\ell)}}+{i \tilde A^\leftrightarrow_{(k,\ell)}}+{i \tilde A^\updownarrow_{(k,\ell)}}\right)\bullet \prod_{k,\ell} \tilde g_{(k,\ell)}\right)(\gamma)\\
=\Ch^{(p,q\U)}(\mathfrak h^{(p,q,\U)})(\gamma)=hol^{(p,q,\U)}(\gamma),
\end{multline*}
But this is the claim, which thus completes the proof of the proposition.
\end{proof}
\begin{cor}\label{COR:hol-global}
The functions $It^{(p,q,\U)}(\h^{(p,q,\U)})\in \Omega^0(\NN(p,q,\U),\C)$ glue together to give a global function $hol\in \Om^0(M^\T,\C)$ with
$$ hol|_{\NN(p,q,\U)}=hol^{(p,q,\U)}. $$
\end{cor}

\subsection{Higher holonomies for the torus}\label{SUBSEC:torus-hol-2k-2l}
In analogy with section \ref{SEC:local-vector-bundle} for vector bundles, we define higher holonomies in a way that allows us to build an equivariantly closed form on the torus mapping space $M^\T$. In fact, we define a sequence of forms $hol_{2k,2\ell}\in \Om(M^\T,\C)$, that may be put together in a variety of ways to give an equivariantly closed form on the torus mapping space $M^\T$. The guiding diagram, that is the analogy of figure \ref{FIG:nabla_hol=-i_hol} for the vector bundle case, is the diagram depicted in figure \ref{DIAG:hol_k,l}.
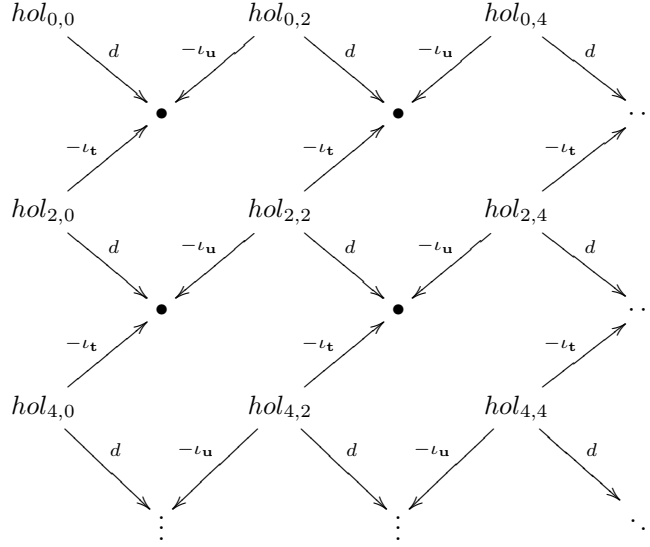
\begin{figure}\[
\xymatrix{
hol_{0,0} \ar[dr]^{d} & & hol_{0,2} \ar[dl]_{-\iota_\wf}  \ar[dr]^{d} & & hol_{0,4} \ar[dl]_{-\iota_\wf}\ar[dr]^{d}\\
 & \bullet & & \bullet && \cdots\\
hol_{2,0} \ar[ur]^{-\iota_{\vf}} \ar[dr]^{d} & & hol_{2,2} \ar[ur]^{-\iota_{\vf}} \ar[dl]_{-\iota_{\wf}}  \ar[dr]^{d} & & hol_{2,4} \ar[ur]^{-\iota_{\vf}} \ar[dl]_{-\iota_{\wf}}\ar[dr]^{d}\\
 & \bullet & & \bullet && \cdots\\
hol_{4,0} \ar[ur]^{-\iota_{\vf}} \ar[dr]^{d} & & hol_{4,2} \ar[ur]^{-\iota_{\vf}} \ar[dl]_{-\iota_{\wf}}  \ar[dr]^{d} & & hol_{4,4} \ar[ur]^{-\iota_{\vf}} \ar[dl]_{-\iota_{\wf}}\ar[dr]^{d}\\
 & \vdots & & \vdots && \ddots
}\]
\caption{Diagram for $hol_{2k,2\ell}$'s, stating that for all $k,\ell\geq 0$, $d(hol_{2k,2\ell})=-\iota_\vf(hol_{2k+2,2\ell})=-\iota_\wf(hol_{2k,2\ell+2})$}\label{DIAG:hol_k,l}
\end{figure}
Here, we have set $hol_{0,0}=hol$ form the previous subsection \ref{SUBSEC:torus-hol}. More precisely, diagram \ref{DIAG:hol_k,l} depicts the relation, 
\begin{equation*}
 d_{DR}(hol_{2k,2\ell})=-\iota_{\vf}(hol_{2k+2,2\ell})=-\iota_{\wf}(hol_{2k,2\ell+2}), \quad \forall k,\ell\in\N_0.
\end{equation*}

Let us first make some motivating remarks. One can calculate that  $d_{DR}(hol)$ is the 1-form on $M^\T$ given by $i\cdot It(H) \wedge hol$ where $H=dB$  is the $3$-curvature (see theorem \ref{d-it-is} below). Recall the similar situation in the previous section, where for a vector bundle, $d_{DR}( hol) = -It(1\otimes R) \wedge hol$ where $R$ is the $2$-curvature (see proposition \ref{prop:De^a}), and we illustrated how this differential 1-form on $LM$ can be expressed using $E = \rho^* \circ It$.
Namely, it was the image under $E$ of an element of local Hochschild complex, roughly given by shuffling $Rdt$ into the expression whose iterated integral is the holonomy. Similarly, shuffling in $(Rdt)^{\ot k}$ gave all the higher holonomies.

In much the same way, we define a map $E$ for the gerbe-torus case, and define elements given by shuffling terms given by the $3$-curvature $H$, using these to define the higher holonomies for this gerbe-torus case. These will satisfy the equations above and the total sum will give an equivariantly closed form on $M^\T$.

First we start be defining a map $E^{(p,q,\U)}$, analogously to $E^{(p,\U)}$ from definition \ref{DEF:E(p,U)}.
\begin{defn}\label{DEF:E(q,p,U)}
For a choice of $p, q \in \N$ and a choice of open sets $\U=\{U_{i_{(k,\ell)}}\}_{k,\ell}$ of $M$, we have induced open sets $\U\times \T=\{U_{i_{(k,\ell)}}\times \T\}_{k,\ell}$ of $M\times \T$. We have the iterated integral map $It^{(p,q,\U\times \T)}:CH_\bullet^{(p,q,\U\times \T)}\to\Omega(\NN(p,q,\U\times \T))$ for $\NN(p,q,\U\times \T)\subset (M\times \T)^\T$, and we furthermore define $\rho_\NN:\NN(p,q,\U)\to \NN(p,q,\U\times \T), \gamma\mapsto \rho_\NN(\gamma)$ by setting $\rho_\NN(\gamma):\T\to M\times \T, \rho_\NN(\gamma)(t,u)=(\gamma(t,u),(-t,-u))$. Using these maps, we define the extended iterated integral $E^{(p,q,\U)}:CH_\bullet^{(p,q,\U\times \T)}\to \Omega(\NN(p,q,\U))$ as the composition $E^{(p,q,\U)}:=(\rho_\NN)^*\circ It^{(p,q,\U\times \T)}$,
$$ E^{(p,q,\U)}:CH_\bullet^{(p,q,\U\times \T)}\stackrel {It^{(p,q,\U\times \T)}} \longrightarrow\Omega(\NN(p,q,\U\times \T)) \stackrel {(\rho_{\NN})^*} \longrightarrow \Omega(\NN(p,q,\U)). $$
\end{defn}
With this, we can now define higher holonomy elements $hol_{2k,2\ell}^{(p,q,\U)} \in \Omega(\NN(p,q,\U),\C)$.
\begin{defn}\label{DEF:h(p,q,u)(2k,2l)}
If we now denote by $dt'$ and $du'$ the two canonical $1$-forms on the torus, then $Hdt'$ and $Hdu'$ are $4$-forms in $\Om^4(M\times \T)$. The restriction of these forms to $U_{i_{(r,s)}}\times \T$ is denoted by $(Hdt')_{(r,s)}, (Hdu')_{(r,s)} \in \Om^4(U_{i_{(r,s)}} \times \T)$, and placing them at the $(r,s)^{\text{th}}$ spot in the Hochschild complex will be denoted by,
\begin{eqnarray*}
\widetilde{Hdt'}_{(r,s)}&=& 1\otimes\dots\otimes (Hdt')_{(r,s)}\otimes \dots\otimes 1\in\Oma^{0,\dots,1,\dots,0}_{0,\dots,1,\dots,0},
\\
\widetilde{Hdu'}_{(r,s)}&=& 1\otimes\dots\otimes (Hdu')_{(r,s)}\otimes \dots\otimes 1\in\Oma^{0,\dots,1,\dots,0}_{0,\dots,1,\dots,0}.
\end{eqnarray*}

Denote by $\h_{2k,2\ell}^{(p,q,\U)}\in CH_{2k+2\ell}^{(p,q,\U\times \T)}$ the following sum,
\begin{multline*}
\h_{2k,2\ell}^{(p,q,\U)} :=\sum_{\scriptsize \begin{matrix}k_{(1,1)},\dots, k_{(p,q)}\geq 0\\ k_{(1,1)}+\dots+ k_{(p,q)} =k \end{matrix}}\quad \sum_{\scriptsize \begin{matrix}\ell_{(1,1)},\dots, \ell_{(p,q)}\geq 0\\ \ell_{(1,1)}+\dots+ \ell_{(p,q)} =\ell \end{matrix}}
 \prod_{r=1}^p \prod_{s=1}^q  \frac{1}{k_{(r,s)}!\cdot \ell_{(r,s)}!}\\
 \cdot  \left(i\widetilde{Hdt'}_{(r,s)}\right)^{\bullet k_{(r,s)}} \bullet \left(i\widetilde{Hdu'}_{(r,s)}\right)^{\bullet \ell_{(r,s)}} \bullet  \exp\left({i \tilde B_{(r,s)}}+{i \tilde A^\leftrightarrow_{(r,s)}}+{i \tilde A^\updownarrow_{(r,s)}}\right)
\bullet \tilde g_{(r,s)}.
\end{multline*}
Furthermore, $hol_{2k,2\ell}^{(p,q,\U)}\in \Omega(\NN(p,q,\U),\C)$ is defined as,
\[
hol_{2k,2\ell}^{(p,q,\U)}:=E^{(p,q,\U)}\left(\h_{2k,2\ell}^{(p,q,\U)}\right)\quad \in \Omega^{2k+2\ell}(\NN(p,q,\U),\C).
\]
\end{defn}
The following proposition expresses each higher holonomy as a product of DeRham forms on the total mapping space $M^\T$.
\begin{prop}\label{PROP:hol_(2k,2l)-glue}
The locally defined forms $hol_{2k,2\ell}^{(p,q,\U)}$ define a global $(2k+2\ell)$-form $hol_{2k,2\ell}\in \Om^{2k+2\ell}(M^\T,\C)$, with the restrictions $hol_{2k,2\ell}\big|_{\NN(p,q,\U)}= hol_{2k,2\ell}^{(p,q,\U)}$. Explicitly, this function is given by
\begin{equation} \label{condensedhol2k2l}
hol_{2k,2\ell} = \frac{(-1)^\ell}{k!\cdot \ell !} \cdot i^{k+\ell} \cdot \left( \int_\T (\iota_\wf H )dtdu\right)^{\wedge k} \cdot \left( \int_\T (\iota_\vf H) dtdu \right)^{\wedge \ell} \cdot hol.
\end{equation}
In particular, $hol_{0,0}=hol$.
\end{prop}
\begin{proof}
We will show that \eqref{condensedhol2k2l} holds on any open set $\NN(p,q,\U)\subset M^\T$ for some local data $\U$.
The last statement about $hol_{0,0}$ does not include any $dt'$ or $du'$, so that,
\[
hol_{0,0}^{(p,q,\U)}=(\rho_\NN)^*\circ It^{(p,q,\U\times \T)}\left(\h_{0,0}^{(p,q,\U)}\right)=It^{(p,q,\U)}\left(\h^{(p,q,\U)}\right)=hol^{(p,q,\U)}.
\]
Now, since both $(\rho_\NN)^*$ and $It^{(p,q,\U\times \T)}$ are algebra maps, we obtain,
\begin{eqnarray*}
hol^{(p,q,\U)}_{2k,2\ell}&=& \sum_{\scriptsize \begin{matrix}k_{(1,1)}+\dots+ k_{(p,q)} =k \\ \ell_{(1,1)}+\dots+ \ell_{(p,q)} =\ell \end{matrix}} \left(\prod_{r=1}^p \prod_{s=1}^q  \frac{1}{k_{(r,s)}!}\cdot E^{(p,q,\U)}\left(i\widetilde{Hdt'}_{(r,s)}\right)^{\wedge k_{(r,s)}} \right)
\\
&&\quad\quad\cdot \left(\prod_{r=1}^p \prod_{s=1}^q  \frac{1}{ \ell_{(r,s)}!} 
 \cdot E^{(p,q,\U)} \left(i\widetilde{Hdu'}_{(r,s)}\right)^{\wedge \ell_{(r,s)}} \right)
 \\
 && \quad\quad\cdot   E^{(p,q,\U)}\left(\prod_{r=1}^p \prod_{s=1}^q\exp\left({i \tilde B_{(r,s)}}+{i \tilde A^\leftrightarrow_{(r,s)}}+{i \tilde A^\updownarrow_{(r,s)}}\right) \bullet \tilde g_{(r,s)}\right).
\end{eqnarray*}
The last factor is evaluated as 
\begin{multline}\label{EQ-E(p,q,u)(exp)=hol}
E^{(p,q,\U)}\left(\prod_{r=1}^p \prod_{s=1}^q\exp\left({i \tilde B_{(r,s)}}+{i \tilde A^\leftrightarrow_{(r,s)}}+{i \tilde A^\updownarrow_{(r,s)}}\right) \bullet \tilde g_{(r,s)}\right)
\\ =hol_{0,0}^{(p,q,\U)}=hol^{(p,q,\U)},
\end{multline}
so that,
\begin{multline*}
hol^{(p,q,\U)}_{2k,2\ell}=
\sum_{k_{(1,1)}+\dots+ k_{(p,q)} =k } \left(\prod_{r=1}^p \prod_{s=1}^q  \frac{1}{k_{(r,s)}!}\cdot E^{(p,q,\U)}\left(i\widetilde{Hdt'}_{(r,s)}\right)^{\wedge k_{(r,s)}} \right)
\\
\cdot \sum_{\ell_{(1,1)}+\dots+ \ell_{(p,q)} =\ell } \left(\prod_{r=1}^p \prod_{s=1}^q  \frac{1}{\ell_{(r,s)}!}\cdot E^{(p,q,\U)}\left(i\widetilde{Hdu'}_{(r,s)}\right)^{\wedge \ell_{(r,s)}} \right)\cdot hol^{(p,q,\U)}
\\
=\frac{1}{k!}\left(\sum_{r=1}^p\sum_{s=1}^q E^{(p,q,\U)}\left(i\widetilde{Hdt'}_{(r,s)}\right)\right)^{\wedge k}
\cdot \frac{1}{\ell!}\left(\sum_{r=1}^p\sum_{s=1}^q E^{(p,q,\U)}\left(i\widetilde{Hdu'}_{(r,s)}\right)\right)^{\wedge \ell}\cdot hol^{(p,q,\U)}.
\end{multline*}
We claim that
\begin{eqnarray} 
\label{EQ-E(Hdt)}
E^{(p,q,\U)}\left(\widetilde{Hdt'}_{(r,s)} \right) &=& \int_{\Delta^1\times \Delta^1} (\iota_\wf H)\Big(\frac{r-1+t}{p},\frac{s-1+u}{q} \Big) dtdu
 \\ \label{EQ-E(Hdu)}
E^{(p,q,\U)}\left(\widetilde{Hdu'}_{(r,s)} \right) &=& - \int_{\Delta^1\times \Delta^1} (\iota_\vf H)\Big(\frac{r-1+t}{p},\frac{s-1+u}{q} \Big) dtdu 
\end{eqnarray}
Before proving these identities, we show how to finish the proof using \eqref{EQ-E(Hdt)} and \eqref{EQ-E(Hdu)}. Summing \eqref{EQ-E(Hdt)} and \eqref{EQ-E(Hdu)} over $r=1,\dots,p$ and $s=1,\dots,q$, we obtain
\begin{eqnarray*}
hol^{(p,q,\U)}_{2k,2\ell} &=&\frac{1}{k!}\left(i\cdot \int_{\Delta^1\times \Delta^1} (\iota_{\wf}H)(t,u)dtdu \right)^{\wedge k} \\
&& \cdot \frac{1}{\ell!}\left((-i) \cdot \int_{\Delta^1\times \Delta^1} (\iota_{\wf}H)(t,u)dtdu \right)^{\wedge \ell}\cdot hol^{(p,q,\U)}.
\end{eqnarray*}
Thus, we see that equation \eqref{condensedhol2k2l} holds locally on any $\NN(p,q,\U)$. Equation \eqref{condensedhol2k2l} does in fact define a global $2k+2\ell$-form $M^\T$, since $hol$ is globally defined (see corollary \ref{COR:hol-global}), and so are $\int_\T \iota_\wf H dt du$ and $\int_\T \iota_\vf H dt du$, since $H$ is global on $M$. (For example, $\int_\T \iota_\wf H dt du$ may also be defined via the extended iterated integral $E^{(1,1,\{M\})}\big(\widetilde{Hdt'}_{(1,1)}\big)\in\Omega(\NN(1,1,\{M\}))$ of the global cover $\U=\{M\}$ for which $\NN(1,1,\{M\})=M^\T$.) 

It remains to prove equations \eqref{EQ-E(Hdt)} and\eqref{EQ-E(Hdu)}.
\begin{proof}[Proof of \eqref{EQ-E(Hdt)} and \eqref{EQ-E(Hdu)}]
To prove the two equations, we will need to consider vector fields on $M^\T$ as well as vector fields on $(M\times \T)^\T$. Since the latter space has two natural $\T$ actions, we will distinguish them by writing the torus in the base with a prime, \emph{i.e.} we write the space as $(M\times \T')^\T$ where $\T'=\T$. The two natural vector fields on $M^\T$ are denoted by $\vf$ and $\wf$, and similarly $\vf(M\times \T')$ and $\wf(M\times \T')$ are the vector fields on $(M\times \T')^\T$ coming from the $\T$ action in the exponent, and $\partial/\partial t'$ and $\partial/\partial u'$ are the vector fields on $(M\times \T')^\T$ coming from the $\T$ action in the base $\T=\T'$.
Now, $\rho_\NN:\NN(p,q,\U)\to \NN(p,q,\U\times \T'), \rho_\NN:\gamma\mapsto (\gamma,-id)$, from definition \ref{DEF:E(q,p,U)} is $\T$-equivariant where the $\T$ action on $(M\times \T')^\T$ is the diagonal action of both tori $\T$ and $\T'$,
\begin{multline*}
(t,u).\rho_\NN(\gamma)=(t,u).(\gamma(\_),-id(\_))=(\gamma(\_+(t,u)),-id(\_+(t,u))+(t,u)) 
\\ 
=(\gamma(\_+(t,u)),-id(\_))=((t,u).\gamma,-id)=\rho_\NN((t,u).\gamma).
\end{multline*}
Thus, we see that $(\rho_\NN)_*(\vf)=\vf(M\times\T')+\partial/\partial t'$ and similarly $(\rho_\NN)_*(\wf)=\wf(M\times\T')+\partial/\partial u'$. Note furthermore that $\iota_{(\rho_\NN)_*\mathbf v} (d t')=\iota_{(\rho_\NN)_*\mathbf v} (d u')=0$ for any vector field $\mathbf v$, since $\rho_\NN:\gamma\mapsto (\gamma,-id)$ is constant on the second factor.

With these remarks, we calculate $E^{(p,q,\U)}\big(\widetilde{Hdt'}_{(r,s)} \big)\in \Omega^2(\NN(p,q,\U))$ by applying it to two vectors fields $\mathbf v$ and $\mathbf w$ on $\NN(p,q,\U)$ as follows,
\begin{multline*}
E^{(p,q,\U)}\left(\widetilde{Hdt'}_{(r,s)}\right)(\mathbf v , \mathbf w) =(\rho_\NN)^*\circ\Ch^{(p,q,\U\times \T')}\left(\widetilde{Hdt'}_{(r,s)}\right)(\mathbf v , \mathbf w)\\
=\int_{\scriptsize \begin{matrix}
(\Delta^0\times\dots\times\Delta^1\times\dots\times\Delta^0)\\
\times(\Delta^0\times\dots\times\Delta^1\times\dots\times\Delta^0)
\end{matrix}} ev^*(1\otimes\dots\otimes Hdt'_{(r,s)}\otimes\dots\otimes 1)((\rho_\NN)_*\mathbf v,(\rho_\NN)_*\mathbf w)\\
=\int_{\Delta^1\times \Delta^1}  dtdu \wedge \iota_{\wf(M\times \T')}\iota_{\vf(M\times \T')} (H\wedge dt')\Big(\frac{r-1+t}{p},\frac{s-1+u}{q} \Big) ((\rho_\NN)_*\mathbf v,(\rho_\NN)_*\mathbf w).
\end{multline*}
where we have used from \eqref{EQ:T-pullback} that $\int_{I\times J} ev^*(\omega)=\int_{I\times J} dtdu \wedge \iota_{\wf(M\times \T')}\iota_{\vf(M\times \T')} \omega(t,u)$, for any intervals $I$ and $J$. Using the remarks from the first paragraph, we obtain that this is equal to
\begin{multline*}
\int_{\Delta^1\times \Delta^1}  dtdu \wedge (\iota_{(\rho_\NN)_*\wf}-\iota_{\partial/\partial u'})(\iota_{(\rho_\NN)_*\vf}-\iota_{\partial/\partial t'}) \\
\quad\quad\quad\quad\quad\quad\quad\quad (H\wedge dt')\Big(\frac{r-1+t}{p},\frac{s-1+u}{q} \Big) ((\rho_\NN)_*\mathbf v,(\rho_\NN)_*\mathbf w)
\\
=\int_{\Delta^1\times \Delta^1}  dtdu \wedge (\iota_{(\rho_\NN)_*\wf})(-\iota_{\partial/\partial t'}) (H\wedge dt')\Big(\frac{r-1+t}{p},\frac{s-1+u}{q} \Big) ((\rho_\NN)_*\mathbf v,(\rho_\NN)_*\mathbf w)
\\
=\int_{\Delta^1\times \Delta^1}  dtdu \wedge (\iota_{(\rho_\NN)_*\wf}) H\Big(\frac{r-1+t}{p},\frac{s-1+u}{q} \Big) ((\rho_\NN)_*\mathbf v,(\rho_\NN)_*\mathbf w)
\\
=\int_{\Delta^1\times \Delta^1}  dtdu \wedge (\iota_{\wf}) H\Big(\frac{r-1+t}{p},\frac{s-1+u}{q} \Big) (\mathbf v,\mathbf w).
\end{multline*}
This shows \eqref{EQ-E(Hdt)}, and a similar calculation for $E^{(p,q,\U)}\big(\widetilde{Hdu'}_{(r,s)}\big)$ (with an additional minus coming from commuting $(-\iota_{\partial/\partial u'})(\iota_{(\rho_\NN)_*\vf})=(\iota_{(\rho_\NN)_*\vf})(\iota_{\partial/\partial u'})$) also shows equation \eqref{EQ-E(Hdu)}.
\end{proof}
This completes the proof of proposition \ref{PROP:hol_(2k,2l)-glue}.
\end{proof}
We may rewrite $hol_{2k,2\ell}^{(p,q,\U)}$, at least partially, in a more explicit way using the notion of allowable matrix as defined in remark \ref{REM:h-via-allowmatrix}.
\begin{rmk}
We have for a torus $(\gamma:\T\to M) \in \NN(q,p,\U)$, that
\begin{multline*}
hol_{2k,2\ell}^{(p,q,\U)}|_\gamma = \sum_{\scriptsize\begin{matrix}m_1,\dots,m_p \geq 0\\ n_1,\dots, n_q\geq 0\end{matrix}}\quad \sum_{\scriptsize\begin{matrix}\text{allowable}\\ \text{ matrix } M \\ \text{size $m \times n$} \end{matrix}}\quad  \sum_{\scriptsize\begin{matrix} K,L\subset \{\text{entries of $2$'s in }M\\
\text{that are not in an edge}\\\text{row or edge column}\}\\ K\cap L=\emptyset, |K|=k, |L|=\ell \end{matrix}} \\
\pm \prod_{i=1}^q \prod_{j=1}^p g_{(i,j)}\left(\gamma\left(\frac{i-1}{q},\frac{j-1}{p}\right)\right)\\
 \cdot \int_{(\Delta^{m_1}\times\dots \times \Delta^{m_q})\times(\Delta^{n_1}\times\dots \times \Delta^{n_p})} \,\,
\bigwedge_{\scriptsize\begin{matrix}\text{entries ``$2$'' in $M$,}\\ \text{given at position }(r,s) \\ \text{in the $(i,j)$'s subrectangle}\end{matrix}}
X_{(i,j)}\left(\frac{i-1+t^i_r}{q},\frac{j-1+u^j_s}{p}\right)\\ dt^1_1 \dots du^p_{n_p},
\end{multline*}
where 
\[
X_{(i,j)}(t,u) = \left\{
\begin{array}{ll}
\iota_\wf A^\leftrightarrow_{(i,j)}(u), & \text{if the ``$2$'' is on an edge row} \\
\iota_\vf A^\updownarrow_{(i,j)}(t), & \text{if the ``$2$'' is on an edge column} \\
\iota_\wf H(t,u), & \text{if the ``$2$'' is at a position in }K \\
\iota_\vf H(t,u), & \text{if the ``$2$'' is at a position in }L \\
\iota_\vf\iota_\wf B_{(i,j)}(t,u), & \text{otherwise}
\end{array} \right.
\]
\begin{proof}[Sketch of proof.]
It is $hol_{2k,2\ell}^{(p,q,\U)}:=(\rho_\NN)^*\circ It^{(q,p,\U\times \T)}\left(\h_{2k,2\ell}^{(p,q,\U)}\right)$, and similarly to remark \ref{REM:h-via-allowmatrix} we can write,
\begin{multline}\label{EQ:h_(2k,2l)}
\h_{2k,2\ell}^{(p,q,\U)} = \sum_{\scriptsize\begin{matrix}m_1,\dots,m_p \geq 0\\ n_1,\dots, n_q\geq 0\end{matrix}} \sum_{\scriptsize\begin{matrix}\text{allowable}\\ \text{ matrix } M \\ \text{of size} \\ (m+p) \times (n+q) \end{matrix}}\quad  \sum_{\scriptsize\begin{matrix} K,L\subset \{\text{entries of $2$'s in }M\\
\text{that are not in an edge}\\\text{row or edge column}\}\\ K\cap L=\emptyset, |K|=k, |L|=\ell \end{matrix}}
\quad  \pm \mathfrak a_{K,L}(M), \\ 
\text{where } \mathfrak a_{K,L}(M)\in CH_{2k+2\ell}^{(p,q,\U\times \T)}.
\end{multline}
Here, the allowable matrices are the same as in remark \ref{REM:h-via-allowmatrix}, and $m = m_1 + \dots + m_p $, $n = n_1 + \dots + n_q$. The Hochschild chain $\mathfrak a_{K,L}(M)$ is given by replacing $0$'s by the appropriate $g_{(k,\ell)}$, and the $2$'s by the appropriate $A^\leftrightarrow_{{(r,s)}}$, $A^\updownarrow_{{(r,s)}}$, or $B_{{(r,s)}}$, except for $2$'s in $K$ or $L$, for which we place the appropriate $(Hdt')_{(r,s)}$ or $(Hdu')_{(r,s)}$. An example is displayed in figure \ref{FIG:a_(K,L)(M)}.

\begin{figure}
 \[ \resizebox{12.5cm}{6cm}{
\begin{pspicture}(0.5,0.5)(16.5,9.5)
\pspolygon[fillstyle=solid,fillcolor=lightgray, linestyle=none](.5,.5)(1.5,.5)(1.5,9.5)(.5,9.5)
\pspolygon[fillstyle=solid,fillcolor=lightgray, linestyle=none](7.5,.5)(8.5,.5)(8.5,9.5)(7.5,9.5)
\pspolygon[fillstyle=solid,fillcolor=lightgray, linestyle=none](11.5,.5)(12.5,.5)(12.5,9.5)(11.5,9.5)
\pspolygon[fillstyle=solid,fillcolor=lightgray, linestyle=none](.5,9.5)(.5,8.5)(16.5,8.5)(16.5,9.5)
\pspolygon[fillstyle=solid,fillcolor=lightgray, linestyle=none](.5,5.5)(.5,4.5)(16.5,4.5)(16.5,5.5)
\rput(1,9){$g_{(1,1)}$} \rput(2,9){$1$} \rput(3,9){$1$} \rput(4,9){$A^\leftrightarrow_{(1,1)}$}
\rput(5,9){$1$} \rput(6,9){$1$} \rput(7,9){$1$} \rput(8,9){$g_{(1,2)}$}
\rput(9,9){$1$} \rput(10,9){$A^\leftrightarrow_{(1,2)}$} \rput(11,9){$1$} \rput(12,9){$g_{(1,3)}$}
\rput(13,9){$A^\leftrightarrow_{(1,3)}$} \rput(14,9){$1$} \rput(15,9){$1$} \rput(16,9){$1$}
\rput(1,8){$1$} \rput(2,8){$1$} \rput(3,8){$1$} \rput(4,8){$1$}
\rput(5,8){$(Rdu')_{(1,1)}$} \rput(6,8){$1$} \rput(7,8){$1$} \rput(8,8){$1$}
\rput(9,8){$1$} \rput(10,8){$1$} \rput(11,8){$1$} \rput(12,8){$1$}
\rput(13,8){$1$} \rput(14,8){$1$} \rput(15,8){$1$} \rput(16,8){$1$}
\rput(1,7){$1$} \rput(2,7){$1$} \rput(3,7){$1$} \rput(4,7){$1$}
\rput(5,7){$1$} \rput(6,7){$1$} \rput(7,7){$1$} \rput(8,7){$1$}
\rput(9,7){$1$} \rput(10,7){$1$} \rput(11,7){$1$} \rput(12,7){$A^\updownarrow_{(1,3)}$}
\rput(13,7){$1$} \rput(14,7){$1$} \rput(15,7){$1$} \rput(16,7){$1$}
\rput(1,6){$1$} \rput(2,6){$1$} \rput(3,6){$1$} \rput(4,6){$1$}
\rput(5,6){$1$} \rput(6,6){$1$} \rput(7,6){$B_{(1,1)}$} \rput(8,6){$1$}
\rput(9,6){$1$} \rput(10,6){$1$} \rput(11,6){$1$} \rput(12,6){$1$}
\rput(13,6){$1$} \rput(14,6){$1$} \rput(15,6){$1$} \rput(16,6){$1$}
\rput(1,5){$g_{(2,1)}$} \rput(2,5){$A^\leftrightarrow_{(2,1)}$} \rput(3,5){$1$} \rput(4,5){$1$}
\rput(5,5){$1$} \rput(6,5){$A^\leftrightarrow_{(2,1)}$} \rput(7,5){$1$} \rput(8,5){$g_{(2,2)}$}
\rput(9,5){$A^\leftrightarrow_{(2,2)}$} \rput(10,5){$1$} \rput(11,5){$1$} \rput(12,5){$g_{(2,3)}$}
\rput(13,5){$1$} \rput(14,5){$1$} \rput(15,5){$A^\leftrightarrow_{(2,3)}$} \rput(16,5){$A^\leftrightarrow_{(2,3)}$}
\rput(1,4){$1$} \rput(2,4){$1$} \rput(3,4){$1$} \rput(4,4){$1$}
\rput(5,4){$1$} \rput(6,4){$1$} \rput(7,4){$1$} \rput(8,4){$1$}
\rput(9,4){$1$} \rput(10,4){$1$} \rput(11,4){$1$} \rput(12,4){$1$}
\rput(13,4){$1$} \rput(14,4){$(Rdu')_{(2,3)}$} \rput(15,4){$1$} \rput(16,4){$1$}
\rput(1,3){$A^\updownarrow_{(2,1)}$} \rput(2,3){$1$} \rput(3,3){$1$} \rput(4,3){$1$}
\rput(5,3){$1$} \rput(6,3){$1$} \rput(7,3){$1$} \rput(8,3){$1$}
\rput(9,3){$1$} \rput(10,3){$1$} \rput(11,3){$1$} \rput(12,3){$1$}
\rput(13,3){$1$} \rput(14,3){$1$} \rput(15,3){$1$} \rput(16,3){$1$}
\rput(1,2){$1$} \rput(2,2){$1$} \rput(3,2){$1$} \rput(4,2){$1$}
\rput(5,2){$1$} \rput(6,2){$1$} \rput(7,2){$1$} \rput(8,2){$1$}
\rput(9,2){$1$} \rput(10,2){$1$} \rput(11,2){$B_{(2,2)}$} \rput(12,2){$1$}
\rput(13,2){$1$} \rput(14,2){$1$} \rput(15,2){$1$} \rput(16,2){$1$}
\rput(1,1){$1$} \rput(2,1){$1$} \rput(3,1){$(Rdt')_{(2,1)}$} \rput(4,1){$1$}
\rput(5,1){$1$} \rput(6,1){$1$} \rput(7,1){$1$} \rput(8,1){$1$}
\rput(9,1){$1$} \rput(10,1){$1$} \rput(11,1){$1$} \rput(12,1){$1$}
\rput(13,1){$1$} \rput(14,1){$1$} \rput(15,1){$1$} \rput(16,1){$1$}
\end{pspicture}}
\]
\caption{An example of $\mathfrak a_{K,L}(M)$, with $|K|=1, |L|=2$}\label{FIG:a_(K,L)(M)}
\end{figure}

With this, we can write a formula for $It^{(p,q,\U\times \T)}(\mathfrak a_{K,L}(M))$ for an allowable matrix $M$ and choices of $K$ and $L$, at a point $(\gamma:\T\to M\times \T)\in \NN(p,q,\U\times \T)$, up to sign, as follows,
\begin{multline}\label{EQ:It(a_KL(M))}
It^{(p,q,\U\times \T)}\big(\mathfrak a_{K,L}(M)\big)\Big|_\gamma = \pm \prod_{i=1}^p \prod_{j=1}^q g_{(i,j)}\left(\gamma\left(\frac{i-1}{p},\frac{j-1}{q}\right)\right)\\
 \cdot \int_{(\Delta^{m_1}\times\dots \times \Delta^{m_p})\times(\Delta^{n_1}\times\dots \times \Delta^{n_q})} \,\,
\bigwedge_{\scriptsize\begin{matrix}\text{entries ``$2$'' in $M$,}\\ \text{given at position }(r,s) \\ \text{in the $(i,j)$'s subrectangle}\end{matrix}}
X_{(i,j)}\left(\frac{i-1+t^i_r}{p},\frac{j-1+u^j_s}{q}\right)\\ dt^1_1 \dots du^q_{n_q},
\end{multline}
where
\begin{align*}
\Delta^{m_i} &= \{ 0 \leq t^i_1 \leq \dots \leq t^i_{m_i} \leq 1 \} \\
\Delta^{n_j} & = \{ 0 \leq u^j_1 \leq \dots \leq u^j_{n_j} \leq 1 \} 
\end{align*}
and
\[
X_{(i,j)}(t,u) = \left\{
\begin{array}{ll}
\iota_{\wf} A^\leftrightarrow_{(i,j)}(u), & \text{if the ``$2$'' is on an edge row} \\
\iota_{\vf} A^\updownarrow_{(i,j)}(t), & \text{if the ``$2$'' is on an edge column} \\
\iota_{\vf}\iota_{\wf} Hdt'(t,u), & \text{if the ``$2$'' is at a position in }K \\
\iota_{\vf}\iota_{\wf} Hdu'(t,u), & \text{if the ``$2$'' is at a position in }L \\
\iota_{\vf}\iota_{\wf} B_{(i,j)}(t,u), & \text{otherwise}
\end{array} \right.
\]
with $v$ and $w$ being the two natural vector fields on $(M\times \T)^\T$ coming from the  torus action of the exponent of $(M\times \T)^\T$.
\end{proof}
\end{rmk}

For the next theorem, we denote by $\Ch(H):=\Ch^{(1,1,\{M\})}(\tilde H_{(1,1)})\in\Om^1(M^\T)$ the global 1-form, defined without subdividing the torus (\emph{i.e.} $p=q=1$), using the cover $\{M\}$ for which $\NN(1,1,\{M\})=M^\T$, and where $\tilde H_{(1,1)}\in \Oma^1_1$ is given by placing $H\in\Omega^3(M)$ at the non-degenerate $2$-simplex; \emph{c.f.} definition \ref{DEF:torus-hol_0}.

We now state our main result from this section, namely that the forms $hol_{2k,2\ell}$ satisfy the relation from figure \ref{DIAG:hol_k,l}.
\begin{thm}\label{THM:d(hol)=i(hol)=i(hol)}
For all $k, \ell \geq 0$, we have,
\begin{equation}\label{d-it-is}
 d_{DR}(hol_{2k,2\ell})=-\iota_\vf(hol_{2k+2,2\ell})=-\iota_\wf(hol_{2k,2\ell+2}).
\end{equation}
and all these expressions are equal to $i\cdot \Ch(H)\cdot hol_{2k,2\ell}$.
\end{thm}
\begin{proof}
First, we rewrite the $\Ch(H)$ in terms of an iterated integral,
\[
\Ch(H)=\int_\T (\iota_\wf \iota_\vf H) dtdu = \iota_\wf\int_\T ( \iota_\vf H) dtdu=-\iota_\vf \int_\T (\iota_\wf  H) dtdu.
\] 
We will use the equation \eqref{condensedhol2k2l} from proposition \ref{PROP:hol_(2k,2l)-glue}, and since $\iota_\vf$ and $\iota_\wf$ are derivations with $(\iota_\vf )^2=(\iota_\wf)^2=0$, and $\iota_\vf (hol)=\iota_\wf (hol)=0$ by degree reasons, we obtain 
\begin{eqnarray*}
\iota_\vf (hol_{2(k+1),2\ell}) & = (k+1)\cdot \iota_\vf\left (\int_\T \iota_\wf H dtdu \right)\cdot \frac{i}{k+1} hol_{2k,2\ell} & = -i\cdot \Ch(H) hol_{2k,2\ell},\\
\iota_\vf (hol_{2k,2(\ell+1)}) & = (\ell+1)\cdot \iota_\wf\left (\int_\T \iota_\vf H dtdu \right)\cdot \frac{(-i)}{\ell+1} hol_{2k,2\ell} & = -i\cdot \Ch(H) hol_{2k,2\ell}.
\end{eqnarray*}

It remains to check that $d_{DR}(hol_{2k,2\ell})=i \cdot \Ch(H) \cdot hol_{2k,2\ell}$. We apply $d_{DR}$ to $hol_{2k,2\ell}^{(p,q,\U)}$ which is given in definition \ref{DEF:h(p,q,u)(2k,2l)} as
\begin{multline}\label{EQ-E[SPE]}
E^{(p,q,\U)}\Bigg[\Bigg(\sum_{k_{(i,j)},\ell_{(i,j)}} \prod_{r,s} \frac{1}{k_{(r,s)}!\cdot \ell_{(r,s)}!}    \left(i\widetilde{Hdt'}_{(r,s)}\right)^{\bullet k_{(r,s)}} \bullet \left(i\widetilde{Hdu'}_{(r,s)}\right)^{\bullet \ell_{(r,s)}} \Bigg)
\\
\quad\quad\quad \bullet  \prod_{r,s} \exp\left({i \tilde B_{(r,s)}}+{i \tilde A^\leftrightarrow_{(r,s)}}+{i \tilde A^\updownarrow_{(r,s)}}\right)\bullet \tilde g_{(r,s)} \Bigg].
\end{multline}
Here, $E^{(p,q,\U)}=(\rho_\NN)^*\circ \Ch^{(p,q,\U\times \T)}$ is an algebra map, and $d_{DR}$ is a derivation, and
\begin{multline*}
d_{DR}\circ (\rho_\NN)^* \circ \Ch^{(p,q,\U\times \T)}\Big(\sum_{r,s}\widetilde{Hdt'}_{(r,s)}\Big) \\
= \sum_{r,s} (\rho_\NN)^* \circ d_{DR}\int_{\Delta^1\times \Delta^1} ev^*(1\dots H dt'_{(r,s)}\dots 1)
\\
=\sum_{r,s}(\rho_\NN)^*\int_{\Delta^1\times \Delta^1} d_{DR} \circ ev^*(1\dots H dt'_{(r,s)}\dots 1)
\\
-\sum_{r,s}(\rho_\NN)^*\int_{\partial(\Delta^1\times \Delta^1)} ev^*(1\dots H dt'_{(r,s)}\dots 1)=0,  \quad\quad\quad\quad\quad
\end{multline*} 
(since $d_{DR}(Hdt')=0$, and summing the integral of the global form $Hdt'$ over $\partial(\Delta^1\times \Delta^1)$ for all $r,s$ gives the integral over $\partial \T=\varnothing$). Similarly, we obtain that $d_{DR}\left( E^{(p,q,\U)}\big(\widetilde{Hdu'}_{(r,s)}\big)\right)=0$. The last factor in \eqref{EQ-E[SPE]} is $hol^{(p,q,\U)}$, \emph{c.f.} equation \eqref{EQ-E(p,q,u)(exp)=hol}, so that the claim $d_{DR}(hol_{2k,2\ell})=i \cdot \Ch(H) \cdot hol_{2k,2\ell}$ follows from the following equation,
\begin{equation}\label{EQ-d(hol(p,q,U))=i*H*hol(p,q,U)}
d_{DR}(hol)=i\cdot \Ch(H)\cdot hol.
\end{equation}
\begin{proof}[Proof of \eqref{EQ-d(hol(p,q,U))=i*H*hol(p,q,U)}] (The proof is analogous to the proof of the general proposition \ref{D(hol)}.) We calcualte on a local set $\NN(p,q,\U)$ as,
\begin{multline*}
 hol^{(p,q,\U)}
\\
 =  It^{(q,p,\U)}\left( \exp\left(\sum_{k,\ell}{i \tilde B_{(k,\ell)}}+{i \tilde A^\leftrightarrow_{(k,\ell)}}+{i \tilde A^\updownarrow_{(k,\ell)}}\right)\bullet \prod_{k,\ell} \tilde g_{(k,\ell)}\right) \quad\quad\quad\quad \quad\quad\quad\quad \quad
\\
= \exp \Bigg(\sum_{k,\ell} i\cdot\int_{\Delta^1\times \Delta^1} ev^*(1 \dots B_{i_{(k,\ell)}}\Big|_{\ou{f}{k,\ell}}\dots 1) \quad\quad\quad\quad \quad\quad\quad\quad \quad\quad\quad\quad \quad\quad\quad\quad
\\
+ i\cdot \int_{\Delta^1} ev^*(1 \dots A_{i_{(k-1,\ell)},i_{(k,\ell)}}\Big|_{\ou{\ehor}{k,\ell}}\dots 1) 
+ i\cdot \int_{\Delta^1} ev^*(1 \dots A_{i_{(k,\ell)},i_{(k,\ell-1)}}\Big|_{\ou{\ever}{k,\ell}}\dots 1)\Bigg)
\\
\cdot \prod_{k,\ell} g_{i_{(k,\ell)},i_{(k-1,\ell)},i_{(k-1,\ell-1)}}\Big|_{\ou{v}{k,\ell}}\cdot g^{-1}_{i_{(k,\ell)},i_{(k,\ell-1)},i_{(k-1,\ell-1)}}\Big|_{\ou{v}{k,\ell}}.
\end{multline*}
Since $d_{DR}$ is a derivation, we get that,
\begin{multline*}
d_{DR}( hol^{(p,q,\U)})
\\
=hol^{(p,q,\U)} \cdot\Bigg[ d_{DR}\Bigg(\sum_{k,\ell} i\cdot\int_{\Delta^1\times\Delta^1} ev^*(1 \dots B_{i_{(k,\ell)}}\Big|_{\ou{f}{k,\ell}}\dots 1) \quad\quad\quad\quad \quad\quad\quad\quad \quad\quad
\\
 + i\cdot \int_{\Delta^1} ev^*(1 \dots A_{i_{(k-1,\ell)},i_{(k,\ell)}}\Big|_{\ou{\ehor}{k,\ell}}\dots 1) + i\cdot \int_{\Delta^1} ev^*(1 \dots A_{i_{(k,\ell)},i_{(k,\ell-1)}}\Big|_{\ou{\ever}{k,\ell}}\dots 1)\Bigg)
\\
 + \sum_{k,\ell}  g^{-1}_{i_{(k,\ell)},i_{(k-1,\ell)},i_{(k-1,\ell-1)}}\Big|_{\ou{v}{k,\ell}}\cdot g_{i_{(k,\ell)},i_{(k,\ell-1)},i_{(k-1,\ell-1)}}\Big|_{\ou{v}{k,\ell}}\quad\quad\quad
\\
 \cdot d_{DR}\bigg(g_{i_{(k,\ell)},i_{(k-1,\ell)},i_{(k-1,\ell-1)}}\Big|_{\ou{v}{k,\ell}}\cdot g^{-1}_{i_{(k,\ell)},i_{(k,\ell-1)},i_{(k-1,\ell-1)}}\Big|_{\ou{v}{k,\ell}}\bigg) \Bigg].
 \end{multline*}
We will show that the term in the square bracket is equal to $i\cdot \Ch(H)|_{\NN(p,q,\U)}$. Integration along a fiber formula together with the relations of the connection of a gerbe (definition \ref{DEF:gerbe}), give the following results,
\begin{multline*}
d_{DR} \int_{\Delta^1\times\Delta^1} ev^*(\dots B_{i_{(k,\ell)}}\dots ) \\ = (-1)^2 \int_{\Delta^1\times\Delta^1} d_{DR} (ev^*(\dots B_{i_{(k,\ell)}}\dots )) - (-1)^2\int_{\partial (\Delta^1\times\Delta^1)} ev^*(\dots B_{i_{(k,\ell)}}\dots )
\\ 
= \int_{\Delta^1\times\Delta^1} ev^*(\dots H_{i_{(k,\ell)}}\dots ) - \int_{\partial (\Delta^1\times\Delta^1)} ev^*(\dots B_{i_{(k,\ell)}}\dots ),
\end{multline*}
where $H_{i_{(k,\ell)}}:=H|_{U_{i_{(k,\ell)}}}$. Now, integrating over $\partial(\Delta^1\times\Delta^1)$ means, that on the boundary of the face $\ou{f}{k,\ell}$ we integrate $-B_{i_{(k,\ell)}}$ with the following orientations,
\[
\xymatrix@=10pt{ \ar[dd]_{\ou{\ever}{k,\ell}} &&\ar[ll]_{\ou{\ehor}{k,\ell}} \\& -B_{i_{(k,\ell)}} &\\\ar[rr]_{\ou{\ehor}{k+1,\ell}} &&\ar[uu]_{\ou{\ever\quad}{k,\ell+1}} }
\]
Next, we have
\begin{multline*}
d_{DR} \int_{\Delta^1} ev^*(\dots A_{i_1,i_2}\dots )\\ = (-1)^1 \int_{\Delta^1} d_{DR} (ev^*(\dots A_{i_1,i_2})\dots ) - (-1)^1\int_{\partial \Delta^1} ev^*(\dots A_{i_1,i_2}\dots )
\\ 
= -\int_{\Delta^1} ev^*(\dots (B_{i_2}-B_{i_1})\dots ) +  (A_{i_1,i_2}|_{(\text{endpt. $v$ of }\Delta^1)}-A_{i_1,i_2}|_{(\text{beginningpt. $v$ of }\Delta^1)}),
\end{multline*}
Thus, at the horizontal edge $\ou{\ehor}{k,\ell}$, we integrate
\[
\xymatrix@=10pt{ \ar[rrrrrrr]^{ -B_{i_{(k,\ell)}} +B_{i_{(k-1,\ell)}}} &&&&&&& }
\]
which cancel with the terms from the integral over $\partial(\Delta^1\times\Delta^1)$, and similarly at the vertical edge $\ou{\ever}{k,\ell}$, we integrate
\[
\xymatrix@=10pt{ \ar[dd]^{ -B_{i_{(k,\ell-1)}} +B_{i_{(k,\ell)}}} \\ \,  \\ \, }
\]
which also cancels for the same reasons. Note, that at the vertex $\ou{v}{k,\ell}$, we have the following evalutations,
\[
\xymatrix@=10pt{ &&&&&&& \ar[dd]^{A_{i_{(k-1,\ell)},i_{(k-1,\ell-1)}}}  &&&&&&\\ &&&&&&&&&&& \\ \ar[rrrrrrr]^{\quad\quad A_{i_{(k-1,\ell-1)},i_{(k,\ell-1)}}} &&&&&&& \ar[rrrrrr]^{-A_{i_{(k-1,\ell)},i_{(k,\ell)}}\quad\quad} \ar[dd]^{-A_{i_{(k,\ell)},i_{(k,\ell-1)}}} &&&&&& \\ &&&&&&&&&&&&& \\ &&&&&&&&&&&&& }
\]
Furthermore, these terms cancel with the application of $d_{DR}$ on the $g_{i,j,k}$ at the vertex $\ou{v}{k,\ell}$,
\begin{eqnarray*}
&& g^{-1}_{i_{(k,\ell)},i_{(k-1,\ell)},i_{(k-1,\ell-1)}} \cdot g_{i_{(k,\ell)},i_{(k,\ell-1)},i_{(k-1,\ell-1)}}
\\
&&\cdot  d_{DR}\left(g_{i_{(k,\ell)},i_{(k-1,\ell)},i_{(k-1,\ell-1)}} \cdot  g^{-1}_{i_{(k,\ell)},i_{(k,\ell-1)},i_{(k-1,\ell-1)}}\right) 
 \\
 &=& g^{-1}_{i_{(k,\ell)},i_{(k-1,\ell)},i_{(k-1,\ell-1)}} \cdot d_{DR}\left(g_{i_{(k,\ell)},i_{(k-1,\ell)},i_{(k-1,\ell-1)}} \right)
\\
&& +g_{i_{(k,\ell)},i_{(k,\ell-1)},i_{(k-1,\ell-1)}}\cdot d_{DR}\left( g^{-1}_{i_{(k,\ell)},i_{(k,\ell-1)},i_{(k-1,\ell-1)}}\right) 
\\
&=& -i\cdot \left(A_{i_{(k,\ell)},i_{(k-1,\ell)}}+A_{i_{(k-1,\ell)},i_{(k-1,\ell-1)}}+A_{i_{(k-1,\ell-1)}, i_{(k,\ell)}}\right)
\\
&& + i \cdot \left(A_{i_{(k,\ell)},i_{(k,\ell-1)}}+A_{i_{(k,\ell-1)},i_{(k-1,\ell-1)}}+A_{i_{(k-1,\ell-1)},i_{(k,\ell)}}\right)
\\
&=& i\cdot \left(-A_{i_{(k,\ell)},i_{(k-1,\ell)}}-A_{i_{(k-1,\ell)},i_{(k-1,\ell-1)}}+A_{i_{(k,\ell)},i_{(k,\ell-1)}}+A_{i_{(k,\ell-1)},i_{(k-1,\ell-1)}}\right).
\end{eqnarray*}
Therefore, the only terms that are left over in the square bracket are
\begin{eqnarray*}
\sum_{k,\ell} i\cdot \int_{\Delta^1\times\Delta^1} ev^*(1\dots H_{i_{(k,\ell)}}\dots 1)&=& i\cdot \Ch^{(p,q,\U)}\Big(\sum_{k,\ell} (1\dots H_{i_{(k,\ell)}}\dots 1)\Big)
\\
&=& i\cdot It(H)|_{\NN(p,q,\U)}.
\end{eqnarray*}
This is what we needed to show.
\end{proof}
This completes the proof of \eqref{EQ-d(hol(p,q,U))=i*H*hol(p,q,U)} and with this the proof of the theorem.
\end{proof}
Thus, we have constructed the higher holonomy forms $hol_{2k,2\ell}$, which satisfy the relations from equation \eqref{d-it-is}.
Using this equation \eqref{d-it-is}, and a choice of numbers $a,b\in \R$ with $a+b=1$, we can now define the equivariant Chern character $Ch(\mathcal G,a,b)$ for the gerbe $\mathcal G$ as
\[
 Ch(\mathcal G,a,b):=\sum_{k\geq 0, \ell\geq 0}a^{k}  \cdot b^{\ell} \cdot hol_{2k,2\ell}.
 \]
To see where the equivariant Chern character lives, recall \cite[\S 5]{AB}, and define 
\[
\Omega(M^\T)^{inv(\vf+\wf)}=\bigg\{ \omega\in \Omega(M^\T)\Big| \mathcal L_{\vf+\wf}(\omega)=0\bigg\}
\]
the space of $\vf+\wf$ invariant forms on $\Omega(M^\T)$ with the Witten differential $D_\T=d+\iota_\vf+\iota_\wf$. In fact, for every $a,b\in \R$ with $a+b=1$, we can do slightly better by considering the following subcomplex $\Omega(M^\T)^{(a,b)}$ of $\Omega(M^\T)^{inv(\vf+\wf)}$. Define $\Omega(M^\T)^{(a,b)}$ to be given by the space 
\begin{eqnarray*}
\Omega(M^\T)^{(a,b)}&:=& \Omega(M^\T)^{inv(\vf), inv(\wf), hor(-b\vf+a\wf)}\\
&=&\left\{ \omega\in \Omega(M^\T)\Big| \mathcal L_\vf(\omega)=\mathcal L_\wf(\omega)=\iota_{-b\vf+a\wf}(\omega)=0 \right\}.
\end{eqnarray*}
$\Omega(M^\T)^{(a,b)}$ has the same induced differential $D_\T=d+\iota_{\vf}+\iota_{\wf}$ satisfying $(D_\T)^2=0$. As a corollary, we have that $Ch(\mathcal G, a,b)$ is a closed element in this complex.
\begin{cor} \label{gerbechern}
For any $a+b=1$ we have 
\[\mathcal L_\vf(Ch(\mathcal G,a,b))=\mathcal L_\wf(Ch(\mathcal G,a,b))=\iota_{-b\vf+a\wf}(Ch(\mathcal G,a,b))=0, \] and furthermore,
\[D_\T(Ch(\mathcal G,a,b))=0.\]
Thus, $Ch(\mathcal G,a,b)$ is a closed element of $\Omega(M^\T)^{(a,b)}\subset \Omega(M^\T)^{inv(\vf+\wf)}$.
\end{cor}
\begin{proof}
The first two statements follow from $\mathcal L=d\iota+\iota d$ and theorem \ref{THM:d(hol)=i(hol)=i(hol)}. For the horizontal statement, theorem \ref{THM:d(hol)=i(hol)=i(hol)} gives that 
\[
-b\cdot \iota_\vf (a^{k+1} b^\ell hol_{2k+2,2\ell})+a\cdot \iota_{\wf}(a^k b^{\ell+1} hol_{2k,2\ell+2})=0.
\]
Finally, to evaluate $D_\T(Ch(\mathcal G,a,b))$, we calculate the $a^{k} b^{\ell}$ term of $D_\T(Ch(\mathcal G,a,b))$ as,
\begin{multline*}
 d\left(a^{k} b^{\ell} hol_{2k,2\ell}\right)+ \iota_\vf\left(a^{(k+1)}b^{\ell}hol_{2(k+1),2\ell}\right)  +\iota_\wf\left(a^{k}b^{(\ell+1)}hol_{2k,2(\ell+1)}\right) \\
 \stackrel {\eqref{d-it-is}} {=}(1-a-b)\cdot a^{k} b^{\ell} \cdot d(hol_{2k,2\ell})=0.
\end{multline*}
This completes the proof of the corollary.
\end{proof}

\begin{rmk}\label{REM:no-u's}
As in the previous sections, we could as well define an equivariant Chern character involving formal variables $u$ and $v$ of degree $2$, by setting
\[
 Ch^{(u,v)}(\mathcal G,a,b):=\sum_{k\geq 0, \ell\geq 0}a^{k} u^{-k} \cdot b^{\ell} v^{-\ell} \cdot hol_{2k,2\ell}.
 \]
By the same reasons given in corollary \ref{gerbechern}, we can check that $ Ch^{(u,v)}(\mathcal G,a,b)$ is a closed element of the complex $\Omega(M^\T)[u,v;u^{-1},v^{-1}]]^{inv(\vf), inv(\wf), hor(-b \cdot u \cdot \vf+a\cdot v\cdot \wf)}$ with differential $d+u\cdot \iota_\vf+v\cdot \iota_\wf$. However, the following proposition \ref{THM:Ch(G,1,0)-to-Ch(G,a,b)}, which relates the various equivariant Chern characters for different $a$ and $b$, will not hold for $Ch^{(u,v)}(\mathcal G,a,b)$. There are also similar consequences when relating the equivariant Chern characters for the gerbe and its induced line bundle on $LM$ which we will study in the next section, \emph{c.f.} remark \ref{REM:no-u's-compat} below.
\end{rmk}

We now show how the equivariant Chern characters $Ch(\mathcal G,a,b)$ for different $a, b\in \R$ with $a+b=1$ relate to each other. Denote by $\phi_{a,b}$ the matrix 
\[
\phi_{a,b} =\begin{bmatrix} a & b \\ -1 & 1  \end{bmatrix}
\in SL(2,\R),
\]
with inverse $\phi^{-1}_{a,b} =\begin{bmatrix} 1 & -b \\ 1 & a  \end{bmatrix}$. 
There is an induced map $\Phi_{a,b}:M^\T\to M^\T$ by precomposition. Note, that under this map, we have the pushforwards, $(\Phi^{-1}_{a,b})_*(\vf)=\vf+\wf$, and $(\Phi^{-1}_{a,b})_*(\wf)=-b\vf+a\wf$. Using the standard fact $(\Phi_{a,b})^*\circ \iota_{\mathbf v}=\iota_{(\Phi^{-1}_{a,b})_*(\mathbf v)}\circ (\Phi_{a,b})^*$ for any vector field $\mathbf v$, we see that there is an induced map $\Phi_{a,b}^\T:\Omega(M^\T)^{(1,0)}\to \Omega(M^\T)^{(a,b)}$. We claim that the map $\Phi_{a,b}^\T$ relates the equivariant Chern characters $Ch(\mathcal G,1,0)$ and $Ch(\mathcal G,a,b)$.

\begin{prop}\label{THM:Ch(G,1,0)-to-Ch(G,a,b)}
For $a, b \in \R$ with $a + b = 1$, let $\phi_{a,b}=\begin{bmatrix} a & b \\ -1 & 1 \end{bmatrix}  \in SL(2,\R)$, and denote by $\Phi_{a,b}$ the induced map as above. Then,
\[
(\Phi_{a,b})^* (Ch(\mathcal G, 1 , 0)) = Ch(\mathcal G, a, b).
\]
\end{prop}
\begin{proof}  
Using the fact that $(\Phi^{-1}_{a,b})_*(\wf)=-b\vf+a\wf$, and proposition \ref{PROP:hol_(2k,2l)-glue}, we calculate
\begin{multline*}
(\Phi_{a,b})^*(hol_{2k,0})=(\Phi_{a,b})^*\left(\frac{i^k}{k!}\left(\int_\T \iota_\wf H dtdu\right)^{\wedge k} \cdot hol_{0,0}\right)
\\
=\frac{i^k}{k!}\left(\int_\T \iota_{(\Phi^{-1}_{a,b})_* (\wf)} H dtdu\right)^{\wedge k} \cdot (\Phi_{a,b})^*(hol_{0,0})=\frac{i^k}{k!}\left(\int_\T \iota_{-b\vf+a\wf} H dtdu\right)^{\wedge k} \cdot hol_{0,0}
\\
=\frac{i^k}{k!}\left(-b\int_\T \iota_{\vf} H dtdu+a \int_\T \iota_{\wf} H dtdu\right)^{\wedge k} \cdot hol_{0,0}
\\
=\sum_{r+s=k}  \frac{(-1)^s}{r! s!} i^{r+s} \cdot b^s \left(\int_\T \iota_{\vf} H dtdu\right)^{\wedge s}\cdot a^r \left(\int_\T \iota_{\wf} H dtdu\right)^{\wedge r}\cdot hol_{0,0}
\\
=\sum_{r+s=k} a^r b^s hol_{2r,2s}.
\end{multline*}
Summing this equality over all $k\geq 0$ gives the claimed result.
\end{proof}

\subsection{Compatibility check} \label{subsec:compat}
In this subsection we first recall the well known fact that a gerbe on $M$ with connection induces a line bundle on $LM$ with connection. We then show, for a given gerbe with connection, how the equivariantly closed extensions of 2-holonomy to $M^\T$,
are related to the equivariant Chern character of $L(LM)$, given by applying the construction of the previous section to the associated line bundle with connection on $LM$. 

Let $C = \{U_1, \ldots , U_n\}$ be a covering of $M$, and let $(g_{ijk}, A_{ij}, B_i)$ be the datum of an abelian gerbe with connection $\mathcal G$ on $M$. Recall from definition \ref{DEF:gerbe}, that this means $g_{i,j,k}\in \Omega^0(U_{i,j,k}, U(1))$, $A_{i,j}\in \Omega^1(U_{i,j},\R)$, and $B_i\in \Omega^2(U_i,\R)$ are symmetric in their indices
\[ 
A_{i,j}=-A_{j,i}, \text{ and }  g_{i,j,k}=g^{-1}_{j,i,k}=g_{j,k,i}
\]
and 
\begin{eqnarray*}
&g_{j,k,l}g^{-1}_{i,k,l}g_{i,j,l}g^{-1}_{i,j,k}=1 &\text{on } U_{i,j,k,l}\\
&A_{j,k}-A_{i,k}+A_{i,j}=i \cdot g^{-1}_{i,j,k}dg_{i,j,k}& \text{on } U_{i,j,k}\\
&B_j-B_i=d A_{i,j} &\text{on } U_{i,j}
\end{eqnarray*}
As before, we denote the $3$-curvature by $H$, where $H|_{U_i}=dB_i$.
We now construct a line bundle with connection on $LM$. As before, we denote by $I^s_q$ the interval $I^s_q = \left[ \frac{s-1}  q , \frac s q \right]$ for $1 \leq s \leq q$.
\begin{defn}\label{DEF:hat-A-hat-g}
First, we define a covering $\mathcal C$ of $LM$ from the covering $C$ of $M$. Let
\begin{multline*}
\mathcal C =  \{ U(i_1, \ldots, i_q) : q\geq 1, \text{ and } U_{i_j} \in C \text{ for all } j=1,\dots,q \}, \\
 \text{where } \quad U(i_1, \ldots, i_q) = \{ \gamma \in LM : \gamma|_{I^s_q} \subset U_{i_s} \text{ for } 1 \leq s \leq q \}.
\end{multline*}
By compactness and the Lebesgue lemma, $\mathcal C$ is a covering of $LM$.

First, we define a local 1-form $\hat A_{\hat i}$ on $LM$ for a multi-index $\hat i=(i_1,\dots,i_q)$. On the open set  $U_{\hat i}= U(i_1, \ldots, i_q)$, we set $ev^s_q:LM\to M, \gamma\mapsto \gamma\left(\frac {s-1} q\right)$, and define
\[
\hat A_{\hat i} = \sum_{s=1}^q -i\cdot  \left( {\int_{I^s_q} \iota_\wf B_{i_s} du } +(ev^s_q)^* A_{i_{s-1},i_s}\right).
\]

Let $U_{\hat i} = U(i_1, \ldots, i_{q}) \in \mathcal C$ and $U_{\hat j}= U(j_1, \ldots, j_{q'}) \in \mathcal C$. Without loss of generality, we may assume that $q=q'$, (since we may take $q''=\text{lcm}(q,q')$, and write $U(i_1, \ldots, i_{q})=U(k_1, \ldots,k_{q''})$ and $U(j_1, \ldots, j_{q'})=U(\ell_1,\dots,\ell_{q''})$, where the $k_r$ are given by the $i_s$ repeated $q''/q$ times, and the $\ell_r$ by the $j_s$ repeated $q''/q'$ times). For $\gamma \in U_{\hat i} \cap U_{\hat j}$ and $1 \leq s \leq q$, we have that $\gamma|_{I^s_q} \subset U_{i_s} \cap U_{j_s}$.
Define $\hat g_{\hat i,\hat j} : U_{\hat i} \cap U_{\hat j} \to U(1)$  by
\[
\hat g_{\hat i,\hat j}(\gamma) = \prod_{s=1}^q \left( e^{-i\cdot \int_{I^s_q} \iota_\wf A_{j_s,i_s} du} \cdot g^{-1}_{i_{s},j_{s},j_{s-1}} \left( \gamma \left( \frac{s-1}{q} \right) \right) \cdot g_{i_{s},i_{s-1},j_{s-1}} \left( \gamma \left( \frac{s-1}{q} \right) \right) \right)
\]
where we have set $i_0=i_p$ and $j_0=j_p$ to unify notation.

\end{defn}
\begin{lem}\label{LEM:gerbe-to-bundle}
Using the notation from the last definition, the $\hat g_{\hat i,\hat j}$ and $\hat A_{\hat i}$ define a line bundle with connection on $LM$. That is they satisfy the relations
\begin{eqnarray*}
\hat g_{\hat i,\hat j}&=&\hat g_{\hat j,\hat i}^{-1} \quad  \text{on } U_{\hat i, \hat j},\\
\hat g_{\hat i,\hat j} \hat g_{\hat j,\hat k}&=& \hat g_{\hat i,\hat k} \quad \text{on } U_{\hat i, \hat j,\hat k},\\
\hat A_{\hat j}-\hat A_{\hat i}&=&\hat g_{\hat i,\hat j}^{-1} \cdot d(\hat g_{\hat i,\hat j})\quad \text{on } U_{\hat i, \hat j},\\
d \hat A_{\hat i}&=&d \hat A_{\hat j}\quad \text{on } U_{\hat i, \hat j}.
\end{eqnarray*}
The curvature $\hat R$, which is locally given by $\hat R|_{U_{\hat i}}=d \hat A_{\hat i}$, can be written as $$\hat R=i \cdot\int_I \iota_\wf H du. $$
\end{lem}
\begin{proof}
We will keep using the notation from definition \ref{DEF:hat-A-hat-g}. Since $A_{i_s,j_s}=-A_{j_s,i_s}$ and $g^{-1}_{i_{s},j_{s},j_{s-1}}  \cdot g_{i_{s},i_{s-1},j_{s-1}}=g_{j_{s},i_{s},i_{s-1}}  \cdot g^{-1}_{j_{s},j_{s-1},i_{s-1}}$, we have $\hat g_{\hat i,\hat j} = \hat g_{\hat j,\hat i}^{-1}$.  
If we have another open set $U_{\hat k}=U(k_1, \ldots, k_q)\in \mathcal C$ with $\gamma\in U_{\hat i}\cap U_{\hat j}\cap U_{\hat k}$, then $\hat g_{\hat i,\hat j} \hat g_{\hat j,\hat k}= \hat g_{\hat i,\hat k} $ since for all $1 \leq s \leq q$
\begin{multline*}
e^{-i \int_{I^s_q} \iota_\wf A_{j_s,i_s} dt} \cdot e^{-i \int_{I^s_q} \iota_\wf A_{k_s,j_s}du}
 =
e^{-i \int_{I^s_q} \iota_\wf (A_{j_s,i_s} + A_{k_s,j_s})du} \\
= 
e^{ \int_{I^s_q} \iota_\wf (-i\cdot A_{k_s,i_s} - d \log g_{i_s, j_s,k_s}) du}
 \\
 = g_{i_s, j_s,k_s} \left( \gamma \left( \frac{s-1}{q} \right) \right) e^{-i \int_{I^s_q} \iota_\wf A_{k_s,i_s} du } g^{-1}_{i_s, j_s,k_s}  \left(\gamma \left(\frac s q \right)  \right),
\end{multline*}
and, at $\gamma((s-1)/q)$ we obtain,
\begin{multline*}
(g^{-1}_{i_{s},j_{s},j_{s-1}}  \cdot g_{i_{s},i_{s-1},j_{s-1}}) \cdot (g^{-1}_{j_{s},k_{s},k_{s-1}}  \cdot g_{j_{s},j_{s-1},k_{s-1}}) \cdot g^{-1}_{i_{s-1},j_{s-1},k_{s-1}} \cdot g_{i_s,j_s,k_s}\\
=g^{-1}_{i_{s},k_{s},k_{s-1}}  \cdot g_{i_{s},i_{s-1},k_{s-1}}
\end{multline*}

Now, on $U_{\hat i} \cap U_{\hat j}$ as above, we have $\hat A_j - \hat A_i = d \log \hat g_{\hat i,\hat j}$, since
\begin{align*}
& \sum_{s=1}^q -i {\int_{I^s_q} \iota_\wf ( B_{j_s} - B_{i_s} ) du}
= \sum_{s=1}^q -i  {\int_{I^s_q} \iota_\wf d A_{i_s , j_s}du},
\\
& -i \sum_{s=1}^q (ev^s_q)^* A_{j_{s-1},j_s} +i \sum_{s=1}^q (ev^s_q)^* A_{i_{s-1},i_s}  =  \sum_{s=1}^q -i \left((ev^s_q)^* A_{j_{s-1},j_s} - (ev^s_q)^* A_{i_{s-1},i_s}\right),
\\
& d \log \left(\prod_{s=1}^q e^{-i \int_{I^s_q} \iota_\wf A_{j_s , i_s}du}\right)
= d \sum_{s=1}^q {-i \int_{I^s_q}  \iota_\wf A_{j_s , i_s} du},
\\
& d\log\left(\prod_{s=1}^q (ev^s_q)^* g^{-1}_{i_{s},j_{s},j_{s-1}} \right)= \sum_{s=1}^q -i \left( (ev^s_q)^* A_{j_{s-1},j_s} - (ev^s_q)^* A_{j_{s-1},i_s} + (ev^s_q)^* A_{j_{s},i_{s}}\right), 
\\
& d\log \left(\prod_{s=1}^q (ev^s_q)^* g_{i_{s},i_{s-1},j_{s-1}} \right )=\sum_{s=1}^q -i \left(- (ev^s_q)^* A_{j_{s-1},i_{s-1}}+(ev^s_q)^* A_{j_{s-1},i_{s}}-(ev^s_q)^* A_{i_{s-1},i_{s}}\right).
\end{align*}
We see that the claim follows, once we remark that, the integration along the fiber formula implies that
\[
d \Big(\int_{I^s_q}  \iota_\wf A_{j_s , i_s} du\Big)= -\int_{I^s_q}  \iota_\wf dA_{j_s , i_s} du+(ev^{s+1}_q)^* A_{j_s,i_s} - (ev^{s}_q)^* A_{j_s,i_s}.
\]

Finally, on $U_{\hat i}$, we have,
\begin{multline*}
\hat R |_{U_{\hat i}}= d \hat A_i = \sum_{s=1}^q -i\cdot \left( d{\int_{I^s_q} \iota_\wf B_{i_s} du} +(ev^s_q)^* d A_{i_{s-1},i_{s}}\right) 
\\ 
= \sum_{s=1}^q -i\cdot \left( - {\int_{I^s_q} \iota_\wf dB_{i_s} du} + (ev^{s+1}_q)^*B_{i_s}  -(ev^s_q)^*B_{i_s}  + (ev^s_q)^* \left(B_{i_{s}}-B_{i_{s-1}}\right)\right)\\
= i\cdot \int_I \iota_\wf H du,
\end{multline*}
so that the $d \hat A_{\hat i}$ also coincide on intersecting open sets.
\end{proof}
The following theorem gives a relation between the (higher) holonomies of the gerbe and the line bundle on $LM$ defined above.
\begin{thm} \label{compat} 
Let $\mathcal G = (g_{ijk}, A_{ij}, B_i, H)$ be an abelian gerbe with connection on $M$, and let $E=(\hat g_{ij}, \hat A, \hat R)$ be the induced line bundle on $LM$ as in defintion \ref{DEF:hat-A-hat-g} and lemma \ref{LEM:gerbe-to-bundle} above. Denote by $\Gamma: L(LM) \to M^\T$ the adjoint mapping $\Gamma(\gamma)(t,u) = \gamma(t)(u)$. Then, for all $k\geq 0$, we have that
\[
hol_{2k}^E = \Gamma^* ( hol_{2k,0}^{\mathcal G}),
\]
where $hol_{2k}^E$ is the bundle higher holonomy from definition \ref{DEF:E(p,U)}, and $hol_{2k,0}^{\mathcal G}$ is the gerbe higher holonomy from definition \ref{DEF:h(p,q,u)(2k,2l)}.
\end{thm}
\begin{proof}
Let $\gamma : S^1 \to LM$ be a loop in $LM$  such that for each $1 \leq r \leq p$, $\gamma |_{I^r_p} \subset U_{\hat {i}_r} = U\left(i_{(r,1)}, \ldots, i_{(r,q_r)}\right) \subset LM$.  We will calculate the $hol_{2k}$, and since these are well defined global forms, we can assume that $q_1 = \dots = q_p = q$, for otherwise we can repeat open sets as before. So we have  $\Gamma(\gamma) |_{I^r_p \times I^{s}_q} \subset U_{i_{(r,s)}}$ for all $1 \leq r \leq p$ and $1 \leq s \leq q$.

We first calculate $hol_{0}^E$. From remark \ref{abelianhol2k}, we know that $hol_0^E$ applied to $\gamma$ can be written as
\[
 hol_{0}^{E} (\gamma) =  \prod_{r=1}^p \hat g_{\hat i_{r-1},\hat i_r} \left( \gamma \left( \frac{r-1}p  \right) \right) \cdot e^{\int_{I^r_p} \iota_\vf \hat A_{\hat i_r} dt}.
\]
Now, we have
\begin{multline*}
\hat g_{\hat i_{r-1},\hat i_r}\left(\gamma \left(\frac{r-1}{p}\right)\right) =  \prod_{s=1}^q \Bigg( \exp\left({(-i)\cdot \int_{I^s_q} \iota_\wf A_{i_{(r, s)}, i_{(r-1,s)}}\left(\frac{r-1}{p},u\right)du}\right)
\\
 \cdot g^{-1}_{i_{(r-1,s)},i_{(r,s)},i_{(r,s-1)}}  \cdot g_{i_{(r-1,s)},i_{(r-1,s-1)},i_{(r,s-1)}} \left( \gamma \left( \frac{r-1}{p} , \frac{s-1}{q} \right) \right) \Bigg) 
\end{multline*}
and
\[
\hat A_{\hat i_r} = \sum_{s=1}^q -i\cdot \left( \int_{I^s_q} \iota_\wf B_{i_{(r,s)}}du +(ev^s_q)^* A_{i_{(r,s-1)},i_{(r,s)}} \right).
\]
Substituting for $\hat A_{\hat i_r}$ we have
\begin{multline*}
 \exp\left(\int_{I^r_p} \iota_\vf 
\sum_{s=1}^q (-i)\cdot \left( \int_{I^s_q} \iota_\wf B_{i_{(r,s)} }du +  (ev^s_q)^* A_{i_{(r,s-1)},i_{(r,s)}}  \right) dt\right)
 \\ =  \exp\left(
\sum_{s=1}^q (-i)\cdot \left( \int_{I^r_p\times I^s_q} \iota_\vf \iota_\wf B_{i_{(r,s)} }dtdu
+ \int_{I^r_p} \iota_\vf  A_{i_{(r,s-1)},i_{(r,s)}} \left(t,\frac{s-1}{q}\right)dt \right)\right)
\end{multline*}
Now, taking the product over $r=1,\dots,p$, and using the fact that
\begin{multline*}
g^{-1}_{i_{(r-1,s)},i_{(r,s)},i_{(r,s-1)}} \cdot g_{i_{(r-1,s)},i_{(r-1,s-1)},i_{(r,s-1)}} \\ = 
g_{i_{(r,s)},i_{(r-1,s)},i_{(r-1,s-1)}}\cdot g^{-1}_{i_{(r,s)},i_{(r,s-1)},i_{(r-1,s-1)}} ,
\end{multline*}
we get that
\begin{multline*}
hol_{0}^E(\gamma) = 
\prod_{r=1}^p \prod_{s = 1}^q \exp  \left(  i\cdot \int_{I^r_p\times I^s_q} \iota_\wf \iota_\vf B_{i_{(r,s)} }dtdu + i\cdot
\int_{I^r_p \times \{\frac{s-1}{q}\}} \iota_\vf  A_{i_{(r,s)},i_{(r,s-1)}} dt \right. 
\\
\left. + i \cdot 
\int_{\{\frac{r-1}{p}\}\times I^s_q} \iota_\wf A_{i_{(r-1, s)}, i_{(r,s)}} du
\right) \cdot g_{i_{(r,s)},i_{(r-1,s)},i_{(r-1,s-1)}}\cdot g^{-1}_{i_{(r,s)},i_{(r,s-1)},i_{(r-1,s-1)}}.
\end{multline*}
Comparing this to equation \eqref{EQ:hol-gamma-for-gerbe-expliciti} for the choices made in proposition \ref{PROP:hol-for-torus}, we see that this is equal to $hol^{\mathcal G}_{0,0}(\Gamma(\gamma))=\Gamma^*(hol^{\mathcal G}_{0,0})(\gamma)$.

Now, again from remark \ref{abelianhol2k}, we write $hol^E_{2k}$ explicitly for the line bundle with connection on $LM$ as follows,
\begin{multline*}
hol_{2k}^E  = \left( \int_{\Delta^k} \hat R(t_1) \dots \hat R(t_k)dt_1\dots dt_k \right) \wedge hol_0^E
\\
=\frac{1}{k!} \left(\int_{I} \hat R(t) dt \right)^{\wedge k} \wedge hol_0^E=\frac{i^k}{k!} \left(\int_{I\times I} \iota_\wf H(t,u) dt du \right)^{\wedge k} \wedge \Gamma^*(hol_{0,0}^{\mathcal G}),
\end{multline*}
where we have used lemma \ref{LEM:gerbe-to-bundle} in the last step. Now, from the explicit formula for the higher gerbe holonomy given in equation \eqref{condensedhol2k2l} in proposition \ref{PROP:hol_(2k,2l)-glue}, we see that the last term equals $\Gamma^*(hol^{\mathcal G}_{2k,0})$.
\end{proof}

We have seen that the equivariant Chern characters live in the spaces $Ch(E,\hat A)\in \Omega(L(LM))^{inv(\vf)}$ and $Ch(\mathcal G, 1,0)\in \Omega(M^\T)^{(1,0)}=\Omega(M^\T)^{inv(\vf), inv(\wf), hor(\wf)}$, and both spaces have a differential given by $D=d+\iota_\vf$, since $\iota_\wf=0$ on $\Omega(M^\T)^{(1,0)}$. Since $\Gamma$ maps $\Gamma :L(LM)\to M^\T$, we obtain an induced chain map
\[
\Gamma^*:\Omega(M^\T)^{inv(\vf), inv(\wf), hor(\wf)}\to \Omega(L(LM))^{inv(\vf)}
\]
 mapping the equivariant Chern characters to each other by the previous theorem. Furthermore, we have the following.
\begin{cor} \label{compatclasses}
Let $\mathcal G$ be an abelian gerbe with connection, and let $Ch(\mathcal G, 1, 0) \in \Om (M^\T)^{(1,0)}$ be the equivariantly closed extension of 2-holonomy, using $a=1$ and $b=0$ in corollary \ref{gerbechern}. Let $Ch(E,\hat A)\in \Omega(L(LM))^{inv(\vf)}$ be the equivariant Chern character on $L(LM)$ of the line bundle $E$ on $LM$ with connection $\hat A$ associated to $\mathcal G$. Then, 
\[
(\Phi^{-1}_{a,b}\circ \Gamma)^* (Ch (\mathcal G, a, b))= Ch(E; \hat A).
\]
\end{cor}
\begin{proof}
From proposition \ref{THM:Ch(G,1,0)-to-Ch(G,a,b)}, we know $(\Phi^{-1}_{a,b})^* (Ch(\mathcal G, a, b))=Ch(\mathcal G, 1, 0)$, and using theorem \ref{compat}, we have 
\[
\Gamma^* ( Ch(\mathcal G, 1, 0))= \sum_{k\geq 0}  \Gamma^*(hol^{\mathcal G}_{2k,0})= \sum_{k\geq 0}  hol^E_{2k}=Ch(E; \hat A).
\]
This completes the proof of the corollary.
\end{proof}
\begin{rmk}\label{REM:no-u's-compat}
We remark again that corollary \ref{compatclasses} does not hold for $Ch^{(u,v)} (\mathcal G, a, b)$. What is still true by theorem \ref{compat} is that the chain map
\begin{multline*}
\Gamma^*\otimes (v\mapsto 1):(\Omega(M^\T)[u,v;u^{-1},v^{-1}]]^{inv(\vf), inv(\wf), hor(\wf)},d+u\cdot \iota_\vf)
\\
\longrightarrow ( \Omega(L(LM))[u,u^{-1}]]^{inv(\vf)},d+u\cdot \iota_\vf)
\end{multline*}
maps $Ch^{(u,v)}(\mathcal G, 1,0)$ to $Ch^{(u)}(E;\hat A)$. However, the equivariant Chern characters $Ch^{(u,v)}(\mathcal G, a,b)$ can not be related this way, since proposition \ref{THM:Ch(G,1,0)-to-Ch(G,a,b)} fails in this case, \emph{c.f.} remark \ref{REM:no-u's}.
\end{rmk}

\section{Holonomy along a surface of an abelian gerbe}\label{SEC:surface-hol}
In this section, we define the holonomy of a gerbe along a closed surface $\Sigma$ in terms of a iterated integral construction similar to the one in the previous section. We will show that this coincides with the usual holonomy function defined in the literature, see proposition \ref{hol=exp^int}.

\subsection{Simplicial sets and the local Hochschild complex}\label{SUBSEC:local-gerbe-Hochschild-51}
We now define the notions of simplicial sets, Hochschild complex subject to a local data, and its iterated integral map, which will be needed in the following subsection to obtain the holonomy of along any closed surface of an abelian gerbe. 
\begin{defn}\label{D:simplicial}
$\Delta$ denotes the category whose objects are the ordered sets $[k]=\{0,1,\dots,k\}$, and morphisms $f:[k]\to [\ell]$ are non-decreasing maps $f(i)\geq f(j)$ for $i>j$. A finite simplicial set $X_\bullet$ is a contravariant functor from $\Delta$ to the category of finite sets $\Sets$, or written as a formula, $X_\bullet:\Delta^{op}\to\Sets$. Denote by $X_k=X_\bullet([k])$, and call its elements $x\in X_k$ the $k$-simplicies of $X_\bullet$. The morphisms of $\Delta$ under the functor $X_\bullet$ are represented by face maps $d_i:X_{k}\to X_{k-1}$, for $i=0,\dots,k$, and by degeneracies $s_i:X_{k}\to X_{k+1}$, for $i=0,\dots,k$, satisfying the relations
\begin{eqnarray*}
d_i \circ d_j&=&d_{j-1} \circ d_i, \quad \text{ for } i<j\\
s_i\circ s_j&=&s_{j+1}\circ s_i, \quad \text{ for } i\leq j\\
d_i\circ s_j &=& \left\{
\begin{array}{ll}
s_{j-1}\circ d_i, &  \text{ for } i<j  \\ 
id,  &  \text{ for } i=j, \text{ or }i=j+1\\
s_j\circ d_{i-1}, & \text{ for } i>j+1
\end{array} \right.
\end{eqnarray*}
A $k$-simplex $x\in X_k$ is called degenerate, if it is in the image under some degeneracy $s_i$, in other words $x=s_i(y)$ for some $y\in X_{k-1}$ and some $s_i$. If $x$ is not degenerate, then $x$ is called non-degenerate.

The geomertic realization of $X_\bullet$ is denoted by $X=|X_\bullet|=(\coprod_k X_k\times \Delta^k)/\sim$, where $\Delta^k$ is the standard $k$-simplex. For more details, see \cite{GJ, Ma}.
\end{defn}

For later use, we record the following lemma which shows that every simplex has support on a unique non-degenerate simplex.
\begin{lem}\label{!-non-deg}
For each $x\in X_k$, there exists a unique non-degenerate $\bar x\in X_j$, where $j\leq k$, such that $x=s_{i_{1}}\circ\dots\circ s_{i_{k-j}}(\bar x)$.
\end{lem}
\begin{proof}
We first prove the existence. When $x$ itself is non-degenerate, then we can set $\bar x=x$. If $x$ is degenerate, then $x=s_{i_1}(x_1)$ for some $x_1\in X_{k-1}$. If $x_1$ is non-degenerate, then we may take $\bar x=x_1$, or otherwise $x_1=s_{i_2}(x_2)$ for some $x_2$. In this way we arrive at some $x=s_{1}\circ\dots\circ s_{k-p}(x_{k-p})$, which must terminate with some non-degenerate element that we denote by $\bar x=x_{k-p}$.

Now for the uniqueness, assume that $x=s_{i_{1}}\circ\dots\circ s_{i_{k-p}}(\bar x)=s_{j_{1}}\circ\dots\circ s_{j_{k-q}}(\bar y)$ with non-degenerate $\bar x$ and $\bar y$. We need to show that $\bar x = \bar y$. Without loss of generality, we may assume that $p\leq q$. The simplicial relations imply that for each degeneracy $s_{i_r}$ there is a face map $d_{i'_r}$ such that $d_{i'_r} s_{i_r}=id$. Then $\bar x=d_{i'_{k-p}}\dots d_{i'_1}s_{i_1}\dots s_{i_{k-p}}(\bar x)=d_{i'_{k-p}}\dots d_{i'_1}s_{j_{1}}\dots s_{j_{k-q}}(\bar y)$. We now use the simplicial relations to move all the $d_{i'}$'s to the right of the $s_{j}$'s, noting that either a face map $d_{k}$ appears to the right of $s_j$, or they combine to the identity. When $q>p$ we see that there are more degeneracies than face maps, so that not all $s_j$'s can be composed to an identity, and thus $\bar x$ can be written as $\bar x=s_{j'}\dots (\bar y)$, which contradicts the assumption that $\bar x$ is non-degenerate. Similarly, if $p=q$ but not all  $s_j$'s are composed to an identity, we get that $\bar x=s_{j'}\dots (\bar y)$ for some $s_{j'}$ with the same contradiction as before. Thus, the only possible case is when $p=q$, and each $d_{i'}$ composes with some $s_{j}$ to the identity, showing that $\bar x=id(\bar y)=\bar y$. This is precisely what we needed to show.
\end{proof}

We now define the Hochschild complex, shuffle product, and iterated integral subject to local data on a manifold $M$.

Let $\Sigma_\bullet$ be a 2-dimensional set, \emph{i.e.} the non-degenerate simplicies are only in $\Sigma_0$, $\Sigma_1$ and $\Sigma_2$. We use the notation $v\in \Si_0$ for a non-degenerate vertex, $e\in \Si_1$ for a non-degenerate edge, and $f\in \Si_2$ for a non-degenerate face in $\Si_\bullet$. For a vertex $v$, we will often consider inclusions $v\subset e\subset f$ of a vertex $v$ included in an edge $e$, included in a face $f$, all of which are non-degenerate. Here an inclusion $v\subset e$ means that $v=d_i(e)$ for some $i=0,1$, and similarly $e\subset f$ means that $e=d_i(f)$ for some $i=0,1,2$.

For the next three definitions, we fix the following local data, starting from an open cover $\{U_i\}_i$ of $M$. To each non-degenerate vertex $v$ of $\Si_\bullet$, we choose an open set $U_{i_{v}}$ of $M$, similarly, for each non-degenerate edge $e$, we choose an open set $U_{i_{e}}$, and for each non-degenerate face $f$, we choose an open set $U_{i_{f}}$. We denote this choice of open sets by $\U=\{U_{i_{v}}\}_v\cup \{U_{i_{e}}\}_e\cup \{U_{i_{f}}\}_f$. We will be interested in providing an iterated integral map for the open subset $\NN(\Sigma_\bullet,\U)$ of $M^\Si$, which is defined as,
$$\NN(\Sigma_\bullet,\U):=\{\sigma\in M^\Si : \sigma|_v\subset U_{i_{v}}, \sigma|_e\subset U_{i_{e}}, \sigma|_f\subset U_{i_{f}}, \forall \text{ non-deg. }v, e, f \}.$$
(Here, for $\sigma\in M^\Si$ and a simplex $s\in\Si_\ell$, $\sigma|_s$ denotes the restriction of $\sigma: \left((\coprod_k \Si_k\times \Delta^k)/\sim\right) \to M$ to $\{s\}\times \Delta^\ell$.)

For a choice of local data $\U$, we will define the Hochschild complex associated to $\U$.
\begin{defn} \label{DEF:CurlyA}
We write $\Om(U_{i_{v}}), \Om(U_{i_{e}}), \Om(U_{i_{f}})$ for the DeRham algebras on the open sets $U_{i_{v}}, U_{i_{e}}, U_{i_{f}}$, respectively. For non-degenerate simplicies $v, e$, and $f$, define
\begin{equation}\label{AfAeAv}
\left\{
\begin{array}{ll}
\Oma(f) =& \Omega(U_{i_{f}}),\\
\Oma(e) =& \bigotimes_{e\subset f} \Omega(U_{i_{e}}\cap U_{i_{f}}), \\
&\text{where the tensor product is over }\Omega(U_{i_{e}}),\\
\Oma(v) =& \bigotimes_{v\subset e\subset f} \Omega(U_{i_{v}}\cap U_{i_{e}} \cap U_{i_{f}})/ \sim,
\end{array} \right.
\end{equation}
where for the last algebra $\Oma(v)$, we take a tensor product over all non-degenerate edges $e$, and faces $f$, with $v\subset e\subset f$. The relation $\sim$ identifies tensors over common edges and faces. More precisely, for $\alpha=\bigotimes a_{(v,e,f)}\in \bigotimes_{v\subset e\subset f} \Omega(U_{i_{v}}\cap U_{i_{e}} \cap U_{i_{f}})$ and common edges $v\subset e\subset f_1$ and $v\subset e\subset f_2$, we identify,
$$ \big( b\cdot a_{(v,e,f_1)}\otimes a_{(v,e,f_2)}\otimes\dots \big) \sim \big(a_{(v,e,f_1)}\otimes b\cdot a_{(v,e,f_2)}\otimes\dots\big), \quad \forall b\in \Om(U_{i_{v}}\cap U_{i_{e}}), $$
and for common faces $v\subset e_1\subset f$ and $v\subset e_2\subset f$, we identify,
$$ \big( c\cdot a_{(v,e_1,f)}\otimes a_{(v,e_2,f)}\otimes\dots \big) \sim \big(a_{(v,e_1,f)}\otimes c\cdot a_{(v,e_2,f)}\otimes\dots\big), \quad \forall c\in \Om(U_{i_{v}}\cap U_{i_{f}}). $$

Note, that for $e\subset f$, $\Oma(e)$ is a module over $\Oma(f)$, coming from the inclusion $U_{i_{e}}\cap U_{i_{f}}\subset U_{i_{e}}$. Also, for $v\subset e$, $\Oma(v)$ is a module over $\Oma(e)$, since with $v\subset e\subset f_{1/2}$ the map 
\begin{multline*}
\Om(U_{i_{e}}\cap U_{i_{f_1}})\otimes \Om(U_{i_{e}}\cap U_{i_{f_2}}) \to \Om(U_{i_{v}}\cap U_{i_{e}}\cap U_{i_{f_1}})\otimes \Om(U_{i_{v}}\cap U_{i_{e}}\cap U_{i_{f_2}}) 
\\
\to \Om(U_{i_{v}}\cap U_{i_{e}}\cap U_{i_{f_1}})\otimes_{\Om(U_{i_{v}}\cap U_{i_{e}})} \Om(U_{i_{v}}\cap U_{i_{e}}\cap U_{i_{f_2}})
\end{multline*}
factors through $\Om(U_{i_{e}}\cap U_{i_{f_1}})\otimes_{\Om(U_{i_{e}})} \Om(U_{i_{e}}\cap U_{i_{f_2}})$.
\end{defn}
Next, we define a Hochschild-type space $CH^{(\Si_\bullet,\U)}_\bullet $ subject to the local data $\U$ and its shuffle product, as in sections \ref{SEC:local-vector-bundle} and \ref{SEC:Chern-gerbes}. These are essentially the higher Hochschild complex and shuffle product as defined in \cite[definitions 2.1.2 and 2.4.1]{GTZ}. 
\begin{defn}\label{DEF:CH-simplicial}
Let $\Si_\bullet$ be a simplicial set, and let $\U=\{U_{i_{v}}\}_v\cup \{U_{i_{e}}\}_e\cup \{U_{i_{f}}\}_f$ be a choice of local data as before. Define the Hochschild complex with respect to $\Si_\bullet$ and $\U$ to be
$$ CH^{(\Si_\bullet,\U)}_\bullet=\bigoplus_{k\geq 0}\left(\bigotimes_{x\in \Si_k} \Oma(\bar x) \right)[k], $$
where, for any (possibly degenerate) $k$-simplex $x\in \Si_k$, $\bar x$ is the unique non-degenerate simplex associated to $x$ from lemma \ref{!-non-deg}, $\Oma(\bar x)$ is the algebra defined in equation \eqref{AfAeAv} above, and $[k]$ denotes a degree shift by $k$ as in definition \ref{DEF:Torus-CH}. The tensor product over $\Sigma_k$ is defined as in section \ref{SUBSEC:torus-hol}, namely as a coequalizer over all possible linear orderings of the set $\Sigma_k$. For an element $\mathfrak a=\bigotimes_{x\in\Si_k} \alpha_x\in CH_\bullet^{(\Sigma_\bullet,\U)}$, we call $k$ the simplicial degree of $\mathfrak a$, and we write $|\mathfrak a|=\sum_x |\alpha_x|-k$ for the total degree of $\mathfrak a$.

The differential on $CH^{(\Sigma_\bullet,\U)}_\bullet$ is defined similar to the differential in definition \ref{DEF:local-Hoch}. That is, if $d_i:\Si_k\to \Si_{k-1}$ denotes an $i^{\text{th}}$ boundary of the simplicial set $\Si_\bullet$, the simplicial relations show that either $\overline{d_i(x)}=\bar x$, or $\overline{d_i(x)}=d_i(\bar x)$. 
In either case, $d_i$ induces a map $d_i^\Oma:\Oma(\bar x)\to \Oma\left(\overline{d_i(x)}\right)$, $d_i^\Oma(\alpha)=\alpha.1$ where $1\in{\Oma\left(\overline{d_i(x)}\right)}$. The map $d_i^\Oma$ is either the identity or comes from the fact that $\Oma(d_i(\bar x))$ is a module over $\Oma(\bar x)$, as described in definition \ref{DEF:CurlyA}. We therefore obtain an induced map $(d_i)_\sharp:\bigotimes_{x\in \Si_k} \Oma(\bar x) \to \bigotimes_{y\in \Si_{k-1}} \Oma\left( \,\overline{y}\,\right)$,
$$ (d_i)_\sharp: \bigotimes_{x\in\Si_k} \alpha_x \mapsto \bigotimes_{y\in\Si_{k-1}} 1_{\Oma(\bar y)}\cdot \left(\prod_{x\in\Si_k, d_i(x)=y} d_i^\Oma (\alpha_x) \right), $$
where the expression in the parenthesis is a product in $\Oma(\bar y)$. With this notation, the differential $D$ applied to an element $\mathfrak a=\bigotimes_{x\in\Si_k} \alpha_x\in CH_\bullet^{(\Sigma_\bullet,\U)}$ of simplicial degree $k$ is given by a sum of the DeRham differential and the simplicial face maps $(d_i)_\sharp$, 
\begin{equation}\label{EQ:Hoch-differential}
 D\left(\bigotimes_{x\in\Si_k} \alpha_x\right):= (-1)^k d_{DR}\left( \bigotimes_{x\in\Si_k} \alpha_{x}\right) 
 +(-1)^{k+1}\sum_{i=0}^k (-1)^{i+1} (d_i)_\sharp\left( \bigotimes_{x\in\Si_k} \alpha_x \right).
\end{equation}
It is $D^2=0$, since $d_{DR}^2=0$, and $d':=\sum_{i=0}^k (-1)^{i+1} (d_i)_\sharp$ satisfies  $(d')^2=0$ (which can be seen from the simplicial relations of the maps $d_i$) and $d_{DR}\circ d'=d'\circ d_{DR}$. The homology of this complex is denoted by $HH_\bullet^{(\Sigma_\bullet,\U)}:=H_\bullet(CH_\bullet^{(\Sigma_\bullet,\U)},D)$.

For the shuffle product $\bullet$ on $CH^{(\Sigma_\bullet,\U)}_\bullet$, we consider a degeneracy $s_i:\Si_k\to \Si_{k+1}$. It is always $\overline{s_i(x)}=\bar x$, so that we get the induced maps $(s_i)_\sharp:\bigotimes_{x\in \Si_k} \Oma(\bar x) \to \bigotimes_{y\in \Si_{k+1}} \Oma\left( \,\overline{y}\,\right)$,
\begin{equation}\label{EQ:s-sharp}
 (s_i)_\sharp: \bigotimes_{x\in\Si_k} \alpha_x \mapsto \bigotimes_{y\in\Si_{k+1}} 1_{\Oma(\bar y)}\cdot \left(\prod_{x\in\Si_k, s_i(x)=y} \alpha_x \right).
 \end{equation}
If we denote by $\star:\left(\bigotimes_{x\in \Si_k} \Oma(\bar x)\right)^{\otimes 2} \to \bigotimes_{x\in \Si_k} \Oma(\bar x) $ the product
$ \left( \bigotimes_{x\in\Si_k} \alpha_x\right)\star \left( \bigotimes_{x\in\Si_k} \beta_x\right)= \bigotimes_{x\in\Si_k} (\alpha_x\cdot \beta_x), $
then we can define the shuffle product $\bullet$ as the product generated by
\begin{multline*}
\Big(\bigotimes_{x\in \Si_k} \Oma(\bar x)\Big)\otimes\Big( \bigotimes_{x\in \Si_\ell} \Oma(\bar x)\Big)\to\Big(\bigotimes_{x\in \Si_{k+\ell}} \Oma(\bar x)\Big)\otimes\Big( \bigotimes_{x\in \Si_{k+\ell}} \Oma(\bar x)\Big) \stackrel {\star}\to \bigotimes_{x\in \Si_{k+\ell}} \Oma(\bar x)\\
\mathfrak a \bullet \mathfrak b := \sum_{\sigma\in S(k,\ell)} (-1)^{|{\mathfrak a}|\cdot\ell} \cdot \sgn(\sigma)\cdot \quad\quad\quad\quad\quad\quad\quad\quad\quad\quad\quad\quad \quad\quad\quad\quad\quad \\
\cdot \left((s_{\sigma(k+\ell)})_\sharp\circ \dots\circ ( s_{\sigma(k+1)})_\sharp(\mathfrak a)\right)\star \left((s_{\sigma(k)})_\sharp\circ \dots\circ  (s_{\sigma(1)})_\sharp(\mathfrak b)\right).
\end{multline*}
Here, the sum is over all shuffles $\sigma\in S(k,\ell)$.
\end{defn}
We also have an iterated integral map $It^{(\Sigma_\bullet,\U)}:CH^{(\Sigma_\bullet,\U)}_\bullet\to \Om(\NN(\Sigma_\bullet,\U))$, where the image of this map is the De Rham forms on $\NN(\Sigma_\bullet,\U)\subset M^\Sigma$ defined via its Fr\'echet manifold structure (see \cite[I.4.1.3]{Ha}). 
Using this notion of forms, we are now ready to define the iterated integral map analogous to the definition \ref{DEF:torus-It} in section \ref{SUBSEC:torus-hol}.

\begin{defn} \label{DEF:It-simplicial}
The iterated integral map $It^{(\Sigma_\bullet,\U)}:CH^{(\Sigma_\bullet,\U)}_\bullet\to \Om(\NN(\Sigma_\bullet,\U))$ is given by pulling back and integrating along the fiber of the maps defined below.
First, there is an evaluation map $ev:  M^\Si\times \Delta^\ell\to M^{\Si_\ell}=Map(\Si_\ell,M)$, given by
\begin{multline*}
ev\left(\sigma:\bigg(\big(\coprod_k \Si_k\times \Delta^k\big)/\sim\bigg)\to M,0\leq t_1\leq\dots\leq t_\ell\leq 1\right)(x\in\Si_\ell)\\
:= \sigma\Big(\big[x,0\leq t_1\leq\dots\leq t_\ell\leq 1\big]\Big).
\end{multline*}
Now, for an element $x\in \Si_\ell$, we consider $\bar x$ from lemma \ref{!-non-deg}, which is the non-degenerate simplex that supports $x$. Denote by $\inc{x}$ the set consisting of $\bar x$ together with all non-degenerate simplicies, in which $\bar x$ is included as a face, or face of a face,
\begin{eqnarray*}
\bar x = f\in \Si_2 &\Rightarrow & \inc{x} = \{f\},\\
\bar x = e\in \Si_1 &\Rightarrow & \inc{x} = \{e\}\cup \{f: e\subset f\},\\
\bar x = v\in \Si_0 &\Rightarrow & \inc{x} = \{v\}\cup \{e: v\subset e\}\cup \{f:v\subset f\}.
\end{eqnarray*}
{\bf Claim.} We claim that when starting with an element $\sigma\in \NN(\Si_\bullet,\U)$, the evaluation map lands in the combined intersection $ev(\sigma,t_1\leq\dots\leq t_\ell)(x) \in \bigcap_{y\in \inc{x}} U_{i_{y}}$. 
\begin{proof}
To see this claim, first note, that for $x=s_i(y)$, it is $ev(\sigma,t_1\leq\dots\leq t_\ell)(s_i(y)) = \sigma\big([y, t_1\leq\dots\leq \widehat{t_i}\leq\dots\leq t_\ell]\big)$, so that when $x=s_{i_1}\circ \dots \circ s_{i_r} (\bar x)$, we have (for some $j_1,\dots,j_r$)
\[
ev(\sigma,t_1\leq\dots\leq t_\ell)(x)= \sigma\big([\bar x, t_1\leq\dots\leq \widehat{t_{j_1}}\leq\dots\leq \widehat{t_{j_r}}\leq\dots\leq t_\ell]\big) \in U_{i_{\bar x}}.
\]
Similarly, when $\bar x=d_i (y)$ then $ev(\sigma,t_1\leq\dots\leq t_k)(d_i(y)) = \sigma\big([y, t_1\leq\dots\leq t_i\leq t_i \leq\dots\leq t_k]\big)$, so that for non-degenerate $y$, this also lands in $U_{i_y}$. A similar argument works for $\bar x= d_i(d_j(z))$.
\end{proof}
\noindent
Thus, we have the map $ev:\NN(\Si_\bullet,\U)\times \Delta^\ell\to \prod_{x\in\Si_\ell} \left(\bigcap_{y\in \inc{x}} U_{i_{y}}\right)\subset M^{\Si_\ell}$.

On the other hand, we also have an induced map $\psi_x: \Oma(\bar x)\to \Om(\bigcap_{y\in \inc{x}} U_{i_{y}})$. For non-degenerate faces $\bar x=f$ this is the identity $id:\Oma(f)= \Om(U_{i_{f}})$, for non-degenerate edges $\bar x=e$ the projections $\Om(U_{i_e}\cap U_{i_f})\to \Om(U_{i_e}\cap \bigcap_{e\subset f} U_{i_f})$ for all $(e\subset f)$ factor through tensoring over $\Om(U_{i_e})$, and for vertices $\bar x=v$ the maps $\Om(U_{i_v}\cap U_{i_e}\cap U_{i_f})\to \Om(U_{i_v}\cap\bigcap_{v\subset e} U_{i_e}\cap \bigcap_{v\subset f} U_{i_f})$ for all $(v\subset e\subset f)$ factor simultaneous through $\Om(U_{i_v}\cap U_{i_e})$ and $\Om(U_{i_v}\cap U_{i_f})$.

For an element $\mathfrak a \in CH^{(\Si_\bullet,\U)}_\bullet=\bigoplus_{\ell\geq 0}\left(\bigotimes_{x\in \Si_\ell} \Oma(\bar x) \right)$ that is in the simplicial degree-$\ell$ component, the iterated integral $It^{(\Si_\bullet,\U)}(\mathfrak a)\in \Omega(\NN(\Si_\bullet,\U))$ is given by using the maps $\psi_x$ together with the diagram
\begin{equation*}
\xymatrix{
 \NN(\Si_\bullet,\U)\times \Delta^\ell \ar[d]_{\int_{\Delta^\ell}}  \ar[r]^{ev \quad \quad \quad} & \prod_{x\in\Si_\ell} \left(\bigcap_{y\in \inc{x}} U_{i_{y}}\right)\subset M^{\Si_\ell}
 \\ \NN(\Si_\bullet,\U) & }
\end{equation*}
More precisely, we can define the iterated integral on the degree-$\ell$ component $\bigotimes_{x\in \Si_\ell} \Oma(\bar x)$ as the composition
\begin{multline*}
It^{(\Si_\bullet,\U)}:\bigotimes_{x\in \Si_\ell} \Oma(\bar x) \stackrel {\otimes_x \psi_x} \longrightarrow \bigotimes_{x\in \Si_\ell} \Om \left(\bigcap_{y\in W_x} U_{i_y}\right) \longrightarrow \Om \left(\prod_{x\in\Si_\ell} \bigcap_{y\in W_x} U_{i_y}\right)
\\
\stackrel{ev^*}\longrightarrow \Om(\NN(\Sigma_\bullet,\U)\times \Delta^\ell)\stackrel{\int_{\Delta^\ell}}\longrightarrow \Om(\NN(\Sigma_\bullet,\U)).
\end{multline*}
In short, for an element $\mathfrak a\in CH^{(\Si_\bullet,\U)}_\bullet$ in simplicial degree $\ell$, we can simply write,
\begin{equation}\label{EQ:IT^(S,U)(a)}
It^{(\Si_\bullet,\U)}(\mathfrak a)=\int_{\Delta^\ell} ev^*(\otimes_x\psi_x(\mathfrak a))\quad \in \quad \Omega(\NN(\Si_\bullet,\U)).
\end{equation}
\end{defn}
One of the main properties of the iterated integral is that it respects both the differentials and the products, just as in proposition \ref{LEM:It(q,p,U)-d-sh}.
\begin{prop}\label{LEM:It-diff-prod}
$It^{(\Sigma_\bullet,\U)}$ commutes with differentials and products,
\[
It^{(\Sigma_\bullet,\U)}(D(\mathfrak a))=d_{DR}(It^{(\Sigma_\bullet,\U)}(\mathfrak a)), \text{ and } It^{(\Sigma_\bullet,\U)}(\mathfrak a\bullet \mathfrak b)=It^{(\Sigma_\bullet,\U)}(\mathfrak a)\wedge It^{(\Sigma_\bullet,\U)}(\mathfrak b).
\]
\end{prop}
\begin{proof}
The proof is similar to \cite[lemma 2.2.2, proposition 2.4.6]{GTZ} and proposition \ref{LEM:It(q,p,U)-d-sh}, (see also propositions \ref{dItpU}). We start with the differentials. We calculate $d_{DR}(It^{(\Sigma_\bullet,\U)}(\mathfrak a))$ as follows,
\begin{eqnarray*}
d_{DR}\left(It^{(\Sigma_\bullet,\U)}(\mathfrak a)\right) 
&=& d_{DR} \left(\int_{\Delta^\ell} ev^*(\otimes_x\psi_x(\mathfrak a))\right)
\\
&=&(-1)^\ell\cdot \left(\int_{\Delta^\ell} d_{DR} (ev^*(\otimes_x\psi_x(\mathfrak a)))-\int_{\partial\Delta^\ell} ev^*(\otimes_x\psi_x(\mathfrak a))\right).
\end{eqnarray*}
Since $d$ commutes with $ev^*$ and is a graded derivation under the tensor product of maps $\otimes_x\psi_x$, we see that the first term is exactly $(-1)^\ell \cdot\int_{\Delta^\ell}ev^*(d_{DR}(\otimes_x \psi_x(\mathfrak a)))=
 \int_{\Delta^\ell} ev^*(\otimes_x\psi_x \tilde d(\mathfrak a))=It^{(\Sigma_\bullet,\U)}\big(\tilde d \mathfrak a\big)$, where $\tilde d$ is the first term on the right hand side of \eqref{EQ:Hoch-differential}. For the second term that is being integrated over $\partial \Delta^\ell$, we use that the boundary
$$
\partial\Delta^\ell=\bigcup_{i=0}^\ell \partial_i \Delta^\ell=\bigcup_{i=0}^\ell \{0=t_0\leq t_1\leq \dots\leq t_i=t_{i+1}\leq \dots\leq t_\ell\leq t_{\ell+1}=1\}.
$$
Since $\Delta^{\ell-1}\to \partial_i \Delta^\ell$ is orientation preserving exactly when $i$ is odd, we obtain $\sum_{i=0}^\ell (-1)^{\ell+1+i+1}\cdot\int_{\Delta^{\ell-1}} ev^*(\otimes_x\psi_x (d_i)_\sharp(\mathfrak a))=It^{(\Sigma_\bullet,\U)}\big((-1)^{\ell+1}\sum_{i=0}^\ell (-1)^{i+1} (d_i)_\sharp \mathfrak a\big)$, where $(d_i)_\sharp$ multiplies the terms as in the second term on the right hand side of equation \eqref{EQ:Hoch-differential}. This shows that we get exactly $It^{(\Sigma_\bullet,\U)}\big(D(\mathfrak a)\big)$, proving that the iterated integral is a chain map.

Next, we consider the shuffle and wedge products. First we rewrite the wedge product of two iterated integrals,
\begin{multline*}
It^{(\Sigma_\bullet,\U)}(\mathfrak a)\wedge It^{(\Sigma_\bullet,\U)}(\mathfrak b)
= \int_{\Delta^k}ev^*(\otimes_{x\in\Si_k}\psi_x \mathfrak a)\wedge\int_{\Delta^\ell}ev^*(\otimes_{x\in\Si_\ell}\psi_x \mathfrak b)\\
\quad\quad \,\,=(-1)^{|{\mathfrak a}|\cdot \ell}\int_{\Delta^k\times \Delta^\ell}(ev,ev)^*(\otimes_{x\in\Si_k}\psi_x \mathfrak a\wedge\otimes_{x\in\Si_\ell}\psi_x \mathfrak b)\\
=(-1)^{|{\mathfrak a}|\cdot \ell}\sum_{\sigma\in S(k,\ell)} \int_{\beta^\sigma(\Delta^{k+\ell})} (ev,ev)^* (\otimes_{x\in\Si_k}\psi_x \mathfrak a\wedge \otimes_{x\in\Si_\ell}\psi_x \mathfrak b),
\end{multline*} 
where for a $(k,\ell)$-shuffle $\sigma\in S(k,\ell)$, the map $\beta^\sigma:\Delta^{k+\ell}\to \Delta^k\times \Delta^\ell, (t_1\leq\dots\leq t_{k+\ell})\mapsto (t_{\sigma(1)}\leq \dots\leq t_{\sigma(k)},t_{\sigma(k+1)}\leq \dots\leq t_{\sigma(k+\ell)})$ is used to decompose $\Delta^k\times \Delta^\ell=\bigcup_{\sigma\in S(k,\ell)}\beta^\sigma(\Delta^{k+\ell})$. Now, we have a commutative diagram,
\begin{equation*}
\xymatrix{
 \NN(\Si_\bullet,\U)\times \Delta^{k+\ell} \ar[rr]^{id\times \beta^\sigma} \ar[d]_{ev} && \NN(\Si_\bullet,\U)\times \Delta^k\times \Delta^\ell \ar[d]^{(ev,ev)}\\   
 M^{\Si_{k+\ell}} \ar[r]^{diag \quad \quad} & M^{\Si_{k+\ell}}\times  M^{\Si_{k+\ell}}\ar[r]^{\eta'_\sigma\times \eta''_\sigma} &  M^{\Si_{k}}\times  M^{\Si_{\ell}} }
\end{equation*}
Here, $diag$ is the diagonal on $M^{\Si_{k+\ell}}$, $\eta'_\sigma:M^{\Si_{k+\ell}}\to M^{\Si_{k}}$ is given by $\eta'_\sigma=M^{s_{\sigma(k+1)}}\circ \dots\circ M^{s_{\sigma(k+\ell)}}$, and similarly $\eta''_\sigma:M^{\Si_{k+\ell}}\to M^{\Si_{\ell}}$ is $\eta''_\sigma=M^{s_{\sigma(1)}}\circ \dots\circ M^{s_{\sigma(k)}}$, and we can check that the diagram commutes, since for $x\in\Si_k, y\in \Si_\ell$, we have,
\begin{multline*}
(ev,ev)\circ (id\times \beta^\sigma)(\gamma,t_1\leq\dots\leq t_{k+\ell})(x,y)
\\
=\Big(\gamma[x,t_{\sigma(1)}\leq \dots\leq t_{\sigma(k)}],\gamma[y,t_{\sigma(k+1)}\leq \dots\leq t_{\sigma(k+\ell)}]\Big)\quad\quad\quad\quad\quad\quad\quad\quad\quad\quad\quad
 \\
=\Big(\gamma[s_{\sigma(k+\ell)}\circ\dots \circ s_{\sigma(k+1)}(x),t_1\leq\dots\leq t_{k+\ell}],\gamma[s_{\sigma(k)}\circ\dots \circ s_{\sigma(1)}(y),t_1\leq\dots\leq t_{k+\ell}]\Big)
 \\
 =(\eta'_\sigma\times \eta''_\sigma) \circ \,\, diag \circ ev (\gamma,t_1\leq\dots\leq t_{k+\ell})(x,y).
\end{multline*}
(Note also, that in the above diagram, the evaluation maps really land in the corresponding subsets $\prod_{x\in\Si_k} \left(\bigcap_{y\in \inc{x}} U_{i_{y}}\right)\subset M^{\Si_k}$, etc., as described in definition \ref{DEF:It-simplicial}, which we have suppressed for better readability.) Thus, we see that $It^{(\Sigma_\bullet,\U)}(\mathfrak a)\wedge It^{(\Sigma_\bullet,\U)}(\mathfrak b)$ is equal to,
\begin{equation*}
\sum_{\sigma\in S(k,\ell)} (-1)^{|\mathfrak a|\cdot \ell}\sgn(\sigma)\cdot \int_{\Delta^{k+\ell}} ev^*\circ diag^*( (\eta'_\sigma)^* (\otimes_{x}\psi_x \mathfrak a)\wedge(\eta''_\sigma)^* (\otimes_{x}\psi_x \mathfrak b)).
\end{equation*}
Now, the $\eta'_\sigma$ and $\eta''_\sigma$ add degeneracies to the $\mathfrak a$ and $\mathfrak b$ which, together with $diag$, becomes the shuffle product, since we have the commutative diagram
\[
\xymatrix{
\bigotimes_{x\in \Si_k} \Oma(\bar x)\ar[rr]^{(s_i)_\sharp} \ar[d]_{\otimes_{x\in \Si_{k}} \psi_x}  && \bigotimes_{x\in \Si_{k+1}} \Oma(\bar x) \ar[d]^{\otimes_{x\in \Si_{k+1}} \psi_x}\\
\Om \left(\prod_{x\in\Si_k} \bigcap_{y\in W_x} U_{i_y}\right)\ar[rr]^{(M^{s_i}|_{\cdots})^*} && \Om \left(\prod_{x\in\Si_{k+1}} \bigcap_{y\in W_x} U_{i_y}\right)
}
\]
and the fact that $diag$ acts as the $\star$-product in definition \ref{DEF:CH-simplicial}, $diag^*\circ ((\otimes_x\psi_x)\otimes (\otimes_x\psi_x))=(\otimes_x\psi_x)\circ \star$. We thus see that the above equation becomes,
\begin{equation*}
It^{(\Sigma_\bullet,\U)}(\mathfrak a)\wedge It^{(\Sigma_\bullet,\U)}(\mathfrak b)
=\int_{\Delta^{k+\ell}} ev^*(\otimes_{x\in \Si_{k+\ell}}\psi_x (\mathfrak a\bullet \mathfrak b))=It^{(\Sigma_\bullet,\U)}(\mathfrak a\bullet \mathfrak b).
\end{equation*}
This completes the proof the lemma.
\end{proof}

In the case where all open sets $U_i$ are chosen to be $U_i=M$, we see that $\NN(\Sigma_\bullet,\{M\})=M^\Sigma$. If we furthermore assume that $M$ is $2$-connected, then the iterated integral map is a quasi-isomorphism $It^{(\Sigma_\bullet,\{M\})}:CH^{(\Sigma_\bullet,\{M\})}_\bullet\to \Om(\NN(\Sigma_\bullet,\{M\}))=\Omega(M^\Si)$, see \cite[Proposition 2.5.3]{GTZ}.

\subsection{Holonomy along a surface}\label{SUBSECT:holonomy}
Using the datum of a gerbe from definition \ref{DEF:gerbe}, we now define a natural element in a surface-Hochschild complex $CH^{(\Sigma_\bullet,\U)}$ associated to the local data $\U$.

In this section, we assume that $\Sigma_\bullet$ is a regular simplicial decomposition of a closed surface $\Si$. More precisely, we assume that the geometric realization $|\Si_\bullet|=\Si$, and $\Si_\bullet$ is a regular simplicial set according to the following definition.
\begin{defn}\label{DEF:regular}
Let $\Si_\bullet$ be a 2-dimensional simplicial set, \emph{i.e.} the only non-degenerate simplicies are in $\Si_i$ for $i=0,1,2$. Then $\Si_\bullet$ is called regular, if every edge $e$ is include in exactly two faces, $e\subset f'_e$ and $e\subset f''_e$, and for $v\subset f$, a vertex included in a face $f$, there are exactly two edges that are included in between $v$ and $f$, $v\subset e'_{v,f}\subset f$ and $v\subset e''_{v,f}\subset f$. More formally, this means that
\begin{eqnarray*}
\forall \text{ non-deg. } e\in \Si_1 :\quad  \text{There are exactly two non-deg. } f'_e\neq f''_e\in \Si_2 \\
 \text{ with } d_{i'}(f'_e) = e=d_{i''}(f''_e) \text{ for some } i',i''\in \{0,1,2\}.
\end{eqnarray*}
\begin{eqnarray*}
\forall \text{ non-deg. } v \in \Si_0, f\in \Si_2 :\quad \text{If } d_j(d_i(f))=v \text{ for some }  i\in \{0,1,2\}, j\in \{0,1\},\\
\text{then there exist exactly two non-deg. } e'_{v,f}\neq e''_{v,f}\in \Si_1 \text{ with}\\
 d_{i'}(f)=e'_{v,f}  \text{ and } d_{j'}(e'_{v,f})=v \text{ for some } i'\in \{0,1,2\}, j'\in \{0,1\},\\
 d_{i''}(f)=e''_{v,f}  \text{ and } d_{j''}(e''_{v,f})=v \text{ for some } i''\in \{0,1,2\}, j''\in \{0,1\}.
\end{eqnarray*}
\end{defn}
We remark that for every surface $\Si$, there always exists a regular simplicial set $\Si_\bullet$ whose geometric realization is the given surface, $|\Si_\bullet|=\Si$. For example, one such decomposition is given in \cite[equation (3.1)]{GTZ}.
For a regular simplicial surface, we obtain natural elements in the Hochschild complex which we describe next. For this, we need to use the sign induced by comparing orientations of a simplex and its faces. 
\begin{defn}\label{DEF:rho(v,e,f)}
If $v$ is a vertex in the boundary of an edge $e$ with arbitrarily chosen orientation, then we set $\rho(v,e)\in \{+1,-1\}$ to be $+1$ if $v$ is the beginning point of $e$, and $-1$ otherwise. 

Furthermore, if $f$ is a $2$-simplex, and $e$ is an edge in the boundary of the $f$, both with arbitrarily chosen orientations, then we set $\rho(e,f)\in \{+1,-1\}$ to be $+1$, if the orientation of $e$ coincides with the orientation induced by $f$, and $-1$ otherwise. Here, the induced orientation is the one described in our convention \ref{CONVENT1}(2).
(For example, in a simplicial set, if $e=d_i(f)$, then $\rho(e,f)=(-1)^{i+1}$.)

Finally, define $\rho(v,e,f):=\rho(v,e)\cdot \rho(e,f)$.
\end{defn}
\begin{defn}\label{DEF:BAg^Sigma}
First, there are natural elements in each $\Oma(f), \Oma(e)$, and $\Oma(v)$, given by $B_f:=B_{i_{f}}\in \OmA(f), A_e:=(\rho(e,f'_e)\cdot A_{{i_{e}},{i_{f'_e}}}\otimes 1+\rho(e,f'_e) \cdot 1\otimes A_{{i_{e}},{i_{f''_e}}})\in \OmA(e)$ (where $f'_e$ and $f''_e$ are the two faces adjacent to the edge $e$), and $g_v:=\otimes_{(v\subset e\subset f)} (g_{{i_{v}},{i_{e}},{i_{f}}})^{\rho(v,e,f)}\in \OmA(v)$, where the product is over all $e$ and $f$ with $v\subset e\subset f$ for fixed $v$. Furthermore, these give rise to natural elements (each of total degree $0$) in $CH^{(\Sigma_\bullet,\U)}_\bullet$, by:
\begin{eqnarray*}
& B^{\Sigma_\bullet}_f\in \bigotimes_{s\in \Si_2} \OmA(s)\subset CH^{({\Sigma_\bullet},\U)}_\bullet, & B^{\Sigma_\bullet}_f:=1\otimes\dots \otimes1\otimes B_{f}\otimes 1\otimes \dots\otimes 1, \\
& A^{\Sigma_\bullet}_e\in \bigotimes_{s\in \Si_1} \OmA(s)\subset CH^{({\Sigma_\bullet},\U)}_\bullet, &A^{\Sigma_\bullet}_e:=1\otimes\dots \otimes1\otimes A_e\otimes 1\otimes \dots\otimes 1, \\
& g^{\Sigma_\bullet}_v\in \bigotimes_{s\in \Si_0} \OmA(s)\subset CH^{({\Sigma_\bullet},\U)}_\bullet, &g^{\Sigma_\bullet}_v:=1\otimes\dots \otimes 1\otimes g_v\otimes 1\otimes \dots\otimes 1.
\end{eqnarray*}
Then, define
$$ \hloc:= \exp \left(\sum_{\text{non-deg. }f} i\cdot B^{\Sigma_\bullet}_f - \sum_{\text{non-deg. }e} i \cdot A^{\Sigma_\bullet}_e\right)\bullet \left(\prod_{\text{non-deg. }v} g^{\Sigma_\bullet}_v \right). $$

Then, let $hol^{(\Si_\bullet,\U)}\in \Omega^0(M^\Si, U(1))$ by setting $hol^{(\Si_\bullet,\U)}:=It^{(\Si_\bullet,\U)}(\h^{(\Si_\bullet,\U)})$. 
\end{defn}
We can describe the holonomy function $hol^{(\Si_\bullet,\U)}\in\Omega^0(\NN(\Sigma_\bullet,\U),U(1))$ with the usual formula for holonomy, \cf \cite[(2.14)]{GR}.
\begin{prop}\label{hol=exp^int}
Let $\sigma\in \NN(\Sigma_\bullet,\U)\subset M^\Sigma$, \emph{i.e.} $\sigma:\Sigma\to M$, so that there is an induced orientation on each face $f$ coming from $\Sigma$, and chose an arbitrary (but fixed) orientation for each edge $e$. Then we have,
\begin{multline*}
 hol^{(\Si_\bullet,\U)}(\sigma)\\
=\exp\left(\sum_f i\cdot \int_f \sigma^*(B_{i_f})- \sum_{e\subset f}i\cdot\rho(e,f)\cdot \int_e \sigma^*(A_{i_e,i_f})\right) \cdot \prod_{v\subset e\subset f}  g^{\rho(v,e,f)}_{i_v,i_e,i_f}(\sigma(v))
 \end{multline*}
\end{prop}
\begin{proof}
By proposition \ref{LEM:It-diff-prod}, the iterated integral maps shuffle product on the Hochschild complex to the wedge product on forms of the mapping space $\Omega^\bullet(M^\Sigma)$. Thus, 
\begin{multline*}
 hol^{(\Si_\bullet,\U)}(\sigma)\\
= \Ch^{(\Si_\bullet,\U)}\left(\exp{\left(\sum_{\text{non-deg. }f} i\cdot B^{\Si_\bullet}_f - \sum_{\text{non-deg. }e} i\cdot A^{\Si_\bullet}_e\right)}\bullet \left(\prod_{\text{non-deg. }v} g^{\Si_\bullet}_v \right)\right)(\sigma)\\ 
 =\exp{\left(\sum_{\text{non-deg. }f} i\cdot\Ch^{(\Si_\bullet,\U)}(B^{\Si_\bullet}_f)(\sigma) - \sum_{\text{non-deg. }e}i\cdot \Ch^{(\Si_\bullet,\U)}(A^{\Si_\bullet}_e)(\sigma)\right)}\\
 \cdot \left(\prod_{\text{non-deg. }v} \Ch^{(\Si_\bullet,\U)}(g^{\Si_\bullet}_v)(\sigma) \right).
\end{multline*}
For the result, it now suffices to calculate the iterated integral map on $B^{\Si_\bullet}_f$, $A^{\Si_\bullet}_e$, and $g^{\Si_\bullet}_v$.
\begin{multline}\label{EQ:It(Bf^S)}
\Ch^{(\Si_\bullet,\U)}(B^{\Si_\bullet}_f)(\sigma) \\
=\int_{\Delta^2} \left(M^\Sigma\times \Delta^2\stackrel {ev} \longrightarrow M^{\Sigma_2}\right)^*(1\otimes\dots\otimes B_{i_{f}}\otimes \dots \otimes 1)(\sigma)
= \int_f \sigma^*({B_{i_f}}),
\end{multline}
\begin{multline}\label{EQ:It(Ae^S)}
\Ch^{(\Si_\bullet,\U)}(A^{\Si_\bullet}_e)(\sigma) \\
=\int_{\Delta^1} \left( M^\Sigma\times \Delta^1\stackrel {ev} \longrightarrow M^{\Sigma_1}\right)^*
 (1\otimes\dots\otimes (\rho(e,f'_e)\cdot A_{{i_e,i_{f'_e}}}+\rho(e,f''_e)\cdot A_{{i_{e}},{i_{f''_e}}})\otimes \dots \otimes 1)(\sigma)\\
= \rho(e,f'_e)\int_e \sigma^*({A_{i_e,i_{f'_e}}})+\rho(e,f''_e) \int_e \sigma^*({A_{i_e,i_{f''_e}}}),
\end{multline}
\begin{multline}\label{EQ:It(gv^S)}
\Ch^{(\Si_\bullet,\U)}(g^{\Si_\bullet}_v)(\sigma) \\
=\int_{\Delta^0} \left( M^\Sigma\times \Delta^0\stackrel {ev} \longrightarrow M^{\Sigma_0}\right)^*
(1\otimes\dots\otimes (\otimes_{(v\subset e\subset f)} g^{\rho(v,e,f)}_{i_{v},i_{e},i_{f}})\otimes \dots \otimes 1)(\sigma)\\
= \prod_{v\subset e\subset f} g^{\rho(v,e,f)}_{i_{v},i_{e},i_{f}}(\sigma(v)).
\end{multline}
This completes the proof of the proposition.
\end{proof}
\begin{cor}\label{LEM:Sigma-hol-glues}
The $hol^{(\Si_\bullet,\U)}$ glue together to a global $0$-form $hol\in \Omega^0(M^\Sigma,U(1))$ as
\[
hol|_{\NN(\Si_\bullet,\U)}=hol^{(\Si_\bullet,\U)}.
\]
\end{cor}
\begin{proof}
It is a well-known fact that the formula in proposition \ref{hol=exp^int} give a well-defined function, \emph{i.e.} transform correctly under change of the choices of open sets, see \emph{e.g.} \cite[section 2.2]{GR} or \cite[section 2.2]{RS}.
\end{proof}
We can also calculate the differential of the holonomy in general. To this end, denote by $\sum_f H_f^{\Sigma_\bullet}\in CH_1^{(\Sigma_\bullet,\U)}$ the sum of placing the $3$-curvature $H$ at all non-degenerate faces subject to the local data $(\Sigma_\bullet,\U)$, 
\[
 H^{\Sigma_\bullet}_f\in \bigotimes_{s\in \Si_2} \OmA(s)\subset CH^{({\Sigma_\bullet},\U)}_\bullet,  H^{\Sigma_\bullet}_f:=1\otimes\dots 1\otimes H \otimes 1\otimes \dots\otimes 1, 
 \]
similar to the definition of $B_f^{\Si_{\bullet}}$ in definition \ref{DEF:BAg^Sigma}. The local $1$-forms $\Ch^{(\Si_\bullet,\U)}\big(\sum_f H_f^{\Sigma_\bullet}\big)\in \Om^1(\NN(\Si_\bullet,\U))$ coincide on the pullback of their intersections, and we denote by $\Ch(H)\in \Omega^1(M^\Si)$ the global $1$-form obtained from this.
\begin{prop} \label{D(hol)}
If we write $hol\in\Om^0(M^\Si,\C)$ and $\Ch(H)\in \Omega^1(M^\Si,\C)$ by abuse of notation, then we have the identity,
\[
d_{DR} \left(hol\right)=i\cdot It(H)\cdot hol\quad\quad\in\Om^1(M^\Si,\C).
\]
\end{prop}
\begin{proof}
We check the formula for an open set $\NN(\Si_{\bullet},\U)\subset M^\Si$ of some local data $\U$, and use the calculation from equations \eqref{EQ:It(Bf^S)}-\eqref{EQ:It(gv^S)} in the proof of proposition \ref{hol=exp^int}. Since $d_{DR}$ is a derivation, we get that $d_{DR}( hol^{(\Si_\bullet,\U)})$ is equal to
\begin{eqnarray*}
&=& d_{DR}\left(It^{(\Si_\bullet,\U)}\left(\exp\left(\sum_f i B_f^{\Si_\bullet}-\sum_e i A_e^{\Si_\bullet}\right)\bullet \prod_v g_v^{\Si_\bullet}\right)\right)
\\
&=& d_{DR}\left[
\exp \left(\sum_{f} i\cdot\int_{\Delta^2} ev^*(1\otimes \dots B_{i_f}\dots \otimes 1) \right.\right.
\\
&& \quad\quad 
\left.\left. - \sum_{e\subset f} i\cdot \rho(e,f) \cdot \int_{\Delta^1} ev^*(1\otimes \dots A_{i_e,i_f}\dots\otimes 1)\right)\cdot \left(\prod_{v\subset e\subset f} g_{i_v,i_e,i_f}^{\rho(v,e,f)} \right)\right]
\\
&=&\exp\left(\sum_f i\cdot\int_{\Delta^2} ev^*(\dots B_{i_f}\dots)- \sum_{e\subset f}i\cdot\rho(e,f)\cdot \int_{\Delta^1} ev^*(\dots A_{i_e,i_f} \dots)\right) 
\\
&&
\cdot \prod_{v\subset e\subset f}  g^{\rho(v,e,f)}_{i_v,i_e,i_f}
\cdot\left[ d_{DR}\left(\sum_f i\cdot\int_{\Delta^2} ev^*(\dots B_{i_f}\dots)\right.\right.
\\
&&
\left.\left. - \sum_{e\subset f}i\cdot\rho(e,f)\cdot \int_{\Delta^1} ev^*(\dots A_{i_e,i_f} \dots)\right)+ \sum_{v\subset e\subset f}  g^{-\rho(v,e,f)}_{i_v,i_e,i_f}\cdot d_{DR}(g^{\rho(v,e,f)}_{i_v,i_e,i_f}) \right].
 \end{eqnarray*}
In the last expression, the terms in front of the square bracket is precisely $hol^{(\Si_\bullet,\U)}$, so that we are left to show that the term in the last square bracket is equal to $i\cdot \Ch^{(\Si_\bullet,\U)}\big(\sum_f H_f^{\Sigma_\bullet}\big)$.

Now, the integration along a fiber formula, and the relations of the connection of a gerbe (definition \ref{DEF:gerbe}), give the following results,
\begin{multline*}
d_{DR} \int_{\Delta^2} ev^*(\dots B_{i_f}\dots ) \\ = (-1)^2 \int_{\Delta^2} d_{DR} (ev^*(\dots B_{i_f}\dots )) - (-1)^2\int_{\partial \Delta^2} ev^*(\dots B_{i_f}\dots )\\ 
= \int_{\Delta^2} ev^*(\dots H\dots ) - \int_{\partial \Delta^2} ev^*(\dots B_{i_f}\dots ),
\end{multline*}
together with,
\begin{multline*}
d_{DR} \int_{\Delta^1} ev^*(\dots A_{i_e,i_f}\dots )\\ = (-1)^1 \int_{\Delta^1} d_{DR} (ev^*(\dots A_{i_e,i_f})\dots ) - (-1)^1\int_{\partial \Delta^1} ev^*(\dots A_{i_e,i_f}\dots )\\ 
= -\int_{\Delta^1} ev^*(\dots (B_{i_f}-B_{i_e})\dots ) +  (A_{i_e,i_f}|_{(\text{endpt. $v$ of }e)}-A_{i_e,i_f}|_{(\text{beginningpt. $v$ of }e)}),
\end{multline*}
and,
\[
g^{-\rho(v,e,f)}_{i_v,i_e,i_f}\cdot d_{DR}\big(g^{\rho(v,e,f)}_{i_v,i_e,i_f}\big)=-i\cdot \rho(v,e,f)\cdot(A_{i_v,i_e}+A_{i_e,i_f}+A_{i_f,i_v}).
\]
Now, several terms from the square bracket cancel as follows. For each boundary component $e$ of $f$, \emph{i.e.} $e=d_i(f)$, we have
\[
- \int_{\partial \Delta^2} ev^*(\dots B_{i_f}\dots )- (-\rho(e,f))\int_{\Delta^1} ev^*(\dots B_{i_f}\dots )=\pm \int_{(other)}ev^*(\dots B_{i_f}\dots ),
\]
whereas the term $- (-\rho(e,f))\int_{\Delta^1} ev^*(\dots (-B_{i_e})\dots )$ appears twice for the two different faces $f=f'_e, f''_e$ with opposite orientations $\rho(e,f'_e)=-\rho(e,f''_e)$, so that these terms also cancel. Next, at the beginning point $v'$ and endpoint $v''$ of $e$, we have $\rho(v',e,f)=\rho(e,f)$ and $\rho(v'',e,f)=-\rho(e,f)$, so that,
\begin{eqnarray*}
\text{at }v': & -(-\rho(e,f)) A_{i_e,i_f}-\rho(v',e,f) A_{i_e,i_f}=0,\\
\text{at }v'': & +(-\rho(e,f)) A_{i_e,i_f}-\rho(v'',e,f) A_{i_e,i_f}=0.
\end{eqnarray*}
Furthermore, $\rho(v,e,f)\cdot A_{i_e, i_f}$ appears twice for beginning and endpoints $v'$ and $v''$ with $\rho(v',e,f)=-\rho(v'',e,f)$, so that these terms cancel. Similarly $\rho(v,e,f)\cdot A_{i_f, i_v}$ appears for the two edges $e'_{v,f}$ and $e''_{v,f}$ (see definition \ref{DEF:regular}), and the induced signs also necessarily differ $\rho(v,e'_{v,f},f)=-\rho(v,e''_{v,f},f)$ making these terms vanish. 
The only terms that are left over in the last square bracket are
\[ 
\sum_{f} i\cdot \int_{\Delta^2} ev^*(\dots H\dots )=i\cdot \Ch^{(\Si_\bullet,\U)}\Big(\sum_f H_f^{\Si_{\bullet}}\Big)=i\cdot It(H)|_{\NN(\Si_\bullet,\U)}.
\]
This is what we needed to show.
\end{proof}
One immediate corollary of proposition \ref{D(hol)} is the analog of corollary \ref{COR:nabla-hol0}(1); compare this also with \cite[p. 239, equation (6-9)]{Br}.
\begin{cor} \label{cor:dgerbehol}
If the 3-curvature vanishes, $H=0$, then also $d_{DR} (hol)=0$.
\end{cor}
The above proposition \ref{D(hol)} shows that the holonomy is not, in general, a closed form. In the case of the torus, this is rectified by the higher holonomy terms, leading to an equivariantly closed form on the torus mapping space, as in the previous section \ref{SEC:Chern-gerbes}.

\begin{rmk}
In the previous section, we have chosen to give the proofs of the properties of the holonomy function and its higher analogs on the level of forms on the mapping space; see in particular propositions \ref{PROP:hol-for-torus}, \ref{PROP:hol_(2k,2l)-glue}, \ref{hol=exp^int}, \ref{D(hol)}, and theorem \ref{THM:d(hol)=i(hol)=i(hol)}. However, we could have equally well have given these proof on the Hochschild side, before applying the iterated integral map. We hope that these type of considerations may be useful in a future discussion concerning non-abelian gerbes.
\end{rmk}


\begin{thebibliography}{99999}
\bibitem[A]{A} M.F. Atiyah, ``Circular symmetry and stationary phase approximation,'' Proceedings of the conference in honor of L. Schwartz, Vol. 2, Ast\'erisque, 1985, p.43--59.
\bibitem[AB]{AB} M.F. Atiyah, R. Bott, ``The moment map and equivariant cohomology,'' Topology, Vol. 23, No.1, p.1--28, 1984.
\bibitem[BS]{BS} J.C. Baez, U. Schreiber, ``Higher gauge theory,''  Categories in algebra, geometry and mathematical physics,  7--30, Contemp. Math., 431, Amer. Math. Soc., Providence, RI, 2007.
\bibitem[B]{B} J.-M. Bismut, ``Index theorem and equivariant cohomology on the loop space,'' Commun. Math. Phys. 98 (1985), p. 213--237.
\bibitem[BM]{BM} L. Breen, W. Messing, ``Differential geometry of gerbes,''  Adv. Math.  198  (2005),  no. 2, p. 732--846.
\bibitem[Br]{Br} J.L. Brylinski, ``Loop spaces, characteristic classes and geometric quantization,'' Progress in Mathematics, 107. BirkhŠuser Boston, Inc., Boston, MA, 1993. xvi+300 pp. ISBN: 0-8176-3644-7.
\bibitem[C1]{C1}
K.-T. Chen, ``Iterated path integrals,'' Bull. AMS 83, no. 5, (1977), p. 831--879. 
\bibitem[C2]{C2}
K.-T. Chen, ``Iterated integrals of differential forms and loop space homology,'' Ann. of
Math. (2) 97 (1973), p. 217--246.
\bibitem[D]{D} F. Dumitrescu ``Superconnections and parallel transport", Pacific J. Math. 236 (2008), no. 2, p. 307--332.
\bibitem[GR]{GR} K. Gawedzki, N. Reis, ``WZW branes and gerbes,'' arXiv:hep-th/0205233v1.
\bibitem[GJP]{GJP} E. Getzler, J. D. S. Jones, S. Petrack, ``Differential forms on loop spaces and the cyclic bar complex,''  Topology  30  (1991),  no. 3, p. 339--371. 
\bibitem[GTZ]{GTZ} G.Ginot, T. Tradler, M. Zeinalian. ``A Chen model for mapping spaces and the surface product,'' preprint arXiv:0905.223.
\bibitem[GJ]{GJ} P. G. Goerss, J.F. Jardine, ``Simplicial Homotopy Theory," Birkh\"auser, 1999. 
\bibitem[H]{Ha} R. Hamilton, ``The Inverse Function Theorem of Nash and Moser,'' Bull. of AMS, Vol. 7, No. 1, 1982.
\bibitem[Ha]{H} F. Han, ``Supersymmetric QFT, Super Loop Spaces and Bismut-Chern Character," arXiv:0711.3862v3.
\bibitem[Hi]{Hi} N. Hitchin, \textit{Lectures on special Lagrangians}, arXiv:math/9907034v1.
\bibitem[Ma]{Ma} J. P. May, ``Simplicial Objects in Algebraic Topology," D. Van Norstrand 1967; reprinted by the University of Chicago Press 1982 and 1992.
\bibitem[Mu]{M} M.K. Murray, ``Bundle gerbes,''  J. London Math. Soc. (2)  54  (1996),  no. 2, p. 403--416.
\bibitem[PM]{PM} R. Picken, M. Mackaay, \textit{Holonomy and parallel transport for Abelian gerbes}, preprint arXiv:math/0007053v3.
\bibitem[P]{P} T. Pirashvili, \textit{Hodge Decomposition for higher order Hochschild Homology},  Ann. Sci. \'Ecole Norm. Sup. (4)  33  (2000),  no. 2, p. 151--179.
\bibitem[RS]{RS} I. Runkel, R. Suszek, \textit{Gerbe-holonomy for surfaces with defect networks}, preprint arXiv:0808.1419v1.
\bibitem[SSW]{SSW} U. Schreiber, C. Schweigert, K. Waldorf, ``Unoriented WZW models and holonomy of bundle gerbes,''  Comm. Math. Phys.  274  (2007),  no. 1, p. 31--64.
\bibitem[SW]{SW} C. Schweigert, K. Waldorf. Gerbes and Lie Groups. arxiv:0720.5467.
\bibitem[ST]{ST} S. Stolz and P. Teichner, ``Super symmetric Euclidean field theories and generalized cohomology,'' preprint.
\bibitem[TTW]{TTW} J. Terilla, T. Tradler, S. O. Wilson, ``Homotopy DG algebras induce homotopy BV algebras'', preprint.
\bibitem[W]{W} E. Witten, ``Supersymmetry and Morse theory,'' J. Diff. Geom., 17, p. 661-692, 1982.
\bibitem[Z]{Z} Zhang, Weiping. Lectures on Chern-Weil theory and Witten deformations. (English summary). Nankai Tracts in Mathematics, 4. World Scientific Publishing Co., Inc., River Edge, NJ, 2001. xii+117 pp.
\end{thebibliography}
\end{document}